\documentclass[a4paper,11pt,reqno]{book}
\usepackage[margin=1.2in]{geometry}

\usepackage{amsmath}
\usepackage{amsfonts}
\usepackage{amsthm}
\usepackage{amssymb}
\usepackage{tikz}
\usepackage[all]{xy}
\usepackage{tikz-cd}
\usepackage{enumerate}
\usepackage{mathtools}
\usepackage{setspace}
\usepackage{tocloft}
\usepackage[colorlinks]{hyperref}

\usepackage{cancel}

\usepackage{calc} % multiply scales in tikz
\usepackage{etoolbox} %ifdefempty tikz 

\usepackage{comment}

\usepackage{float}

\definecolor{lightgray}{gray}{0.9}
\definecolor{darkgreen}{rgb}{0,0.5,0}
\definecolor{darkblue}{rgb}{0,0.1,0.5}

\hypersetup{
    colorlinks=true, % make the links colored
    linkcolor=darkblue, % color TOC links in blue
    urlcolor=blue, % color URLs in red
    linktoc=all, % 'all' will create links for everything in the TOC
    citecolor=darkgreen
    }

\usepackage[abbrev,msc-links, alphabetic, backrefs]{amsrefs}
\usepackage{amsrefs}
\BibSpec{arXiv}{%
  +{}{\PrintAuthors}{author}
  +{,}{ \textit}{title}
  +{,}{ arXiv }{eprint}
}
\BibSpec{other}{%
  +{}{\PrintAuthors}{author}
  +{,}{ \textit}{title}
}

\usepackage{calc}
\newcommand{\grau}{gray!60}
\newenvironment{grform}{\begin{tikzpicture}[intext]}{\end{tikzpicture}}
\tikzset{
norm/.style = {ultra thick, color = #1},
norm/.default = black,
intext/.style = {baseline = (current bounding box.center)},
}

\newcommand{\vLine}[5]{
  \draw[ultra thick, color = #5, rounded corners] (#1 , #2) -- (#1 , #2 * 0.7 + #4 * 0.3 ) -- ( #3 , #2 * 0.3 + #4 * 0.7 ) -- (#3 ,#4);
}
\newcommand{\vLineO}[5]{
  \draw[line width = 6pt , color = white, rounded corners] (#1 , #2) -- (#1 , #2 * 0.7 + #4 * 0.3 ) -- ( #3 , #2 * 0.3 + #4 * 0.7 ) -- (#3 ,#4);
  \draw[ultra thick, color = #5, rounded corners] (#1 , #2) -- (#1 , #2 * 0.7 + #4 * 0.3 ) -- ( #3 , #2 * 0.3 + #4 * 0.7 ) -- (#3 ,#4);
}

\newcommand{\dMult}[5]{
\draw[ultra thick, color = #5] (#1 , #2) -- (#1 , #2 + #4 * 0.2) .. controls (#1
, #2 + #4 * 0.5) .. (#1 + #3 *0.25 , #2 + #4 * 0.5) 
-- (#1 + #3 *0.5 , #2 + #4 * 0.5) -- (#1 + #3 *0.75 , #2 + #4 * 0.5) .. controls
(#1 + #3 , #2 + #4 * 0.5) .. (#1 + #3 , #2 + #4 * 0.2) -- (#1 + #3 , #2);
\draw[ultra thick, color = #5] (#1 + #3 *0.5 , #2 + #4 * 0.5) -- (#1 + #3 *0.5 ,
#2 + #4);
}

\newcommand{\dMultO}[5]{
\draw[line width = 6pt , color = white] (#1 , #2) -- (#1 , #2 + #4 * 0.2) ..
controls (#1 , #2 + #4 * 0.5) .. (#1 + #3 *0.25 , #2 + #4 * 0.5) 
-- (#1 + #3 *0.5 , #2 + #4 * 0.5) -- (#1 + #3 *0.75 , #2 + #4 * 0.5) .. controls
(#1 + #3 , #2 + #4 * 0.5) .. (#1 + #3 , #2 + #4 * 0.2) -- (#1 + #3 , #2);
\draw[ultra thick, color = #5] (#1 , #2) -- (#1 , #2 + #4 * 0.2) .. controls (#1
, #2 + #4 * 0.5) .. (#1 + #3 *0.25 , #2 + #4 * 0.5) 
-- (#1 + #3 *0.5 , #2 + #4 * 0.5) -- (#1 + #3 *0.75 , #2 + #4 * 0.5) .. controls
(#1 + #3 , #2 + #4 * 0.5) .. (#1 + #3 , #2 + #4 * 0.2) -- (#1 + #3 , #2);
\draw[ultra thick, color = #5] (#1 + #3 *0.5 , #2 + #4 * 0.5) -- (#1 + #3 *0.5 ,
#2 + #4);
}

\newcommand{\dAction}[6]{
\draw[ultra thick, color = #5] (#1 , #2) -- (#1 , #2 + #4 * 0.2) .. controls (#1
, #2 + #4 * 0.5) .. (#1 + #3 * 0.8 , #2 + #4 * 0.5) -- (#1 + #3 , #2 + #4 *
0.5);
\draw[ultra thick, color = #6] (#1 + #3 , #2) -- (#1 + #3 , #2 + #4);
}

\newtheorem{theorem}{Theorem}[section]

\newtheorem{conjecture}[theorem]{Conjecture}

\newtheorem{definition}[theorem]{Definition}
\newtheorem{example}[theorem]{Example}
\newtheorem{lemma}[theorem]{Lemma}
\newtheorem*{problem*}{Problem}
\newtheorem{question}[theorem]{Question}
\newtheorem{remark}[theorem]{Remark}
\newtheorem{exercise}[theorem]{Exercise}
\newtheorem{problem}{Problem}
\newtheorem{assumption}{Assumption}

\usepackage[abbrev,msc-links, alphabetic, backrefs]{amsrefs}
\usepackage{amsrefs}
\BibSpec{arXiv}{%
  +{}{\PrintAuthors}{author}
  +{,}{ \textit}{title}
  +{,}{ arXiv }{eprint}
}
\BibSpec{other}{%
  +{}{\PrintAuthors}{author}
  +{,}{ \textit}{title}
}

\newenvironment{slogan}{\begin{quotation} \color{darkgreen}}{\end{quotation}}

\newtheorem*{references}{References}

\newcommand{\id}{{\rm id}}

\newcommand{\ad}{{\rm ad}}
\newcommand{\Rep}{{\rm Rep}}
\newcommand{\CoRep}{{\rm CoRep}}

\renewcommand{\dim}{{\rm dim}}
\newcommand{\Hom}{{\rm Hom}}

\newcommand{\End}{{\rm End}}
\newcommand{\Vect}{{\rm Vect}}

\newcommand{\cB}{\mathcal{B}}
\newcommand{\cC}{\mathcal{C}}
\newcommand{\cD}{\mathcal{D}}

\newcommand{\cM}{\mathcal{M}}
\newcommand{\cN}{\mathcal{N}}
\newcommand{\cL}{\mathcal{L}}
\newcommand{\cZ}{\mathcal{Z}}
\newcommand{\cV}{\mathcal{V}}
\newcommand{\cW}{\mathcal{W}}
\newcommand{\cU}{\mathcal{U}}

\newcommand{\B}{\mathbb{B}}
\newcommand{\K}{\mathbb{K}}
\newcommand{\C}{\mathbb{C}}
\newcommand{\D}{\mathbb{D}}
\newcommand{\Z}{\mathbb{Z}}
\newcommand{\N}{\mathbb{N}}
\newcommand{\R}{\mathbb{R}}
\renewcommand{\S}{\mathbb{S}}
\newcommand{\unit}{\mathbf{1}}

\newcommand{\g}{\mathfrak{g}}
\newcommand{\h}{\mathfrak{h}}
\renewcommand{\sl}{\mathfrak{sl}}

\renewcommand{\1}{1}

\newcommand{\loc}{\mathrm{loc}}
\newcommand{\gr}{\mathrm{gr}}
\newcommand{\coin}{\mathrm{coin}}
\newcommand{\forget}{\mathrm{forget}}
\newcommand{\ra}{\mathrm{ra}}

\renewcommand{\i}{\mathrm{i}}
\renewcommand{\d}{\mathrm{d}}

\newcommand{\Y}{\mathcal{Y}}
\newcommand{\zem}{\mathfrak{Z}}

\newcommand{\Nichols}{\mathfrak{N}}
\newcommand{\NicholsOf}{\mathfrak{B}}

\newcommand{\YD}[1]{{^{#1}_{#1}}\mathcal{YD}}

\DeclareFontFamily{U}{wncy}{} 
\DeclareFontShape{U}{wncy}{m}{n}{<->wncyr10}{}
\DeclareSymbolFont{mcy}{U}{wncy}{m}{n}
\DeclareMathSymbol{\sym}{\mathord}{mcy}{"58}

\newcommand{\Cent}{\mathrm{Cent}}
\renewcommand{\O}{\mathcal{O}}

\definecolor{bg}{gray}{100} 

\title{Lecture notes on Nichols algebras}
\author{Simon D. Lentner\\ University of Hamburg }
\date{Wintersemester 2025/2026}

\usepackage{tocloft}
\begin{document}

\maketitle

%\enlargethispage{2.5cm}
\tableofcontents

\newpage

~
\vspace{-1cm}

\newcommand{\SimonRem}[1]{}

\section*{Goal of this lecture}

This lecture is an introduction to Nichols algebras, with a distinct categorical perspective and a final chapter on connections to conformal field theory.

\noindent
I want to make you familiar with the five little wonders of Nichols algebras:
\vspace{-.4cm}
\begin{itemize}
\item They are attached to a completely general situation \newline (to any object in any braided tensor category)
\item Their representations produce a nonsemisimple tensor category \newline
(because they are Hopf algebras inside the braided tensor category)
\item They produce highly nontrivial algebra relations and PBW type bases
\newline
(eg. quantum Borel parts with quantum Serre relations \& root truncation relations) \newline 
solely out of structural necessity, being a Hopf algebra with a universal property. 
\item They possess  a root system in a conceptually nice generalization, \newline which are interesting in its own right 
and can that can be classified.\newline
As a special case, this also helps us understand  Lie super algebras better.
\item They appear in any Hopf algebra and more generally in any tensor category,\newline as a kind of nonsemisimple building block.
\end{itemize}

As a main reference for this lecture I use the the excellent book of Heckenberger, Schneider: Hopf algebras and root systems \cite{HS20} and I reference to the respective chapters and theorems, where the readers can find full details.  

In this lecture I however want to develop the theory from a categorical perspective, for an audience not rooted in the theory of Hopf algebras. In contrast to a comprehensive book, I am also at liberty to introduce new background piece by piece and only to the extend necessary to understand the subject at hand. My hope is that this  significantly lowers the entry barrier into the fascinating theory of Nichols algebras. For the same reason, I emphasize the nice geometric definition of generalized root systems in terms of hyperplane arrangements due to Cuntz before presenting the chronologically earlier and technically more applicable axiomatization by Cartan graphs used in \cite{HS20}.

Some important advanced material with an algebraic focus is clearly underrepresented in these notes: I include a thorough and hands-on introduction to Nichols algebras in the Drinfeld center of vector spaces graded by nonabelian groups, which is the mainly studied class of examples  beyond the diagonal case, but I only briefly summarize the progress on the classification in this case, although this has been a fascinating and productive area of research over the last 20 years and ongoing. Only towards the end I briefly sketch the influential Andruskiewitsch-Schneider program for the classification of Hopf algebras, which has been the driving motivation for studying Nichols algebra and is the typical starting point in expositions. I complement this by some own results that aim at pushing this beyond the Hopf algebra setup. In my perspective, the main categorical question that should vastly profit from the Andruskiewitsch-Schneider program could eventually be: 
\smallskip

\begin{problem}
For a given semisimple braided tensor category $\cC$, classify in terms of Nichols algebras $\NicholsOf$ in $\cC$ all nonsemisimple tensor categories and all braided tensor categories with fixed maximal semisimple part $\cC$ (in a way yet to be defined, see below).

One of the goals of this lecture is to construct and advertise the nonsemisimple modular tensor category $\YD{\NicholsOf}(\cC)$ for a Nichols algebra $\NicholsOf$ in $\cC$, and to me it is entirely open if these are already all such examples. Asked from a different perspective: Do all nonsemisimple modular tensor categories have in their Witt class a semisimple modular tensor category?
\end{problem}
\vspace{-.2cm}
I will try to formulate all constructions and statements as categorical as possible. In fact, much of the theory can be done in general categories and several established construction gain a new perceptive from doing so. To this end I also mention own technical results in this regard. %Some has also influenced the presentation in the book, more precisely reflection functors, some give new characterizations of the category of representations of the quantum group. 
Despite the broad setup, the reader should be advised that the deeper theorems developed in the book \cite{HS20} are proven for a certain class of braided monoidal categories. We make this clear by marking such statements with a reference to
\begin{assumption}\label{ass_HS}
As the standard setup in \cite{HS20} we consider the braided tensor category of finite-dimensional  Yetter-Drinfeld modules over a Hopf algebra with bijective antipode over a field.  
\end{assumption}
\vspace{-.2cm}
This contains the case of a diagonally braided vector space, which is the case we will work with mostly. Also in this case, I prefer a slightly different setup, namely the braided category of graded vector spaces over an abelian group with braiding given by a bicharacter $\chi$. 
I think a general categorical treatment would be largely possible and desirable, but it surely requires substantial research and I formulated respective problems
\begin{problem}
Develop the theory of Nichols algebras, in particular the root system theory, in a general setup of a braided monoidal category. 
\end{problem}
\vspace{-.2cm}
\enlargethispage{2cm}
We use the language of braided tensor categories, but do not require significant previous knowledge. Some  preliminaries and relevant first examples are repeated in the preliminaries. The following notions in the intersection of algebra and category theory, which are interesting in their own right, are introduced during the lecture 
\begin{itemize}
\item The obvious notion of an algebra in a  tensor category in Section \ref{sec_catAlgebra}.
 \item The notion of a Hopf algebra $H$ in a braided tensor category $\cC$ and its tensor category of representations $\Rep(H)(\cC)$ in Section \ref{sec_DefHA} 
 \item Categorical versions of 
 Radford biproduct and projection theorem in Section \ref{sec:RadfordBiproduct}-\ref{sec:RadfordProjection}.
 %$$\Rep(H)(\cC)\cong \Rep(H^{\coinv}\rtimes L)(\cC)\cong \Rep(H^{\coinv})\Rep(L)(\cC)$$
 \item Drinfeld center and relative Drinfeld center $\cZ_\cB(\cC)$ for a given tensor category $\cC$ and a (central) braided tensor subcategory $\cB$ in Section \ref{sec_DrinfeldCenters}. As a particular case, the relative center or $\Rep(H)(\cC)$ are the Yetter-Drinfeld modules $\YD{H}(\cC)$.
 
 A main point of my lecture is to explain  how this reproduces quantum group, with $H$ a Nichols algebra and $\cC$ the braided category of weight spaces.  
 \item Commutative algebras, modules, local modules and reconstruction statements in braided tensor categories in Section \ref{sec_CommutativeAlgebra}
\end{itemize}

%Cartan matrix
\newcommand{\cm}{\mathsf{c}}

\chapter{Preliminaries}

We assume that the reader is at least vaguely familiar with the following concepts that I briefly introduce in this section
\begin{itemize}
    \item The concept of a Lie algebra, for example $\sl_3$ and the classification of finite-dimensional Lie algebras by root systems.
    \item The concept of a braided monoidal category, most importantly the category of $\Gamma$-graded vector spaces $\Vect_\Gamma^{\sigma,\omega}$ with nontrivial braiding and associator.  
    \item The concept of a quantum group, such as $U_q(\sl_2)$, which is a deformation of the universal enveloping algebra of $\g$.  
\end{itemize}  

\section{Preliminaries on Lie algebras}

A Lie algebra $\g$ is a vector space and a bracket $[-,-]:\, \g\times \g\to \g$, which is bilinear, alternating $[x,y]=-[y,x]$, and fulfills the Jacobi identity 
$[x,[y,z]]=[[x,y],z]+[y,[x,z]]$, which I wrote down in a way resembling the Leibniz rule for derivations. 

The trivial example is the abelian Lie algebra with zero bracket. Main examples are algebras together with the commutator $[x,y]=xy-yx$, or subspaces of such algebras closed under the commutator, for example certain matrices or differential operators. A representation of a Lie algebra on a vector space is a linear map $\rho:\g\to \End(V)$ matching the abstract bracket in $\g$ to the commutator of linear endomorphisms in $\End(V)$. There is always the trivial representation by zero on the $1$-dimensional vector spaces. Rearranging the Jacobi identity shows that the Lie algebra itself with the bracket becomes a representation, called adjoint representation, 

A powerful method to understand representations of Lie algebras is choosing a maximal abelian Lie subalgebra (called Cartan algebra) and simultaneously diagonalizing the action of this subalgebra. Applying this to the adjoint representation moreover leads to a classification of a class of Lie algebras in terms of their distributions of simultaneous eigenvalues, called root systems, which can be axiomatized and classified. A generalization of this procedure is a main point of this lecture. We review this now in a main example

\begin{example}
The vector space of complex $2\times 2$ matrices with trace zero is closed under the commutator $[-,-]$ and defines a Lie algebra called $\sl_2$.  Let us choose as a maximal abelian subalgebra $\h$ (Cartan algebra) the diagonal matrices, spanned by  $H_1=\begin{psmallmatrix} 
1 & 0  \\
0 & -1  \\
\end{psmallmatrix}$. If we compute the adjoint action on $\g$ via $[H_1,-]$ and determine the  eigenvectors, which is a good short exercise, we find that $\g$ has the following eigenbasis with eigenvalues $2,-2,0$ 
$$E_1=\begin{pmatrix} 
0 & 1   \\
0 & 0   \\
\end{pmatrix},\qquad
F_1=\begin{pmatrix} 
0 & 0   \\
1 & 0   \\
\end{pmatrix},\qquad
H_1=\begin{pmatrix} 
1 & 0  \\
0 & -1  \\
\end{pmatrix}$$
We call the decomposition into eigenspaces with positive eigenvalue (called positive Borel part),  eigenspaces with negative eigenvalue (called negative Borel part) and  eigenspaces with zero eigenvalue (which is the Cartan part $\h$) a \emph{triangular decomposition} of $\g$.

\medskip

Similarly, the complex Lie algebra $\sl_3$ of traceless $3\times 3$ matrix has a triangular decomposition into an positive Borel part with basis 
$$E_1=\begin{pmatrix} 
0 & 1 & 0  \\
0 & 0 & 0  \\
0 & 0 & 0  \\
\end{pmatrix},\qquad 
E_2=\begin{pmatrix} 
0 & 0 & 0  \\
0 & 0 & 1  \\
0 & 0 & 0  \\
\end{pmatrix},\qquad 
[E_1,E_2]=\begin{pmatrix} 
0 & 0 & 1  \\
0 & 0 & 0  \\
0 & 0 & 0  \\
\end{pmatrix} $$
and similarly a negative Borel part with basis $F_1,F_2,[F_1,F_2]$.
and a Cartan part, here again the diagonal matrices, with basis
$$H_1=\begin{pmatrix} 
1 & 0 & 0  \\
0 & -1 & 0  \\
0 & 0 & 0  \\
\end{pmatrix},\qquad 
H_2=\begin{pmatrix} 
0 & 0 & 0  \\
0 & 1 & 0  \\
0 & 0 & -1  \\
\end{pmatrix}$$
Again this basis is a simultaneous eigenbasis with respect to the commuting endomorphisms $[H_i,-]$. For definiteness, let us list all these relations for $\sl_3$ explicitly
\begin{align*}
[H_1,H_i] &= 0 & [H_2,H_i] &= 0 \\
&\\
[H_1,E_1] &= (+2)E_1 & [H_2,E_1] &= (-1)E_1 \\
[H_1,E_2] &= (-1)E_1 & [H_2,E_1] &= (+2)E_1 \\ 
[H_1,[E_1,E_2]] &= (+1)[E_1,E_2] & [H_2,[E_1,E_2]] &= (+1)[E_1,E_2] \\ 
&\\
[H_1,F_1] &= (-2)F_1 & [H_2,F_1] &= (+1)F_1 \\
[H_1,F_2] &= (+1)F_1 & [H_2,F_2] &= (-2)F_2 \\ 
[H_1,[F_1,F_2]] &= (-1)[F_1,F_2] & [H_2,[F_1,F_2]] &= (-1)[F_1,F_2] 
\end{align*}
The set of all simultaneous nonzero eigenvalues, written as linear maps $\alpha:\h\to \C$, is in our examples as follows, together with the respective negative values
\smallskip
\begin{align*}
    &\boxed{\sl_2}
    &
    &\boxed{\sl_3}\\ 
    \\[-1em]
    &\alpha_1(H_1)=2
    &
    &(\alpha_1(H_1),\alpha_1(H_2))=(2,-1)\\
    &&&(\alpha_2(H_1),\alpha_2(H_2))=(-1,2)\\
     &&&(\alpha_{12}(H_1),\alpha_{12}(H_2))=(1,1)
\end{align*}
with the abbreviation $\alpha_{12}=\alpha_1+\alpha_2$. Note that it follows from the Jacobi identity that the commutator of two eigenvectors is again an eigenvector with the sum of eigenvalues.
\end{example}

A \emph{root system} consists primarily of a set $\Phi$ of vectors in $\R^n$ called roots. The root system attached to a semisimple Lie algebra consists of the simultaneous eigenvalues with respect to a maximal set of commuting elements. Our two examples $\sl_2$ and $\sl_3$ above hence have roots    
\begin{align*}
&\boxed{\sl_2}
    &
    &\boxed{\sl_3}\\ 
\\[-1em]
    &\Phi=\{\pm \alpha_1\} 
     &\hspace{2cm}  &\Phi=\{\pm \alpha_1,\pm \alpha_2,\pm \alpha_{12}\}
\end{align*}
and these root systems are called $A_1$ resp. $A_2$. A choice of basis $\alpha_1$ resp. $\alpha_1,\alpha_2$ is called a set of simple roots. We draw the set of roots, with the scalar product given by the Killing form.

\begin{center}
\hspace{-1cm}
\begin{minipage}{0.4\textwidth}
\includegraphics[scale=.6]{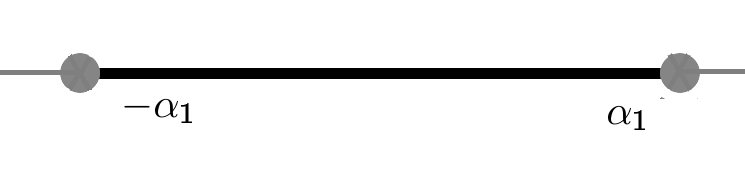}   
\end{minipage}
\hspace{1cm}
\begin{minipage}{0.4\textwidth}
\includegraphics[scale=.6]{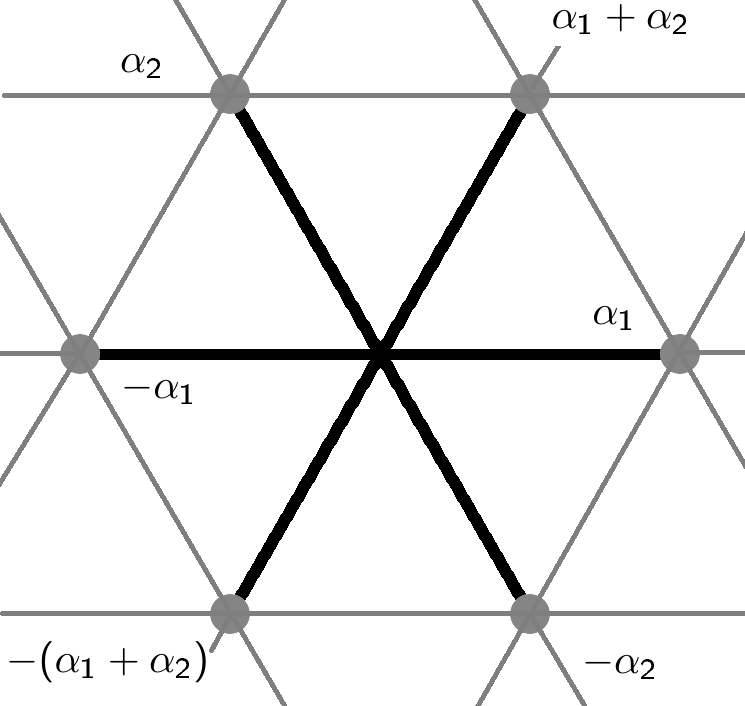}   
\end{minipage}
\hspace{2cm}
\end{center}

Geometrically, a root system defines an arrangement of hyperplanes $\alpha^\perp$ with choices of normal vectors $\alpha$ for all roots $\alpha\in \Phi$. The first main axiom of root systems states that the set of roots is closed under reflections $s_i$ on the hyperplanes $\alpha_i^\perp$, which generate a Coxeter group called \emph{Weyl group} of the root system. In our examples, these are the groups $\S_2$ resp. $\S_3$ with $s_1=(12)$ resp. $s_1=(12),s_2=(23)$. The reflection $s_i$ can be spelled out as subtraction of $\alpha_i$ a certain number of times  
$$s_i:\;\alpha_j\mapsto \alpha_j-\cm_{ij}\alpha_i,\qquad\quad 
\cm_{ij}:=\frac{2(\alpha_i,\alpha_j)}{(\alpha_i,\alpha_i)} \;\text{ (Cartan matrix)}$$
The second main axiom of root systems states that these structure constants $\cm_{ij}$ are  integers, called the crystallographic condition. In our examples the Cartan matrices are 
\smallskip
\begin{align*}
\boxed{\sl_2} && \boxed{\sl_3}
\\
\cm_{ij}&=\begin{pmatrix} 2\end{pmatrix} &\qquad \qquad \cm_{ij}&=\begin{pmatrix} 2 & -1 \\ -1 & 2\end{pmatrix}
\end{align*}
At the same time, the nonegative integers $-\cm_{ij}$ are the maximal number such that $\alpha_j+(-\cm_{ij})\alpha_i\in \Phi$ (root strings); hence they determine iterated commutator relations called \emph{Serre relations}, in our second example $\sl_3$ resp. $A_2$ these are:
$$[E_1,[E_1,E_2]]=0,\qquad [E_2,[E_2,E_1]]=0$$
For rank $2$ (two positive simple roots) the symmetric hyperplane arrangements are the regular $N$-gons, and it is not hard to see that the only ones that admit a choice of normal vectors fulfilling the crystallographic condition are 
\begin{center}
\includegraphics[scale=.2]{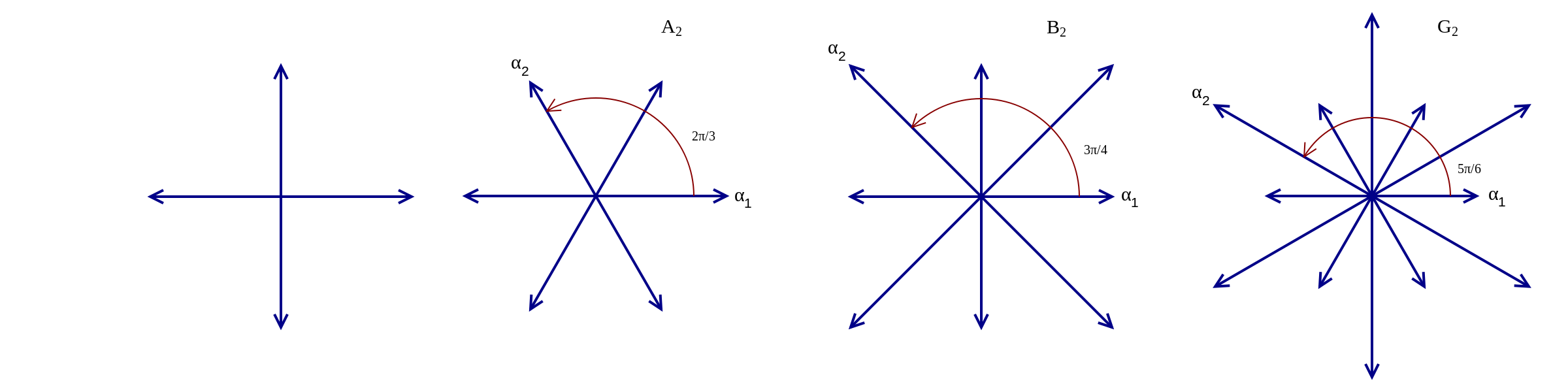}\\
$$\hspace{-2.0cm}\text{Rootsystem}\hspace{.8cm} A_1\times A_1\hspace{2cm}A_2\hspace{3cm}B_2\hspace{3cm}G_2$$
$$
(\alpha_i,\alpha_j)\quad =\quad
\begin{pmatrix} 2 & 0 \\ 0 & 2\end{pmatrix}
\hspace{1.4cm}
\begin{pmatrix} 2 & -1 \\ -1 & 2\end{pmatrix}
\hspace{1.4cm}
\begin{pmatrix} 2 & -2 \\ -2 & 4\end{pmatrix}
\hspace{1.4cm}
\begin{pmatrix} 2 & -3 \\ -3 & 6\end{pmatrix}
\hspace{1.4cm}
$$
$$
\hspace{.8cm}
\cm_{ij}\quad =\quad
\begin{pmatrix} 2 & 0 \\ 0 & 2\end{pmatrix}
\hspace{1.4cm}
\begin{pmatrix} 2 & -1 \\ -1 & 2\end{pmatrix}
\hspace{1.4cm}
\begin{pmatrix} 2 & -2 \\ -1 & 2\end{pmatrix}
\hspace{1.4cm}
\begin{pmatrix} 2 & -3 \\ -1 & 2\end{pmatrix}
\hspace{1.4cm}
$$
\end{center}
Note that the third hyperplane arrangement admits a second choice of normal vectors, then called $C_2$, where $\alpha_1$ is long and $\alpha_2$ is short, but this is equivalent by switching $\alpha_1,\alpha_2$. To classify root systems of rank $n$, we draw \emph{Dynkin diagrams}, whose nodes are the simple roots $\alpha_1,\ldots,\alpha_n$ and whose edges are according to the four cases of Cartan matrices above
\begin{center}
\hspace{1.1cm}\includegraphics[scale=1.95]{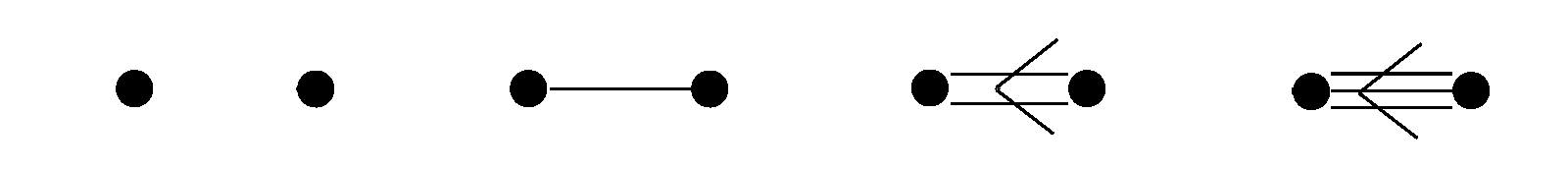}\\
\end{center}
In rank $3$ we encounter the five platonic solids of which three (Tetrahedron, Cube, Octahedron) are spacefilling and hence crystallographic.

\begin{center}
\end{center}

\begin{example}[Type $A_3$]
The root system of type $A_3$ can be constructed from the following Dynkin diagram
\begin{center}
\includegraphics[scale=.7]{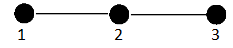}

$(\alpha_i,\alpha_j)=\cm_{ij}=
\begin{pmatrix} 2 & -1 & 0 \\ -1 & 2 & -1  \\ 0 & -1 & 2\end{pmatrix}
$
\end{center}
It has the following roots, where we abbreviate for example $\alpha_{12}=\alpha_1+\alpha_2$ 
$$\Phi=\pm\{\alpha_1,\;\alpha_2,\;\alpha_{12},\;\alpha_{23},\;\alpha_{123}\}$$
Geometrically it is an arrangement of six hyperplanes through origin in a tetrahedron way
\begin{center}
\includegraphics[scale=.3]{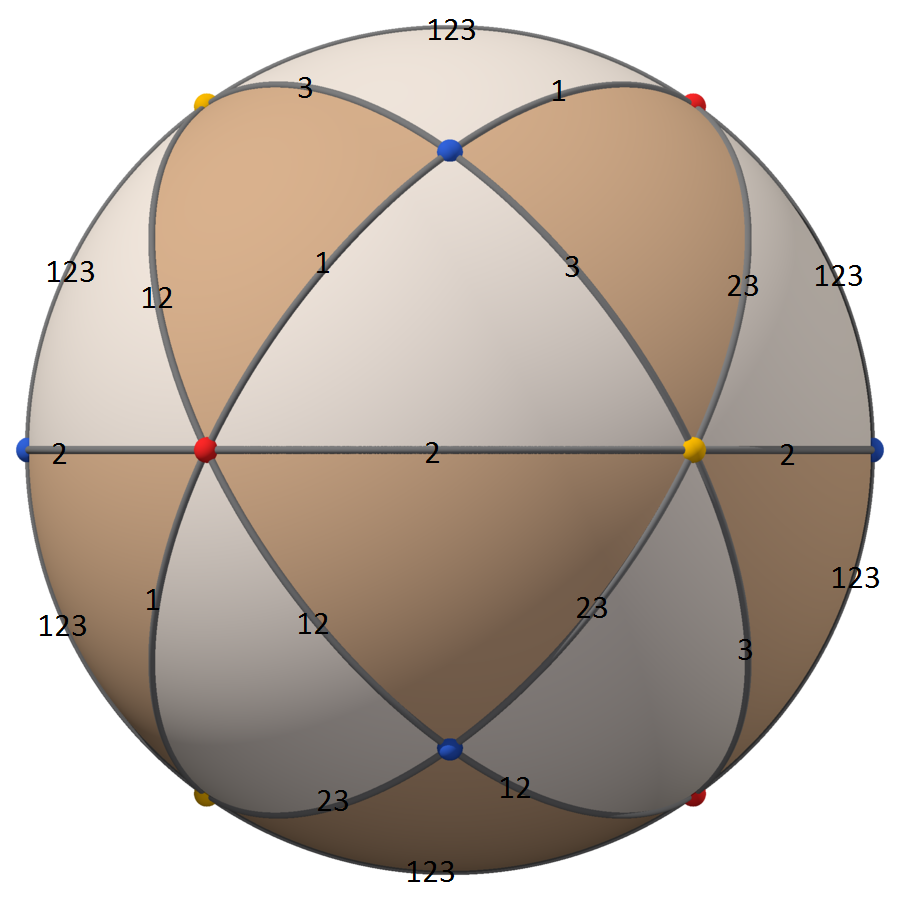}
\hspace{.4cm}
\includegraphics[scale=.65]{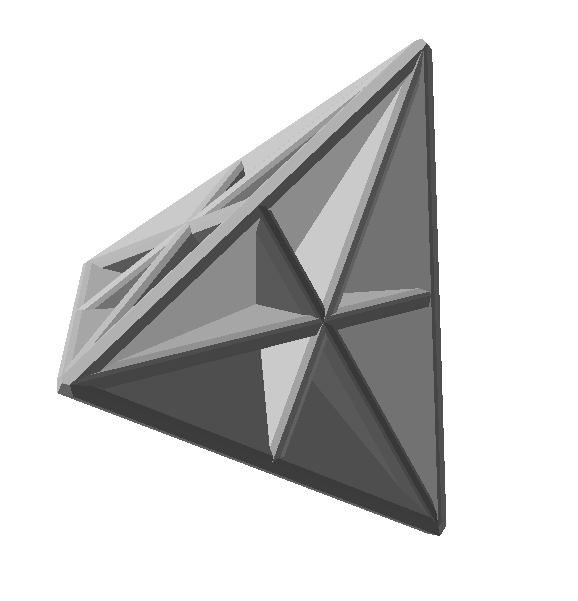}
\end{center}
The corresponding Weyl group is isomorphic to $\S_4$. More precisely, the simple reflections $s_1,s_2,s_3$ correspond say to transpositions $(12),(23),(34)$ and the set of all reflections correspond to the conjugacy class of transpositions $(ij)$ and the $24$ elements correspond to the $24$ triangular regions on the sphere (Weyl chambers). For example, the reflection
$$s_2:\;\alpha_1,\;\alpha_2,\;\alpha_3
\mapsto \alpha_{12},\;-\alpha_{2},\;\alpha_{23}$$
which is the reflection on the equatorial plane in the picture, maps the  three non-simple roots (by additivity or by computing scalar products)
$$\alpha_{12},\;\alpha_{123},\;\alpha_{23}
\mapsto
\alpha_1,\;\alpha_{123},\alpha_{3}
$$
Hence the reflection $s_2$ permutes the root and send a single positive root $\alpha_2$ to a negative root. 

This root system describes the Lie algebra $\sl_4$, see Exercise \ref{ex:sl4}. The Weyl group $\S_4$ corresponds to the permutation matrices in $\mathrm{SL}_4$.
\end{example}

Then classification of finite root systems exhibit $4$ infinite families and $5$ exceptional root systems (Killing, Cartan $\sim 1890$)
\begin{center}
\includegraphics[scale=0.8]{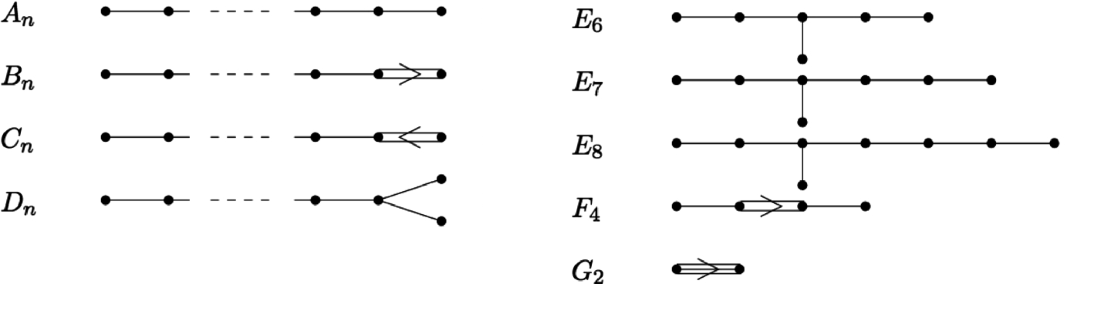}
\end{center}

\begin{exercise}\label{ex:sl4}
Compute all iterated commutators between the following matrices in $\sl_4$
$$E_1=\begin{pmatrix} 
0 & 1 & 0 & 0 \\
0 & 0 & 0 & 0 \\
0 & 0 & 0 & 0 \\
0 & 0 & 0 & 0 \\
\end{pmatrix},\qquad 
E_2=\begin{pmatrix} 
0 & 0 & 0 & 0 \\
0 & 0 & 1 & 0 \\
0 & 0 & 0 & 0 \\
0 & 0 & 0 & 0 \\
\end{pmatrix},\qquad
E_3=\begin{pmatrix} 
0 & 0 & 0 & 0 \\
0 & 0 & 0 & 0 \\
0 & 0 & 0 & 1 \\
0 & 0 & 0 & 0 \\
\end{pmatrix}
$$
Check also that they are eigenvectors of $[H_i,-]$ for 
$$H_1=\begin{pmatrix} 
1 & 0 & 0 & 0 \\
0 & -1 & 0 & 0 \\
0 & 0 & 0 & 0 \\
0 & 0 & 0 & 0 \\
\end{pmatrix},\qquad 
H_2=\begin{pmatrix} 
0 & 0 & 0 & 0 \\
0 & 1 & 0 & 0 \\
0 & 0 & -1 & 0 \\
0 & 0 & 0 & 0 \\
\end{pmatrix},\qquad 
H_3=\begin{pmatrix} 
0 & 0 & 0 & 0 \\
0 & 0 & 0 & 0 \\
0 & 0 & 1 & 0 \\
0 & 0 & 0 & -1 \\
\end{pmatrix}
$$
and conclude that also the iterated commutators are eigenvectors of $[H_i,-]$.\\

\medskip

Compute all iterated commutators between the following matrices in $\mathfrak{so}_5$:
$$E_1=\begin{pmatrix} 
0 & 0 & 1 & 0 & 0\\
0 & 0 & 0 & 1 & 0\\
-1 & 0 & 0 & 0 & 0\\
0 & -1 & 0 & 0 & 0 \\
0 & 0 & 0 & 0 & 0 \\
\end{pmatrix},\qquad 
E_2=\begin{pmatrix} 
0 & 0 & 0 & 0 & 0\\
0 & 0 & 0 & 0 & 0\\
0 & 0 & 0 & 0 & 0\\
0 & 0 & 0 & 0 & 1\\
0 & 0 & 0 & -1 & 0\\
\end{pmatrix}
$$
and check that they are eigenvectors of $[H_i,-]$ for 
$$H_1=\begin{pmatrix} 
0 & 1 & 0 & 0 & 0\\
-1 & 0 & 0 & 0 & 0\\
0 & 0 & 0 & 0 & 0\\
0 & 0 & 0 & 0 & 0 \\
0 & 0 & 0 & 0 & 0 \\
\end{pmatrix},\qquad 
H_2=\begin{pmatrix} 
0 & 0 & 0 & 0 & 0\\
0 & 0 & 0 & 0 & 0\\
0 & 0 & 0 & 1 & 0\\
0 & 0 & -1 & 0 & 0\\
0 & 0 & 0 & 0 & 0\\
\end{pmatrix}
$$
\end{exercise}

\begin{exercise}
Consider the cases $B_3$ and $D_4$, that is
$$
(\alpha_i,\alpha_j)=\begin{pmatrix} 
2 & -2 & 0  \\
-2 & 4 & -4  \\
0 & -4 & 4 \\
\end{pmatrix}
\qquad
(\alpha_i,\alpha_j)=\begin{pmatrix} 
2 & -1 & -1 & -1 \\
-1 & 2 & 0 & 0 \\
-1 & 0 & 2 & 0 \\
-1 & 0 & 0 & 2 \\
\end{pmatrix}
$$
Starting with the simple roots and reflections, find all roots.
\end{exercise}

\section{Preliminaries on braided tensor categories}

\newcommand{\pentagon}[9]{ %4 Spaces, 5 Morphisms
\begin{center}
\begin{tikzpicture}[commutative diagrams/every diagram]
\node (P1) at (180:3.5cm)
{$(($#1$\otimes$#2$)\otimes$#3$)\otimes$#4};
\node (P2) at (180-45+10:2.5cm)
{$($#1$\otimes ($#2$\otimes$#3$))\otimes$#4};
\node (P3) at (180-45-90-10:2.5cm)
{#1$\otimes (($#2$\otimes $#3$)\otimes $#4$))$} ;
\node (P4) at (180-45-90-45:3.5cm)
{#1$\otimes($#2$\otimes ($#3$\otimes $#4$))$};
\node (P0) at (180+90:1cm)
{$($#1$\otimes$#2$)\otimes($#3$\otimes$#4$)$};
\path[commutative diagrams/.cd, every arrow, every label]
(P1) edge node {#5} (P2)
(P2) edge node[shift={(0,0.25)}] {#6} (P3)
(P3) edge node {#7} (P4)
(P1) edge node[swap] {#8} (P0)
(P0) edge node[swap] {#9} (P4);
\end{tikzpicture}
\end{center}
}

\newcommand{\hexagon}[5]{%1+2+2 Spaces, without and with action of degree a of first space A
\def\SpaceA{#1}%
\def\SpaceB{#2}%
\def\aSpaceB{#3}% degree of A acting on B
\def\SpaceC{#4}%
\def\aSpaceC{#5}% degree of A acting on C
\hexagonContinued
}
\newcommand\hexagonContinued[6]{%and 3+3 Morphisms
\begin{center}
\begin{tikzpicture}[commutative diagrams/every diagram,scale=0.70]
\node (P1) at (180:4.5cm)
{$($\SpaceA$\otimes$\SpaceB$)\otimes$\SpaceC};
%Assoz
\node (P2) at (180-60+20:3cm)
{\SpaceA$\otimes($\SpaceB$\otimes$\SpaceC$)$};
%Braid 1 with 23, ggf Tensor structure
\node (P3) at (180-60-60-20:3cm)
{$($\aSpaceB$\otimes$\aSpaceC$)\otimes$\SpaceA};
%Assoz
\node (P4) at (180-60-60-60:4.5cm)
{\aSpaceB$\otimes($\aSpaceC$\otimes$\SpaceA$)$};
%Assoz \\
%From Start. Braid 1 with 2
\node (P5) at (180+60-20:3cm)
{$($\aSpaceB$\otimes$\SpaceA$)\otimes$\SpaceC};
%Assoz
\node (P6) at (180+60+60+20:3cm)
{\aSpaceB$\otimes($\SpaceA$\otimes$\SpaceC$)$};
%Braid 2 with 3 \\
\path[commutative diagrams/.cd, every arrow, every label]
(P1) edge node {#1} (P2)
(P2) edge node[shift={(0,0.25)}] {#2} (P3)
(P3) edge node {#3} (P4)
(P1) edge node[swap] {#4} (P5)
(P5) edge node[swap, shift={(0,-0.25)}] {#5} (P6)
(P6) edge node[swap] {#6} (P4);
\end{tikzpicture}
\end{center}
}

\newcommand{\hexagonInverse}[5]{%1+2+2 Spaces, without and with action of degree a of first space A
\def\SpaceA{#1}%
\def\SpaceB{#2}%
\def\SpaceC{#3}%
\def\bSpaceC{#4}% degree of B acting on C
\def\abSpaceC{#5}% degree of A and B acting on C
\hexagonInverseContinued
}
\newcommand\hexagonInverseContinued[6]{%and 3+3 Morphisms
\begin{center}
\begin{tikzpicture}[commutative diagrams/every diagram,scale=0.70]
\node (P1) at (180:4.5cm)
{\SpaceA$\otimes($\SpaceB$\otimes$\SpaceC$)$};
%Assoz Inverse
\node (P2) at (180-60+20:3cm)
{$($\SpaceA$\otimes$\SpaceB$)\otimes$\SpaceC};
%Braid 12 with 3, ggf Tensor structure
\node (P3) at (180-60-60-20:3cm)
{\abSpaceC$\otimes($\SpaceA$\otimes$\SpaceB$)$};
%Assoz Inverse
\node (P4) at (180-60-60-60:4.5cm)
{$($\abSpaceC$\otimes$\SpaceA$)\otimes$\SpaceB};
%Assoz \\
%From Start. Braid 2 with 3
\node (P5) at (180+60-20:3cm)
{\SpaceA$\otimes($\bSpaceC$\otimes$\SpaceB$)$};
%Assoz Inverse
\node (P6) at (180+60+60+20:3cm)
{$($\SpaceA$\otimes$\bSpaceC$)\otimes$\SpaceB};
%Braid 1 with 2 \\
\path[commutative diagrams/.cd, every arrow, every label]
(P1) edge node {#1} (P2)
(P2) edge node[shift={(0,0.25)}] {#2} (P3)
(P3) edge node {#3} (P4)
(P1) edge node[swap] {#4} (P5)
(P5) edge node[swap, shift={(0,-0.25)}] {#5} (P6)
(P6) edge node[swap] {#6} (P4);
\end{tikzpicture}
\end{center}
}

\begin{definition}\label{def_BTC}
A \emph{braided tensor category} $\cC$ consists of: 
\hspace{-.4cm}
\begin{itemize}
\item Objects $X,Y,\ldots$ and morphisms $f:X\to Y$ between them. \\
The main concrete data we are interested in are:
\begin{itemize}
\item the list of 
\emph{simple objects} i.e. those without proper subobjects
\item whether all objects are \emph{semisimple objects} i.e. direct sum of simples 
\item otherwise information on \emph{indecomposable extensions} $A\to X\to B$, 
\item ultimately the structure of \emph{projective objects} and $\mathrm{Ext}^\bullet(A,B)$.
\end{itemize}
\item A tensor product $X\otimes Y$ and a unit object $\1$.
\item Choices of natural isomorphisms implementing associativity  
$$a_{X,Y,Z}:\;(X\otimes Y)\otimes Z\stackrel{\sim}{\longrightarrow}X\otimes (Y\otimes Z)$$
and as coherence conditions we ask for the following diagram to commute. This is called the \emph{pentagon identity} and it  ensures that different ways of rebracketing give the same result

\pentagon{$X$}{$Y$}{$Z$}{$T$}
{$a_{X,Y,Z}\otimes\id$}
{$a_{X,Y\otimes Z,T}$}
{$a_{X\otimes Y,Z,T}$}
{$\id\otimes a_{Y,Z,T}$}
{$a_{X,Y,z\otimes T}$}

\item Choices of natural isomorphisms  implementing commutativity 
$$X\otimes Y\stackrel{\sim}{\longrightarrow} Y\otimes X$$
and as coherence condition we ask for the following \emph{hexagon identity} that expresses that commuting with a tensor product agrees with commuting with the tensor factors individually

\hexagon{$X$}{$Y$}{$Y$}{$Z$}{$Z$}{$a_{X,Y,Z}$}{$c_{X,Y\otimes Z}$}{$a_{Y,Z,X}$}{$c_{X,Y}\otimes \id$}{$a_{Y,X,Z}$}{$\id\otimes c_{X,Z}$}

and a second version with a tensor product in the first slot of $c$. 

This natural isomorphism is
called \emph{braiding}, because as a consequence it satisfies the braid relation or Yang-Baxter equation. \\
Note that the double braiding is typically not the identity. 
\end{itemize}
\end{definition}

\begin{example}
    The category of vector spaces is a braided tensor category, with trivial associator and braiding. The unique simple object is the 1-dimensional vector spaces. The standard theorem about the existence of a basis means that any vector space is a direct sum of the $1$-dimensional vector space.
\end{example}

\begin{example}
    For a group $\Gamma$, the category $\Vect_\Gamma$ of $\Gamma$-graded $\K$-vector spaces has simple objects $\K_a,\,a\in\Gamma$, the one-dimensional vector space of degree $a$, and any object is a direct sum of these, so the category is semisimple.

    Similar to the category of vector spaces, this becomes a monoidal  category with the trivial associator, and if $\Gamma$ is abelian with the trivial  braiding. But we may also modify them by suitable choices of scalars $\omega:\Gamma\times\Gamma\times\Gamma\to \K^\times$ and $\sigma:\Gamma\times\Gamma\to \K^\times$ as follows:
$$
\begin{tikzcd}
(\K_a\otimes \K_b)\otimes \K_c
\arrow[rr, dashed,"a_{\K_a,\K_b,\K_c}"]
\arrow[d,equal]
&&
\K_a\otimes (\K_b\otimes \K_c)
\arrow[d,equal]
\\
\K_{a+b+c}
\arrow{rr}{\omega(a,b,c)}
&&
\K_{a+b+c}
\end{tikzcd}
\hspace{2cm}
\begin{tikzcd}
\K_a\otimes \K_b
\arrow[rr, dashed,"c_{\K_a,\K_b}"]
\arrow[d,equal]
&&
\K_b\otimes \K_a
\arrow[d,equal]
\\
\K_{a+b}
\arrow{rr}{\sigma(a,b)}
&&
\K_{b+a}
\end{tikzcd}
$$
The choices of scalars $(\omega,\sigma)$ up to the natural equivalence relation are classified by quadratic forms on the abelian group $\Gamma$, see \cites{MacL52,JS93}. The braiding is nondegenerate iff the quadratic form is nondegenerate. 
\end{example}

\begin{example}
    Representations of a group $G$ or a Lie algebra $\g$ become a braided tensor category: For representations $V,W$ we consider the tensor product of vector spaces and endow it with an action
    $$g.(v\otimes w)=g.v\otimes g.w,\qquad g\in G$$
    $$x.(v\otimes w)=x.v\otimes w+v\otimes x.w,\qquad x\in \g$$
    Because the assignment is symmetric, the trivial braiding of vector spaces is compatible with this action, moreover  it is associative on the nose with the associator of vector spaces.

    Recall for example that $\sl_2$ has simple representations $V_j$ of any dimension $n=2j+1$ and tensor product decomposes according to the Clebsch-Gordan rule
    $$V_j\otimes V_{j'}
    =V_{j+j'}\oplus \cdots \oplus V_{|j-j'|}$$
\end{example}

We give some exercises to get familiar with the nonsemisimple setup:

\begin{exercise}
For representations of the group $\Z_2=\langle g\rangle$ over a field $\K$ of characteristic $2$ show that the  representation $P=\K^2$ where $g$ permutes the base vectors is an indecomposable extension 
$$1\to P \to 1$$
In fact it is the unique projective object. Try to compute and decompose $P\otimes P$.
\end{exercise}

\begin{exercise}\label{exercise_RepPolynomialring}
Convince yourself that representations of the polynomial ring in one variable $\C[x]$ can be understood in terms of Jordan normalform, in particular simple objects are 1-dimensional $\C_\lambda$ and indecomposable objects are labeled by Jordan blocks~$\C_{\lambda,k}$. 

Now consider this as the universal enveloping algebra of a $1$-dimensional Lie algebra: Compute  the tensor products $\C_\lambda\otimes \C_\mu$ and $\C_{0,2}\otimes \C_{0,2}$.
\end{exercise}

Recall also the following construction of braided tensor categories

\section{Preliminaries on quantum groups}
\label{sec_quantumgroups}

Standard references for quantum groups are the books \cites{Lusz93, Jan96, Kas97}.
Quantum groups are $q$-deformations of the universal enveloping of a finite complex semisimple Lie algebra, defined by Jimbo and Drinfeld in the 1980s \cite{Drin89}. For example the quantum group $U_q(\sl_2)$ has generators $E,F,K=\exp(H)$ and relations 

\[ KE=q^2 EK,\qquad KF=q^{-2} FK,\qquad 
[E,F]=\frac{K-K^{-1}}{q-q^{-1}}\]

For $q$ not a root of unity, the representation theory is exactly as for the Lie algebra: There are Verma modules freely generated by a highest weight vector

\[ Kv=q^\lambda v,\qquad Ev=0\]

with basis $F^kv,k\geq 0$, and for integer $\lambda$ there are finite-dimensional quotient modules. 

The key property making quantum groups interesting is that they have an action on the tensor product (Hopf algebra, see below) via 
\[\Delta(K)=K\otimes K,\qquad 
\Delta(E)=K\otimes E+E\otimes 1,\qquad
\Delta(F)=1\otimes F+F\otimes K^{-1} \]
Again, the tensor product is the familiar Clebsch-Gordan rule, but there is a nontrivial braiding. 

\begin{exercise}
Explicitly show that for $\lambda=1$ the Verma module has a $2$-dimensional quotient module.
Write the tensor product of two such Verma modules and decompose them into a $1$-dimensional and a $3$-dimensional modules.  
\end{exercise}

For $q$ a primitive $n$-th root of unity Lusztig (who previously classified representations of finite simple groups of Lie type in finite characteristic) observed the representation theory becomes nonsemisimple, in some sense because weights differing by $n$ become equal and we can have circular behavior, similar as for the Lie algebra in finite characteristic (see Example \ref{exm_rankOneHopfAlg} below). On a Hopf algebra level it is astonishing that there is a tensor product on the representations of the quotient $U_q(\sl_2)$ by $E^n,F^n=0$. In this lecture, we want to systematically  understand this kind of heuristic behavior.

\chapter{Nichols algebras}

\section{Nichols algebras via quantum symmetrizer}

\subsection{Categorical algebras}\label{sec_catAlgebra}

In any tensor category there is a straightforward notion of an algebra $(A,\mu,\eta)$ with $\mu:A\otimes A\to A$ and $\eta:\1\to A$, as well as a notion of representation of $A$ inside $\cC$.

\begin{example}
    An algebra $A$ in the category of $\Gamma$-graded vector spaces $\cC=\Vect_\Gamma$ is a $\Gamma$-graded algebra. A module over $A$ in $\cC$ is a $\Gamma$-graded vector spaces with an action of $A$, such that the action is compatible with the $\Gamma$-gradings of $A$ and $M$
\end{example} 
\begin{example}
    For any object $X$ in any monoidal category we can form the tensor algebra
    $$\mathfrak{T}(X):=\bigoplus_{n\geq 0} X^{\otimes n}$$
    with multiplication given by juxtapositition
    $$\bigoplus_{n\geq 0} X^{\otimes n}\otimes \bigoplus_{m\geq 0} X^{\otimes m}
    \to \bigoplus_{n,m\geq 0} X^{\otimes (n+m)}$$
    Its representations $(Y,\rho)$ are in bijection with morphism $\rho|_X:\;X\otimes Y\to Y$. 
\end{example}

\subsection{Quantum symmetrizer}

Recall the symmetric group $\S_n$ on $n$ letters. It is generated by the $n-1$ neigbouring transpositions $\tau_i=(i,i+1)$. By a theorem of Matsumoto the following relations are defining relations 
$$(\tau_i)^2=1,\qquad \tau_i\tau_{i+1}\tau_i=\tau_{i+1}\tau_i\tau_{i+1},\qquad
\tau_i\tau_j=\tau_j\tau_i,\quad|i-j|>1$$
The \emph{length} $|\sigma|$ of a permutation $\sigma\in\S_n$ is the minimal length to express $\sigma$ as a word in the generators $\tau_i$. Such a word of minimal length is called a \emph{reduced expression}.

\begin{exercise}
Prove by induction that $\ell(\sigma)$ is equal to the number of pairs $1\leq i<j\leq n$ such that $\sigma(i)>\sigma(j)$.
\end{exercise}

Recall the group $\B_n$ of braids in $n$ strands. It is generated by $n-1$ overcrossing braid $c_{i,i+1}$. The defining relation are the second and third relation for the symmetric group, while $(c_{i,i+1})^2$ is nontrivial. 

\begin{definition}[Matsumoto section]
	Consider the obvious group homomorphism 
    $\mathbb{B}_n\to \mathbb{S}_n.\hspace{-1cm}$ \newline
	This does not split as a group homomorphism. But there is a well-defined section of sets 
    $$s:\mathbb{B}_n\leftarrow\mathbb{S}_n$$
	that maps the neigbouring transposition $\tau_i$ to the overcrossing braid $c_{i,i+1}$, 
	and maps any reduced expression to the respective product of braids.  
\end{definition}
The Matsumoto section is well defined because Matsumto's theorem asserts in addition that we can pass between any to reduced expressions of the same element in~$\S_n$ by only using the second and third relation. Since these are the defining relations of the braid group, the section to $\B_n$ is well-defined.

\begin{example}[$n=3$]
	The Matsumoto section maps the $6$ elements in the symmetric group $\mathbb{S}_3$ to the following braids in $\B_3$. In this lecture we read braids from top to bottom.
    \begin{center}
   \begin{tikzpicture}[scale=0.38]
     \newcommand{\shift}{0}
     \draw(1+\shift,2.7) node {\tiny $e$};
     \draw (0+\shift,0) -- (0+\shift,2);
     \draw (1+\shift,0) -- (1+\shift,2);
     \draw (2+\shift,0) -- (2+\shift,2);
     \renewcommand{\shift}{5}
     \draw(1+\shift,2.7) node {\tiny $(12)$};
     \draw (0+\shift,0) .. controls (0+\shift,1) and (1+\shift,1) .. (1+\shift,2);
     \draw [color = bg, line width = 5pt] (1+\shift,0) .. controls (1+\shift,1) and (0+\shift,1).. (0+\shift,2);
     \draw (1+\shift,0) .. controls (1+\shift,1) and (0+\shift,1) .. (0+\shift,2);
      \draw (2+\shift,0) -- (2+\shift,2);
      \renewcommand{\shift}{10}
      \draw(1+\shift,2.7) node {\tiny $(23)$};
      \draw (0+\shift,0) -- (0+\shift,2);
      \draw (1+\shift,0) .. controls (1+\shift,1) and (2+\shift,1) .. (2+\shift,2);
      \draw [color = bg, line width = 5pt] (2+\shift,0) .. controls (2+\shift,1) and (1+\shift,1).. (1+\shift,2);
      \draw (2+\shift,0) .. controls (2+\shift,1) and (1+\shift,1) .. (1+\shift,2);
      \renewcommand{\shift}{15}
      \draw(1+\shift,2.7) node {\tiny $(12)(23)$};
      \draw (0+\shift,0) .. controls (0+\shift,1) and (2+\shift,1) .. (2+\shift,2);
      \draw [color = bg, line width = 5pt] (1+\shift,0) .. controls (1+\shift,1) and (0+\shift,1) .. (0+\shift,2);
      \draw (1+\shift,0) .. controls (1+\shift,1) and (0+\shift,1) .. (0+\shift,2);
       \draw [color = bg, line width = 5pt](2+\shift,0) .. controls (2+\shift,1) and (1+\shift,1) .. (1+\shift,2);
      \draw (2+\shift,0) .. controls (2+\shift,1) and (1+\shift,1) .. (1+\shift,2);
      \renewcommand{\shift}{20}
      \draw(1+\shift,2.7) node {\tiny $(23)(12)$};
       \draw (0+\shift,0) .. controls (0+\shift,1) and (1+\shift,1) .. (1+\shift,2);
      \draw (1+\shift,0) .. controls (1+\shift,1) and (2+\shift,1) .. (2+\shift,2);
      \draw [color = bg, line width = 5pt] (2+\shift,0) .. controls (2+\shift,1) and (0+\shift,1).. (0+\shift,2);
      \draw (2+\shift,0) .. controls (2+\shift,1) and (0+\shift,1) .. (0+\shift,2);
	  \renewcommand{\shift}{25}
	  \newcommand{\yoff}{0}
	  \newcommand{\xoff}{-0.7}
	  \newcommand{\xoffoff}{-1}
	  \draw(1+\shift-0.5,3-0.2) node {\tiny $(12)(23)(12)$};
	  \draw (0+\shift,0) .. controls (0+\shift-\xoff,1+\yoff) and (2+\shift-\xoff,1+\yoff) .. (2+\shift,2);
    \draw [color = bg, line width = 5pt] (1+\shift,0)  .. controls (1+\shift+\xoffoff,1)..(1+\shift,2);
    \draw (1+\shift,0)  .. controls (1+\shift+\xoffoff,1)..(1+\shift,2);
	  \draw [color = bg, line width = 5pt] (2+\shift,0) .. controls (2+\shift-\xoff,1-\yoff) and (0+\shift-\xoff,1-\yoff).. (0+\shift,2);
	  \draw (2+\shift,0) .. controls (2+\shift-\xoff,1-\yoff) and (0+\shift-\xoff,1-\yoff) .. (0+\shift,2);
	  \renewcommand{\shift}{28.6}
	  \renewcommand{\xoff}{0.7}
	  \renewcommand{\xoffoff}{1}
	  \draw(1+\shift,3-0.2) node {\tiny $=(23)(12)(23)$};
     \draw (0+\shift,0) .. controls (0+\shift-\xoff,1+\yoff) and (2+\shift-\xoff,1+\yoff) .. (2+\shift,2);
     \draw [color = bg, line width = 5pt](1+\shift,0)  .. controls (1+\shift+\xoffoff,1)..(1+\shift,2);
	  \draw (1+\shift,0)  .. controls (1+\shift+\xoffoff,1)..(1+\shift,2);
	  \draw [color = bg, line width = 5pt] (2+\shift,0) .. controls (2+\shift-\xoff,1-\yoff) and (0+\shift-\xoff,1-\yoff).. (0+\shift,2);
	  \draw (2+\shift,0) .. controls (2+\shift-\xoff,1-\yoff) and (0+\shift-\xoff,1-\yoff) .. (0+\shift,2);
	  
	  \draw(-0.7+\shift,1.2) node {\tiny $=$};
   \end{tikzpicture}
   \end{center}
   \smallskip
   The last element has two different reduced expressions that are mapped to equivalent braids.
\end{example}

\begin{definition}[\cite{HS20} Definition 1.8.7]
    The \emph{quantum symmetrizer} or \emph{braided symmetrizer} is the following formal element in $\Z[\B_n]$, the $\Z$-span of the braid group 
	$$\sym_n: =\sum _{\tau\in \mathbb{S}_n} s(\tau)$$
    For later use we also define the \emph{quantum shuffle}
$$\sym_{i,n-i}: =\sum _{\tau\in \mathbb{S}_{i,n-i}} s(\tau^{-1})$$
    where $\mathbb{S}_{i,n-i}$ denotes the subgroup of shuffles i.e. permutations $\tau$ such that $\tau(1)<\cdots <\tau(i)$ and $\tau(i+1)<\cdots<\tau(n)$.
\end{definition}

\subsection{Nichols algebras}
Let now $X\in \cC$ be an object in a braided monoidal  category.
The braiding isomorphism ${c:X\otimes X\to X\otimes X}$ fulfills the \emph{braid relation} or \emph{Yang-Baxter equation} on $X\otimes X\otimes X$ 
\begin{center}
\begin{tikzpicture}
    \newcommand{\shift}{25}
	  \newcommand{\yoff}{0}
	  \newcommand{\xoff}{-0.7}
	  \newcommand{\xoffoff}{-1}
	  \draw(1+\shift,3-0.2) node {\tiny $(c\otimes \id)(\id\otimes c)(c\otimes \id)$};
	 \draw (0+\shift,0) .. controls (0+\shift-\xoff,1+\yoff) and (2+\shift-\xoff,1+\yoff) .. (2+\shift,2);
    \draw [color = bg, line width = 5pt] (1+\shift,0)  .. controls (1+\shift+\xoffoff,1)..(1+\shift,2);
    \draw (1+\shift,0)  .. controls (1+\shift+\xoffoff,1)..(1+\shift,2);
	  \draw [color = bg, line width = 5pt] (2+\shift,0) .. controls (2+\shift-\xoff,1-\yoff) and (0+\shift-\xoff,1-\yoff).. (0+\shift,2);
	  \draw (2+\shift,0) .. controls (2+\shift-\xoff,1-\yoff) and (0+\shift-\xoff,1-\yoff) .. (0+\shift,2);
	  \renewcommand{\shift}{28.6}
	  \renewcommand{\xoff}{0.7}
	  \renewcommand{\xoffoff}{1}
	  \draw(1+\shift,3-0.2) node {\tiny $(\id\otimes c)(c\otimes \id)(\id\otimes c)$};
	  \draw (0+\shift,0) .. controls (0+\shift-\xoff,1+\yoff) and (2+\shift-\xoff,1+\yoff) .. (2+\shift,2);
     \draw [color = bg, line width = 5pt](1+\shift,0)  .. controls (1+\shift+\xoffoff,1)..(1+\shift,2);
	  \draw (1+\shift,0)  .. controls (1+\shift+\xoffoff,1)..(1+\shift,2);
	  \draw [color = bg, line width = 5pt] (2+\shift,0) .. controls (2+\shift-\xoff,1-\yoff) and (0+\shift-\xoff,1-\yoff).. (0+\shift,2);
	  \draw (2+\shift,0) .. controls (2+\shift-\xoff,1-\yoff) and (0+\shift-\xoff,1-\yoff) .. (0+\shift,2);

      	\draw(-0.7+\shift,2.78) node {\tiny $=\;$};
      \end{tikzpicture}
   \end{center}
Hence it extends to an action of the braid group on tensor powers of $X^{\otimes k}$ for any $k\geq 0$:
$$\rho:\B_k\to \End(X^{\otimes k})$$
Note that in a general braided monoidal category there are associators present in these formulas, which we suppress in our pictures. 
\begin{definition}
%[\small Nichols '78, Woronowicz '89, Lusztig '93]
\label{def:symNichols}
	The Nichols algebra $\NicholsOf(X)$ is the tensor algebra (free algebra) 
	modulo the kernel of the action of the quantum symmetrizer 
    $$\NicholsOf(X):=\bigoplus_{k\geq 0} X^{\otimes k}/\ker(\rho(\sym_k))$$    
\end{definition}

\begin{lemma}
The kernel is indeed an ideal, so this defines indeed an algebra $\NicholsOf(X)$.
\end{lemma}
\begin{proof}
We show the kernel is a right ideal (left ideal is similar): 

Consider $\mathbb{S}_n\subset \mathbb{S}_{n+1}$ as the subgroup of permutation fixing $n+1$. The elements in $\mathbb{S}_{n,1}$ are coset representatives for right $\mathbb{S}_n$-cosets in $\mathbb{S}_{n+1}$. Explicitly: Any permutation of $1,\ldots,n,n+1$ can be written uniquely as a permutation of $1,\ldots,n$ (their designated position relative to each other) followed by shifting $n+1$ to its designated position while keeping the others in their relative order.

A key observation now is that these coset representatives are compatible with the length function. Differently spoken, any reduced expression in $\mathbb{S}_n$ together with a reduced expression iof the coset representative in $\mathbb{S}_{n,1}$ is again a reduced expression. This is clear, because by construction no pair $(i,j)$ inverted by $\mathbb{S}_n$ is again inverted by $\mathbb{S}_{n,1}$.

Hence we have a factorization of the quantum symmetrizer $\sym_{n+1}=\sym_{n, 1}'\sym_n=\sym_n\sym_{n,1}$ where $\sym_{n,1}'$ is the shuffle element without the convention of using inverse elements. In particular multiplying the kernel of $\sym_n$ from the right with some $(n+1)$-st factor $X$ lands again in the kernel.
\end{proof}
Note that the proof gives inductively a decomposition 
\begin{align}\label{formula_factorization}
\sym_{n}=\sym_{1,1}\sym_{2,1}\cdots \sym_{n-1,1}
\end{align}
see \cite{HS20} Corollary 1.8.8, and the entire Section 1.8 for more formulae in this spirit.

\begin{slogan}
The Nichols algebra measures to which extend the braiding does not factor over the symmetric group, and to which extend the Matsumoto section is not a group homomorphism. 
\end{slogan}

\subsection{First examples}

\begin{example}[Braided vector spaces, diagonal braidings]
    Main examples of interest are braided vector spaces $(X,c)$. 
    
    \medskip
    
    An important class of examples are braided vector spaces of \emph{diagonal type}, that is $X$ comes with a distinguished basis  $x_1,\ldots,x_n$ and the braiding is of the form
    $$c_q:\,X\otimes X\to X\otimes X$$
    $$x_i\otimes x_j \mapsto q_{ij}(x_j\otimes x_i)$$
    for a \emph{braiding matrix} $(q_{ij})$ consisting of arbitrary nonzero scalars. Following Heckenberger, we depict such a braiding by writing a $q$-diagram with nodes for each $\alpha_i$, decorated by the self braiding $q_{ii}$, and edges from $\alpha_i$ to $\alpha_j$, decorated by the double braiding $q_{ij}q_{ji}$, and the edge is only drawn if $q_{ij}q_{ji}\neq 1$:
\begin{center}
\begin{tikzpicture}
		\draw (0,0)--(1.6,0);
		\draw (-0.1,0) circle[radius=0.1cm] node[anchor=south]{$ q_{11}$}
		(1.7,0) circle[radius=0.1cm] node[anchor=south]{$ q_{22}$};
		\draw (0.8,0) node[anchor=south]{$ q_{12}q_{21}$};
\end{tikzpicture} 
\end{center}
This does not fix the braiding completely, but it turns out to contain all relevant information.  
\end{example}
Note that the definition of a Nichols algebra only depends on the action of the braid group on $X^{\otimes k}$, not on an ambient braided monoidal category. Any diagonal braiding can be realized in $\Vect_{\Z^n}^{\chi,1}$ by multiplicatively extending $q_{ij}$ to a bicharacter 
$$\chi:\;\Z^n\times\Z^n\to \K^\times$$
$$\chi(\alpha_i,\alpha_j)=q_{ij}$$
It can be of course also realized for other abelian groups $\Z^\to \Gamma$, but for most questions this will not make a difference.\footnote{For example, computing all possible liftings indeed depends crucially on $\Gamma$. For some arguments it is sufficients to ask that $\Gamma$ is large enough so that all roots (degrees of PBW generators) are different, we called this sufficiently unrolled in \cite{CLR23}.} 
From the categorical view on $\Vect_\Gamma^{\omega,\sigma}$, it encodes the quadratic form and the associated bimultiplicative form, which is precisely the information invariant under equivalences of braided tensor categories.

\begin{exercise}
Show that $c_q$ for any braiding matrix $(q_{ij})$ fulfills the braid relation.
\end{exercise}
\begin{example}[$k=2$] The first nontrivial term is  
    $$\sym_{2}=\mathrm{id}_{X,X}+c_{X,X}$$
    If the braiding is symmetric $c^2=\id$, then the image of $\mathrm{id}-c$ is in the kernel, meaning we have a \emph{braided commutator relation} 
    $$[x,y]_c:=x\otimes y-c(x\otimes y)=0,\quad \text{for } c^2=\id$$
    For example, for a braided vector space with diagonal braiding 
    $$[x,y]_q:=x_ix_j-q_{ij} x_jx_i=0,\quad \text{for } q_{ij}q_{ji}=1$$
    as we can see explicitly by computing the quantum symmetrizer
    $$\sym_{q,2}[x,y]_q
    =(x_ix_j-q_{ij} x_jx_i)
    +(q_{ij}x_jx_i-q_{ij}q_{ji}x_ix_j)
    =0$$

    For a non-symmetric braiding the braided commutator relations would not be consistent, because 
    $x_ix_j-q_{ij} x_jx_i=0$ and $x_jx_i-q_{ji} x_ix_j=0$ contradict each other if ${q_{ij}q_{ji}\neq 1}$. However, higher order relations may be consistent in certain constellations of $(q_{ij})$. For ''random'' braidings, the Nichols algebra is the free algebra.   
\end{example}
\begin{definition}
    Let $X,(q_{ij})$ be a diagonally braided vector space, then we define the braided commutator of two arbitrary elements in the free algebra over $X$ by 
    $$[x_{i_1}\cdots x_{j_n},x_{j_1}\cdots x_{j_m}]_q
    =
    (x_{i_1}\cdots x_{i_n})(x_{j_1}\cdots x_{j_m})
    -\left(\prod_{a=1}^n\prod_{b=1}^m q_{i_a j_b}\right)
    (x_{j_1}\cdots x_{j_m})(x_{i_1}\cdots x_{i_n})
    $$
\end{definition}

\begin{example}[Rank $1$]
Let $X=\langle x\rangle_\K$ with diagonal braiding $q_{11}=q$ then 
$$\NicholsOf(X)=\begin{cases}
\K[x]/x^\ell, &\quad \text{if }1+q+q^2+\cdots+q^{\ell-1}=1\\
\K[x], &\quad {\text{else}}
\end{cases}$$
The condition on the first case applies for example if $q$ is an $\ell$-th root of unity or if $q=1$ and $\ell$ divides the characteristic of the base field $\K$.
\begin{proof}
The space $X^{\otimes k}$ is spanned by a single basis vector $(x\otimes\cdots \otimes x)$ and any braid group generator $c_{i,i+1}$ acts by multiplication with $q$. Then we compute
\begin{align*}
\sym_{q,k}(x\otimes\cdots \otimes x)
&=\left(\sum_{\sigma\in\mathbb{S}_k} q^{|\sigma|}\right)(x\otimes\cdots \otimes x)\\
&=\left(\prod_{1\leq i\leq k} \frac{q^i-1}{q-1}\;\right)(x\otimes\cdots \otimes x)\\
\intertext{We list the first cases}
\sym_{q,2}(x\otimes x)
&=(1+q)(x\otimes x)\\
\sym_{q,3}(x\otimes x \otimes x)
&=(1+q+q+q^2+q^2+q^3)(x\otimes x \otimes x)\\
&=(1+q)(1+q+q^2)(x\otimes x \otimes x) 
\end{align*}
This factorization appears in the context of invariant theory, but it is also a special case of the very general decomposition of $\sym_n$ into shuffles in formula \eqref{formula_factorization}.
\end{proof}
\end{example}

\begin{exercise}
For any element in the symmetric group $g\in\S_n$, let $|\sigma|$ be the number of pairs $(i,j)$ with $i<j$ and $\sigma(i)>\sigma(j)$. Prove directly the following factorization of the generating function 
$$\sum_{\sigma\in\mathbb{S}_k} q^{|\sigma|}
=\prod_{1\leq i\leq k} \frac{q^i-1}{q-1}$$
Prove also that $|\sigma|$ is the length of any shortest expression of $\sigma$ as product of neighbouring transpositions $(i,i+1)$. We remark that further interpretations come from interpreting $\S_n$ as Weyl group of the root system $A_{n-1}$, but in our context this could cause confusion between the two different roles of the symmetric group.
\end{exercise}

\begin{lemma}[Rosso's Lemma 14 \cite{Ros98}, for a systematic proof see \cite{HS20} Lemma 15.1.1]\label{lm_Rosso}
Let $X$ with diagonal braiding $(q_{ij})=q$ we have the following relation between two elements 
\begin{align*}
\sym_{q,n+1}\underbrace{[x_i,[x_i,\cdots [x_i}_n,x_j]_q]_q]_q=0
%\sym_{q,n+1}(x_1^{\otimes n}\otimes x_2)
%&=(1+q_{11}+\cdots+q_{11}^{n-1})(1-q_{11}^{n-1}q_{12}q_{21})
\end{align*}
and the relation does not hold for smaller $n$ iff 
$$1+q_{ii}+\cdots+q_{ii}^{n-1}=0 \qquad \text{ or } \qquad q_{ii}^{n-1}q_{ij}q_{ji}=1\quad$$
$$\text{\bf ''Truncation case''}\hspace{2cm}
\text{\bf ''Cartan case''}$$
\end{lemma}
Note that this only depends on $q_{ii}$ and $q_{ij}q_{ji}$, which is the information in the $q$-diagram.

The case $n=1$ reduces to the previous commutativity statement
\begin{exercise}
    Check the case $n=2$ by hand by computing the quantum symmetrizer.
\end{exercise}

\begin{example}[Quantum Borel part]\label{exm_quantumgroup}
	Let $\g$ be a semisimple finite-dimensional Lie algebra $\g$ with simple roots $\alpha_i$, root lattice $\Lambda$ and Killing form $(\alpha_i,\alpha_j)$. Let  $q\in\mathbb{C}^\times$ be fixed. We define the diagonal braiding $q_{ij}=q^{(\alpha_i,\alpha_j)}$.\\
    
	Then $\NicholsOf(q)=u_q(\g)^+$ is the Borel part of the quantum group or small quantum group (in case $q$ is root of unity). As an example, the $q$-diagrams for $\g$ of type $A_1\times A_1$, $A_2$ and $B_2$ are
\[
	\begin{tikzpicture}
	\draw (-0.1,0) circle[radius=0.1cm] node[anchor=south]{$ q^2$}
	(1.7,0) circle[radius=0.1cm] node[anchor=south]{$ q^{2}$};
	\end{tikzpicture}
	\hspace{2cm}
	\begin{tikzpicture}
	\draw (0,0)--(1.6,0);
	\draw (-0.1,0) circle[radius=0.1cm] node[anchor=south]{$ q^2$}
	(1.7,0) circle[radius=0.1cm] node[anchor=south]{$ q^{2}$};
	\draw (0.7,0) node[anchor=south]{$ q^{-2}$};
	\end{tikzpicture}
	\hspace{2cm}
	\begin{tikzpicture}
	\draw (0,0)--(1.6,0);
	\draw (-0.1,0) circle[radius=0.1cm] node[anchor=south]{$ q^2$}
	(1.7,0) circle[radius=0.1cm] node[anchor=south]{$ q^{4}$};
	\draw (0.7,0) node[anchor=south]{$ q^{-4}$};
	\end{tikzpicture}
\]

    We observe that here the second case in Rosso's Lemma \ref{lm_Rosso} produces the quantum Serre relation (whence "Cartan case"): The Cartan matrix for a Lie algebra was defined 
    $${\cm_{ij}=\frac{2(\alpha_i,\alpha_j)}{(\alpha_j,\alpha_j)}}$$
    Thus the braiding matrix $q_{ij}=q^{(\alpha_i,\alpha_j)}$ precisely fulfills the condition $q_{11}^{n-1}q_{12}q_{21}=0$ for
    $$(n-1)(\alpha_i,\alpha_i)+2(\alpha_i,\alpha_j)=0$$
    meaning for $n=1-\cm_{ij}$. Note however that for small order roots of unity the truncation condition can be fulfilled for smaller values of $n$, see Examples \ref{ex:superSl2}  \\

    Similarly, the quantum Borel parts for Lie super algebras are Nichols algebras. For fermionic roots we have the "Truncation case" in Rosso's Lemma.
\end{example}

\begin{exercise}[\cite{Len16}]
For $\g=G_2$ and $q$ a primitive fourth root of unity, what are the quantum Serre relations in view of Rosso's lemma in comparison to the usual quantum Serre relations (i.e. of general $q$)? 

If we take the quotient of the tensor algebra by the usual quantum Serre relation, find an additional primitive element in higher degree. Compute the braiding with the generators $x_1,x_2$ and draw the $q$-diagrams. Verify that this has the same braiding matrix and hence set of Serre relations as $A_3$.
\end{exercise}

\begin{exercise}[\cite{HS20} Example 7.1.3]
As a pathological example: Let $V$ be a vector space and $c=\id_{V\otimes V}$. Compute the Nichols algebra over a base field $\mathbb{K}$ of characteristic $p$. 
\end{exercise}

\section{Nichols algebras as universal Hopf algebras}

As it turns out, the Nichols algebra has much more structure

\subsection{Categorical Hopf algebras}

\begin{definition}
A \emph{bialgebra} $H$ in a braided monoidal  category $\cC$ is an algebra inside $\cC$ endowed with algebra morphisms called coproduct and counit
$$\Delta:H\to H\otimes H,\qquad \epsilon:H\to \1$$
fulfilling axioms dual to those of product and unit in an algebra 
$$(\Delta\otimes \id)\Delta=(\id\otimes \Delta)\Delta,\qquad
(\epsilon\otimes\id)\Delta=(\id\otimes \epsilon)\Delta=\id
$$
It is common to use the \emph{Sweedler notation} $\Delta(h)=h^{(1)}\otimes h^{(2)}$, where there is an implicit summation over possibly several elementary tensors. In particular now coassiciativity reads 
$$(h^{(1)})^{(1)}\otimes (h^{(1)})^{(2)} \otimes h^{(2)}=h^{(1)}\otimes(h^{(2)})^{(1)}\otimes (h^{(2)})^{(2)}$$
and use for both expressions the shorthand notation $h^{(1)}\otimes h^{(2)}\otimes h^{(3)}$, similarly as we write a triple product $abc$ for $(ab)c=a(bc)$.
\end{definition}
In order to define the notion of an algebra morphism $H\to H\otimes H$ we need an algebra structure on $H\otimes H$ and this requires in general a braiding on $\cC$, namely:
$$(H\otimes H)\otimes (H\otimes H)
\cong
H\otimes (H\otimes H)\otimes H
\stackrel{\id\otimes c\otimes \id}{\longrightarrow} 
H\otimes (H\otimes H)\otimes H
\cong 
(H\otimes H)\otimes (H\otimes H)
\stackrel{\mu\otimes \mu}{\longrightarrow} 
H\otimes H
$$

Good references for Hopf algebras in vector spaces are \cites{Mon93, Kas97, Schn95}.
The categorical point is that the additional structure of a bialgebra endows the category $\Rep(H)(\cC)$ of representations of $H$ inside $\cC$ with a monoidal  product (namely $V\otimes_\cC W$ and $H$-action via $\Delta$) and a unit object (namely $1_\cC$ and $H$-action via $\epsilon$.  is by endowed with the structure. The antipode in addition endows the category with a dual. 

\begin{example}
Let $\cC$ be the category of vector spaces with the trivial braiding. Two very familiar examples of Hopf algebras, which motivate much of the theory, are as follows:
\begin{itemize}
\item Let $G$ be a group, then the group algebra $\K G$ is the vector space spanned of $G$, with the multiplication inherited from $G$. It becomes a Hopf algebra with the structures
$$\Delta(g)=g\otimes g,\qquad \varepsilon(g)=1,\qquad S(g)=g^{-1}.$$
This additional structures describes how a group acts usually on the tensor product, the trivial representation and the dual representation.
\item Let $\g$ be a Lie algebra, then the universal enveloping algebra $U(\g)$ is the tensor algebra of $\g$ (i.e. the free algebra in a basis of $\g$) modulo relations $xy-yx=[x,y]_\g$. It becomes a Hopf algebra with the structures
$$\Delta(x)=x\otimes 1+1\otimes x,\qquad \varepsilon(x)=0,\qquad S(x)=-x.$$
This additional structures describes how a Lie algebra, for example an algebra of derivations, acts usually on the tensor product, the trivial representation and the dual representation.
\end{itemize}
%Both tensor categories admit a trivial braiding, because $\Delta$ is cocommutative.
\end{example}

\subsection{Nichols algebras}\label{sec_DefHA}

\begin{slogan}
    A Nichols algebra $\NicholsOf(X)$ is a universal Hopf algebra assigned to an object $X$ in a braided monoidal  category $\cC$. It is the smallest Hopf algebra generated by primitive elements $X$.
\end{slogan}

\begin{definition}
Let $X$ be any object in a braided monoidal  category $\cC$. The tensor algebra 
$$\mathfrak{T}(X)=\bigoplus_{k\geq 0} X^{\otimes k}$$
can be endowed uniquely with a Hopf algebra structure by setting 
$$\Delta(x)=1\otimes x+x\otimes 1,\quad\epsilon(x)=0,\quad S(x)=-x, \quad x\in X,$$
(such an element is called a \emph{primitive element}) and extending multiplicatively to an algebra map. 
\end{definition}

The multiplicative extension depends on the algebra structure on $H\otimes H$ and thus on the braiding. This may lead to surprises, for example a power  of  a primitive element may happen to be primitive itself. Before going to the braided setting, let us discuss the same effect in a finite characteristic example with trivial braiding:

\begin{example}\label{exm_rankOneHopfAlg}
Let $\Delta(x)=1\otimes x+x\otimes 1$ (say, a derivation) and let the characteristic of the base field be $\ell$. Then  
$$\Delta(x^\ell)=\sum_{i=0}^\ell \binom{\ell}{i} (x^i\otimes x^j)
=1\otimes x^\ell+x^\ell\otimes 1$$
is again primitive
because $\ell|\binom{\ell}{i}$ for $\ell$ prime and $i\neq 0,\ell$. Accordingly, there is a quotient Hopf algebra setting $x^\ell=0$ or to another primitive element. This lead to the classical notion of $\ell$-restricted Lie algebras over fields of finite characteristic. It fits to the earlier definition via quantum symmetrizers, because it is simply $\sym_n=n!$, which is zero for $n\geq \ell$.
\end{example}

Now it was Lusztig's observation that the same can happen in characteristic zero with nontrivial braiding:

\begin{example}\label{exm_NicholsRank1}
Let us take $\cC=\Vect_\Gamma^{\omega,\sigma}$, and let $X=\K_a$ be the one-dimensional vector space in degree $a\in \Gamma$ with basis $x$. We compute 
$$\Delta(x^2)=\Delta(x)^2=(1\otimes x+x\otimes 1)^2
=x^2\otimes 1+(x\otimes x)+\sigma(a,a)(x\otimes x)+x^2\otimes 1.$$
In particular if $\sigma(a,a)=-1$, then $x^2$ is a again primitive.
Similarly, if $\sigma(a,a)$ is a primitive $\ell$-th root of unity, then $x^\ell$ is primitive. This follows from the following computations:
\end{example}
\begin{lemma}[\cite{HS20} Section 1.9]
Following Gau{\ss} \cite{Gauss1808}\footnote{The original use in this work was to determine the sign of the $p$-th Gau{\ss} sum , by re-expressing it in terms of the as the alternating sum $\sum_k (-1)^k\binom{p-1}{k}_q$ for $q$ a $p$-th root of unity.}
% (1-q)(1-q^3)\cdots 
% 1+1/q+1/q^3+...triagonal numbers
we define 
\begin{align*}
(n)_q&=\frac{q^n-1}{q-1}
=1+q+\cdots+q^{n-1}\\
(n)_q!&=(n)_q(n)_{q-1}\cdots (1)_q \\
\binom{n}{ k}_q
&=\frac{(n)_q!}{(k)_q!(n-k)_q!}
\intertext{for $0\leq k\leq n$ and zero otherwise. These notions enjoys the following properties}
(n)_q&=(k)_q+q^k(n)_q \\
(n)_q&=0  \text{ for  $q$ an $n$-th root of unity}\\
&\\
\binom{n}{k}_q
&=\binom{n}{ n-k}_q\\
\binom{n}{k}_q
&=\binom{n-1}{k-1}_q
+q^k\binom{n-1}{k}_q
= q^{n-k}\binom{n-1}{k-1}_q
+\binom{n-1}{k}_q\\
\binom{n}{k}_q&\in \Z[q]\\
\binom{n}{k}_q&=0 \text{ for } 0\lneq k\lneq n \text{ and $q$ a primitive $n$-th root of unity.}
\end{align*}
For $q\to 1$ these reduce to natural numbers, factorials and binomial coefficients. Note that $\binom{n}{ k}_q$ has different combinatorial interpretations, for example the number of restricted partitions and the number of complete flags in the field $\mathbb{F}_q$. Directly relevant to our context is the quantum symmetrizer of all $(k,n-k)$-shuffles in \cite{HS20} Example 1.9.6. 
\end{lemma}
\begin{exercise}
Prove the previous lemma on $q$-binomials. 
\end{exercise}
\begin{exercise}
Prove that for braiding $x\otimes x\mapsto q(x\otimes x)$ we have
$$\Delta(x^n)
=\sum_{k=0}^n \binom{n}{ k}_q (x^k\otimes x^{n-k})$$
\end{exercise}

This unexpected appearance of higher primitive elements has the following implication: Usually, for a primitive element $x$ any algebra relation $x^\ell=0$ and compatibility with $\Delta$ forces all lower powers including $x$ itself to be zero. If however all mixed terms in $\Delta(x^\ell)$ happen to cancel as above, then $x^\ell=0$ is compatible with $\Delta$. Differently spoken $(x^\ell)$ is a Hopf ideal. Conversely, any graded Hopf algebra quotient has a smallest element gradewise, which has to be primitive. 

\begin{definition}
    Let $X$ be any object in a braided monoidal category $\cC$. Then the Nichols algebra $\NicholsOf(X)$ can be defined as a Hopf algebra quotient of $\mathfrak{T}(X)$ with the following properties, that are equivalent at least in the case of a braided vector spaces by \cite{HS20} Theorem 7.1.2
    \begin{enumerate}[(a)]
        \item $\NicholsOf(X)$ has no graded Hopf algebra quotients that are injective on $X$.
        \item $\NicholsOf(X)$ has no primitive elements besides $X$.
        \item $\NicholsOf(X)$ is the algebra defined in terms of the quantum symmetrizer in Definition \ref{def:symNichols}.
        \item The Hopf algebra pairing $\mathfrak{T}(X)\otimes \mathfrak{T}(X^*)\to \1$ induced by $X\otimes X^*\to \1$ factors to a nondegenerate pairing on $\NicholsOf(X)\otimes \NicholsOf(X^*)$. Here, a Hopf pairing $\langle-,-\rangle$ is a linear map $H\otimes L\to \K$ such that $\langle a,bc\rangle$ is related to the coproduct of $a$ inserted in $\langle-,b\rangle\langle-,c\rangle$, see \cite{HS20} Definition 3.3.7. 
        \item The only elements in the kernel of all skew derivations $\partial_i$ is $1$. Here $\partial_i$ is defined by $\partial_i(x_j)=\delta_{ij}$ and $\partial_i(ab)=\partial_i(a)b+K_i(a)\partial_i(b)$ where $K_i$ is the automorphism given by $K_i(x_j)=q_{ij}$ and $K_i(ab)=K_i(a)K_i(b)$. See 
        \cite{HS20} Chapter 7.3.
    \end{enumerate}
    Of course in general categories primitive elements have to be replaced by equalizer subobjects.
    \end{definition}
    The last was actually used by Nichols in his empirical study of new bialgebras and the second last was initial characterization of Nichols algebras appearing in  quantum groups in Lusztig's book \cite{Lusz93}, there called \emph{algebra $\mathfrak{f}$}.
    \begin{problem}
        It should be worked out rigorously that these definitions are equivalent in a general categorical setup. This would be a suitable master project. 
    \end{problem}

    Let us quickly summarize the ideas behind these equivalences:
    \begin{itemize}
    \item The equivalence $(a)\Leftrightarrow (b)$ was already  motivated. One direction is clear, because a primitive element by itself is a coideal and it generates a Hopf ideal. For the other direction, observe that for every graded coideal an element of minimal degree $n$ has a coproduct in degrees $(i,j)$ for $i+j=n$ and by the assumed minimality only $(n,0)$ and $(0,n)$. Counitality then shows the element is primitive.
    \item The equivalence $(b)\Leftrightarrow(c)$ uses the following explicit formula for the coproducts in  $\mathfrak{T}(X)$ in terms of shuffles, see \cite{HS20} Theorem 6.4.9:
    $$ \bigoplus_{n\geq 0} X^{\otimes n}\to \bigoplus_{i\geq 0} X^{\otimes i} \otimes \bigoplus_{j\geq 0} X^{\otimes j}$$
    $$\Delta=\bigoplus_{i,j} \sym_{i,j}$$
    This is proven by a simple induction, multiplying one more $\Delta(x)=1\otimes x+x\otimes 1$ with appropriate braiding. As a particular example, the component in degree $(1,\ldots,1)$ of the $n$-th coproduct of $X^{\otimes n}$ is the quantum symmetrizer. 
    \item The equivalence $(c)\Leftrightarrow(d)$ is straightforward, because when we evaluate the pairing of products we have to compute the coproduct of a product.
    \item The equivalence in $(c)\Leftrightarrow(e)$ holds since that skew derivations are essentially the terms in degree $(1,n-1)$ resp. $(n-1,1)$. As such they are practical and efficient in by-hand calculations.
    \end{itemize}

\begin{exercise}
Check the formula for the coproduct in $\mathfrak{T}(X)$ for a diagonally braided vector space in degree $n=4$.
\end{exercise}

\begin{exercise}
    Consider the example of $X$ a braided vector space of rank $1$. Recall and complete that in this case Definitions (a)-(c) for the Nichols algebra all produce the same result.
    For (e) prove that 
    $$\partial x^n=(n)_q x^{n-1}$$
    Also find a formula for the $k$-th derivation. 
    
    Optionally: For (d) prove that that 
    $$((x^*)^n,x^m)=\delta_{n,m}[n]_q!$$  
\end{exercise}
\begin{exercise}
    Show that for $q_{ij}q_{ji}=1$ the braided commutator $[x_i,x_j]_q$ is a primitive element.
\end{exercise}

\begin{exercise}
    Prove the following explicit formulas for $x_i,x_j$ with diagonal braiding $(q_{ij})$ implying Rosso's lemma and the quantum Serre relations, see \cite{HS20} Proposition 4.3.12:
    \begin{align*}
    \underbrace{[x_i,[x_i,\cdots [x_i}_n,x_j]_q]_q]_q
    &=\sum_{k=0}^n (-1)^k q_{ij}^k
    q_{ii}^{k(k-1)/2} \binom{n}{ k}_q
    x_i^{n-k}x_jx_i^k \\
    \Delta(\underbrace{[x_i,[x_i,\cdots [x_i}_n,x_j]_q]_q]_q)
    &=\underbrace{[x_i,[x_i,\cdots [x_i}_n,x_j]_q]_q]_q\otimes 1\\
    &+\sum_{k=0}^n \binom{n}{ k}_q
    \prod_{l=k}^{n-1}(1-q_{ii}\,q_{ij}q_{ji})x_1^{n-k}\otimes \underbrace{[x_i,[x_i,\cdots [x_i}_k,x_j]_q]_q]_q
\end{align*}
and nonzero for smaller $n$ iff 
$$1+q_{ii}+\cdots+q_{ii}^{n-1}=0 \text{ or } q_{ii}^{n-1}q_{ij}q_{ji}=1$$
 
\end{exercise}

\chapter{Root systems}\label{chp_rootsystem}

\begin{slogan}
    Nichols algebras are controlled by a slightly generalized root system, and these root systems can still be classified.
\end{slogan}

\section{Root system from Nichols algebras}

We first discuss a root system as combinatorial data attached to a Nichols algebra. We will discuss axiomatization and classification in the subsequent subsection (so technically the order should be reversed). We finally discuss  what root systems tell us about the structure of the Nichols algebra, for example a PBW basis, but only much later in Section \ref{sec_parabolic} we can explain how Nichols algebra and its reflection are cetegoricall related and sketch how the PBW theorem is proven. The material in this section is \cite{HS20} Chapter 13 and 15 for categories of Yetter-Drinfeld modules over Hopf algebras.

\bigskip

Let $\cC$ be a braided monoidal category, and let $X$ be an object that decomposes into simple summands
$$X=X_1\oplus X_2\oplus \cdots \oplus X_n$$
where we call $n$ the \emph{rank}. For general Hopf algebras $H$ in braided monoidal categories there is a notion of adjoint action of $H$ on itself, see \cite{HS20} Chapter 3.7. The adjoint action of a primitive element $x$ coincides with the braided commutator $\ad_x(y)=[x,y]_c$.

In the Nichols algebra $\NicholsOf(X)$ we can compute repeatedly the adjoint action of $X_i$ on $X_j$ and look at which point it terminates. Recall that for Lie algebras we called these root strings:
\begin{definition}[\cite{HS20} Definition 13.4.2]
The Cartan matrix is defined by  
$$\cm_{ij}:= -\max\{n\in\N \mid \ad(X_i)^n(X_j)\neq 0\},\quad
\cm_{ii}:=2
$$
Of course it is possible the maximum is never attained. If the maximum is attained, then  $X$ is called $i$-finite, we will assume this in the following. 
\end{definition}
For a diagonally braided vector space we have an explicit formula from Rosso's lemma \ref{lm_Rosso}:
\begin{lemma}\label{lm:CartanDiagonal}
For $X,(q_{ij})$ a diagonally braided vector space the Cartan matrix is 
$$\cm_{ij}:=-\max\{m\mid 1+q_{ii}+\cdots+q_{ii}^{m-1}\neq 0 \text{ and } q_{ii}^{m-1}q_{ij}q_{ji} \neq 1\}$$
\end{lemma}
Differently spoken, $m=-\cm_{ij}$ is the smallest number such that  
$$[m+1]_{q_ii}=0\qquad \text{ or } \qquad q_{ii}^{m}q_{ij}q_{ji}=1 $$
$$\text{\bf Truncation case} \hspace{2cm}\text{\bf Cartan case}$$
Note that these conditions depend only on the $q$-diagram. Of course it may happen that both cases apply at the same time, e.g. $u_q^+(\sl_3)$ at $q^4=1$.

Having a formal definition of a Cartan matrix we get a formal definition of a reflection of $X$ by going to the respective last term of the root string. The Nichols algebras are related in a subtle way, as we will discuss only later.

\begin{definition}[\cite{HS20} Section  13.4]\label{def_reflection}
We define the reflection of a
$$R_i(X):= \ad^{-\cm_{i1}}(X_i)(X_1)\oplus \cdots \oplus
X_k^*\oplus \ad^{-\cm_{in}}(X_i)(X_n)$$
Associated is an additive involution on the formal space $\Z^n$ with basis  $\alpha_1,\ldots,\alpha_n$  given by 
$$s_k:\alpha_j\mapsto \begin{cases}
    \alpha_j-\cm_{kj}\alpha_k,&\text{ for } j=k \\
    -\alpha_k,&\text{ for } j=k
\end{cases}$$
At this level, the \emph{root system} $\Phi\subset \Z^n$ in a particular basis of simple roots (Weyl chamber) is the collection of all preimages of the basis vectors under repeated reflections.
\end{definition}
\begin{problem}
An important assertion \cite{HS20} Corollary 13.4.3 is in their setup that the summands after reflection are irreducible. This proof should be worked out in a general categorical setting with duals. 
\end{problem}

\begin{example}[Diagonal case]
Assume the category is $\Vect_{\Z^n}^{\chi,1}$ and 
$$X=\K_{\alpha_1}\oplus \cdots  \K_{\alpha_n}$$
Then after reflection we have literally
$$R_k(X)=\K_{s_k(\alpha_1)}\oplus \cdots  \oplus \K_{s_k(\alpha_n)}$$
The new braiding 
$$R_k(q)_{ij}=\chi(s_k(\alpha_i),s_k(\alpha_j))$$
can be deduced from $(q_{ij})$ multiplicatively extended to a bicharacter $\chi$ with $\chi(\alpha_i,\alpha_j)=q_{ij}$.
Explicit general formulas for $(R_k(q)_{ij})$ in terms of $(q_{ij})$ are given in \cite{HS20} Lemma 15.1.8. 
\end{example}

In general, $R_i(X)$ may be very different from $X$. In particular it may have a different $q$-diagram and even different Cartan matrix. This is an effect known already from Lie super algebras. We would expect that even the treatment of Lie super algebra would be much smoother using the reflection theory discussed below, and in fact generalized root systems in our sense were already anticipated in these examples in \cite{Ser96} .  

\begin{example}
Assume that $\cm_{12}=-1$. Then reflection produces a new braided vector space $R_k(X)$ with basis $x_1^*,x_{12}$ and braiding matrix 
$$(R_1(q)_{ij})
=\begin{pmatrix}
    \chi(-\alpha_1,-\alpha_1) 
    &\chi(-\alpha_1,\alpha_{12}) \\
    \chi(\alpha_{12},-\alpha_1) 
    & \chi(\alpha_{12},\alpha_{12})
\end{pmatrix}
=\begin{pmatrix}
    q_{11} & q_{11}^{-1}q_{12}^{-1}\\
    q_{11}^{-1}q_{21}^{-1} & q_{11}q_{12}q_{21}q_{22}
\end{pmatrix}$$
Note that the information required to compute the new $q$-diagram is only the old $q$-diagram.

\bigskip

There are two conditions for  $\cm_{12}=-1$ in Rosso's Lemma. In both cases we compute further:
\begin{align*}
\text{Cartan case:}\quad q_{11}q_{12}q_{21}=1,\qquad 
(R_1(q)_{ij})&=\begin{pmatrix}
    q_{11} & q_{21}\\
    q_{12} & q_{22}
\end{pmatrix}
\intertext{So the braiding can change, but the $q$-diagram and hence the Cartan matrix does not change.}
\text{Truncation case:}\quad q_{11}=- 1,\qquad
(R_1(q)_{ij})&=
\begin{pmatrix}
    -1 & - q_{12}^{-1}\\
    - q_{21}^{-1} & - q_{12}q_{21}q_{22}
\end{pmatrix}
\intertext{So in this case the $q$-diagram and potentially the Cartan matrix can change. This is different in the other case:}
\end{align*}
\end{example}
The first case is part of a general occurrence
\begin{lemma}
A reflection $R_i$ where $c_{ij}$ is the Cartan case in Rosso's Lemma for all $j$ does not change the $q$-diagram. 
\end{lemma}
\begin{proof}
Suppose $q_{ii}^{m}q_{ij}q_{ji}=1$
\begin{align*}
R_i(q)_{ij}R_i(q)_{ji}
&=\chi(-\alpha_i,m\alpha_i+\alpha_j)\chi(-\alpha_i,m\alpha_i+\alpha_j)
\hspace{2.5cm}\\
&=q_{ii}^{-m}q_{ij}^{-1}\cdot q_{ii}^{-m}q_{ji}^{-1}0\\
&=q_{ij}q_{ji}
\end{align*}
Suppose now  $q_{ii}^{m'}q_{ij'}q_{j'i}=1$ and  $q_{ii}^{m''}q_{ij}q_{ji}=1$, then 
\begin{align*}
R_i(q)_{j'j''}R_i(q)_{j''j'}
&=\chi(m'\alpha_i+\alpha_{j'},m''\alpha_i+\alpha_{j''})\chi(m''\alpha_i+\alpha_{j''},m'\alpha_i+\alpha_{j'})\\
&=q_{ii}^{m'm''}q_{ij'}^{m'}q_{j''i}^{m''}q_{j',j''}\cdot
q_{ii}^{m''m'}q_{ij''}^{m''}q_{j'i}^{m'}q_{j'',j'}\\
&=q_{j',j''}q_{j'',j'} \qedhere
\end{align*}
\end{proof}
On the other hand, for the truncation case with $c_{12}=-m$ the new braiding matrix is 
$$(R_1(q)_{ij})
=\begin{pmatrix}
    \chi(-\alpha_1,-\alpha_1) 
    &\chi(-\alpha_1,m\alpha_{1}+\alpha_{2}) \\
    \chi(m\alpha_{1}+\alpha_2,-\alpha_1) 
    & \chi(m\alpha_{1}+\alpha_{2},m\alpha_{1}+\alpha_{2})
\end{pmatrix}
=\begin{pmatrix}
    q_{11} & q_{11}^{-m}q_{12}^{-1}\\
    q_{11}^{-m}q_{21}^{-1} & q_{11}^{(m^2)}(q_{12}q_{21})^m q_{22}
\end{pmatrix}$$

We give now two explicit examples illustrating the two cases for $c_{12}=-1$:

\begin{example}[$\sl_3$]\label{ex:sl3}
Consider a rank $2$ diagonally braided vector space with some braiding matrix~$(q_{ij})$ with the following $q$-diagram 
\begin{center}
\begin{tikzpicture}
		\draw (0,0)--(1.6,0);
		\draw (-0.1,0) circle[radius=0.1cm] node[anchor=south]{$ q^2$};
        \draw (0.8,0) node[anchor=south]{$ q^{-2}$};
		\draw (1.7,0) circle[radius=0.1cm] node[anchor=south]{$q^2$};
\end{tikzpicture}
\end{center}
for some complex number $q\neq 0$. The Cartan matrix can be computed from Rosso's Lemma \ref{lm:CartanDiagonal}, Cartan case:
$$\cm_{ij}=\begin{pmatrix}
    2 & -1 \\
    -1 & 2
\end{pmatrix}$$
For both reflections the $q$-diagram does not change. For the usual choice $q_{12}=q_{21}=q^{-1}$ related to $\sl_3$ in  Example \ref{exm_quantumgroup} not even the braiding matrix $(q_{ij})$ changes. The root system is as expected from $\sl_3$ of type 
$$\Phi^+=\{\alpha_1,\alpha_2,\alpha_{12}\}$$
with $\alpha_{12}=\alpha_1+\alpha_2$. We continue to call this root system $A_2$.
\end{example}

The following very similar example is related to the Lie superalgebra $\sl(2|1)$:

\begin{example}[$\sl(2|1)$]\label{ex:superSl2}
Consider a rank $2$ diagonally braided vector space with some braiding matrix $(q_{ij})$ with $q$-diagram 
\begin{center}
\begin{tikzpicture}
		\draw (0,0)--(1.6,0);
		\draw (-0.1,0) circle[radius=0.1cm] node[anchor=south]{$ q^2$};
        \draw (0.8,0) node[anchor=south]{$ q^{-2}$};
		\draw (1.7,0) circle[radius=0.1cm] node[anchor=south]{$ -1$};
\end{tikzpicture}
\end{center}
for some complex number $q\neq 0$. The Cartan matrix is as in the previous example 
$$\cm_{ij}=\begin{pmatrix}
    2 & -1 \\
    -1 & 2
\end{pmatrix}$$
where the first entry $-1$ comes from the truncation case in Rosso's lemma  and the second entry $-1$ comes from the first case due to $q_{22}=-1$ (fermionic). 

Reflection on the first summand has been already described in the previous example, the $q$-diagram does not change. Reflection on the second summand however gives by the previous formula

$$(R_2(q_{ij}))=\begin{pmatrix}
    -1 & -q_{12}^{-1}\\
    -q_{21}^{-1} & -1
\end{pmatrix}
$$
and a new $q$-diagram
\begin{center}
\begin{tikzpicture}
		\draw (0,0)--(1.6,0);
		\draw (-0.1,0) circle[radius=0.1cm] node[anchor=south]{$ -1$};
        \draw (0.8,0) node[anchor=south]{$ q^2$};
		\draw (1.7,0) circle[radius=0.1cm] node[anchor=south]{$ -1$};
\end{tikzpicture}
\end{center}
It has again the same Cartan matrix, but now both times the second case in Lemma \ref{lm:CartanDiagonal}. The root system is hence again $A_2$ 
$$\Phi^+=\{\alpha_1,\alpha_2,\alpha_{12}\}$$
We now more explicitly give all reflections, the calculations above are sufficient for this. Note that the lines now denote the hyperplanes orthogonal to the roots in $A_2$ and the diagrams are drawn into the Weyl chambers, which correspond to choices of simple roots given by the adjacient hyperplanes.

 %Parameters #1,#2 define coordinates of the edge center
 %Parameters #3,#4,#5 define q-text on nodes and edge
 %Parameters #6,#7 define text below nodes
\newcommand{\qdiag}[7]{
        \draw (#1-0.8,#2)--(#1+0.8,#2);
		\draw (#1-0.8-0.1,#2) circle[radius=0.1cm] node[above=.1cm]{$#3$} node[below=.1cm]{$#6$};
        \draw (#1,#2) node[anchor=south]{$#4$};
		\draw (#1+0.8+0.1,#2) circle[radius=0.1cm] node[above=.1cm]{$#5$}  node[below=.1cm]{$#7$};
}
\begin{center}
\begin{tikzpicture}[scale=.75]
%\sqrt{3}/2=0.866
\renewcommand{\unit}{6}
\qdiag{0.5*\unit}{0.866*\unit}{q^2}{q^{-2}}{-1}{\alpha_1}{\alpha_2}
\qdiag{-0.5*\unit}{0.866*\unit}{q^2}{q^{-2}}{-1}{-\alpha_{1}}{\alpha_{12}}
\qdiag{-1*\unit}{0*\unit}{-1}{q^{2}}{-1}{\alpha_{2}}{-\alpha_{12}}
\qdiag{-0.5*\unit}{-0.866*\unit}{-1}{q^{-2}}{q^2}{-\alpha_{2}}{-\alpha_{1}}
\qdiag{0.5*\unit}{-0.866*\unit}{-1}{q^{-2}}{q^2}{-\alpha_{12}}{\alpha_{1}}
\qdiag{1*\unit}{0*\unit}{-1}{q^{-2}}{q^2}{\alpha_{12}}{-\alpha_{2}}
\newcommand{\bunit}{7.5}
\newcommand{\cunit}{8}
\draw (0*\bunit,1*\bunit) -- (-0*\bunit,-1*\bunit);
\node at (0*\cunit,1*\cunit){$\alpha_1^\perp$};
\draw (0.866*\bunit,0.5*\bunit) -- (-0.866*\bunit,-0.5*\bunit);
\node at (0.866*\cunit,0.5*\cunit){$\alpha_2^\perp$};
\draw (-0.866*\bunit,0.5*\bunit) -- (0.866*\bunit,-0.5*\bunit);
\node at (-0.866*\cunit,0.5*\cunit){$\alpha_{12}^\perp$};
\newcommand{\dunit}{\unit}
\renewcommand{\angle}{9}
\newcommand{\reflection}[2]{
\draw[<->] ({\dunit*cos(#1-\angle)},{\dunit*sin(#1-\angle)}) 
--  ({\dunit*cos(#1+\angle)},{\dunit*sin(#1+\angle)})
        node[fill=white,midway] {\tiny #2};
}
\reflection{30}{$R_2$}
\reflection{90}{$R_1$}
\reflection{150}{$R_2$}
\reflection{210}{$R_1$}
\reflection{270}{$R_2$}
\reflection{330}{$R_1$}
\end{tikzpicture}
\end{center}
We can understand this further by marking each root with its self-braiding
$$\chi(\alpha_1,\alpha_1)=q^2,\qquad 
\chi(\alpha_2,\alpha_2)=\chi(\alpha_{12},\alpha_{12})=-1
$$
which also practically shows which case of Rosso's lemma applies for reflection on this root.

All Weyl chambers have in this case the same Cartan matrix (or it would not be the root system $A_2$), but in view of the $q$-diagram there are two different types of Weyl chamber: Those with two fermionic simple roots, and those with one fermionic and one $q^2$ simple root.
\end{example}

The following example shows that even the Cartan matrix can change. It is related to the Lie superalgebras and $D(2|1,\alpha)$:

\begin{example}[$D(2|1,\alpha)$]\label{ex:D21}
Consider the diagonally braided vector space $X=x_1\C+x_2\C+x_3\C$ for some braiding matrix $(q_{ij})$ that belongs to the following $q$-diagram
\[
		\begin{tikzpicture}
		\draw (0,0)--(-0.8,-1.4)--(0.8,-1.4)--(0,0);
		\draw (0,0.1) circle[radius=0.1cm] node[anchor=south]{$ -1$}
		(-0.9,-1.4) circle[radius=0.1cm] node[anchor=east]{$ -1$}
		(0.9,-1.4) circle[radius=0.1cm] node[anchor=west]{$ -1$};
		\draw (-0.6,-0.7) node[anchor=east]{$ a$};
		\draw (0.6,-0.7) node[anchor=west]{$ b$};
		\draw (0,-1.4) node[anchor=north]{$ c$};
		\end{tikzpicture}
\]
for $abc=1$. Each edge is a copy of the previous example. The Cartan matrix is 
\[\big(\cm_{ij}^{\mathrm{I}}\big)= \begin{bmatrix}
\hphantom{-}2 & -1 & -1 \\
-1 & \hphantom{-}2 & -1 \\
-1 & -1 & 2
\end{bmatrix}.
\]
We compute the reflection $R_2(X)$ spanned by $x_{12},x_2^*,x_{23}$. All entries of the  $q$-diagram adjacient to $\alpha_2$ can be read off the previous example, except the following, which we now  compute
\begin{align*}
R_2(q)_{13}R_2(q)_{31}
&=
\chi(\alpha_1+\alpha_2,\alpha_2+\alpha_3)
\chi(\alpha_2+\alpha_3,\alpha_1+\alpha_2)\\
&=(q_{12}q_{13}q_{22}q_{23})(q_{21}q_{22}q_{31}q_{32}) \\
&=(q_{22})^2(q_{12}q_{21})(q_{13}q_{31})(q_{23}q_{32}) \\
&=(-1)^2 abc=1
\end{align*}

\begin{center}
\begin{tikzpicture}
	\draw (0,0)--(1.6,0) (1.8,0)--(3.4,0);
	\draw (-0.1,0) circle[radius=0.1cm] node[anchor=south]{$ a$}
	(1.7,0) circle[radius=0.1cm] node[anchor=south]{$ -1$}
	(3.5,0) circle[radius=0.1cm] node[anchor=south]{$ b$};
	\draw (0.8,0) node[anchor=south]{$ a^{-1}$};
	\draw (2.6,0) node[anchor=south]{$ b^{-1}$};
	\end{tikzpicture}
\end{center}
Note that this is where the condition $abc=1$ enters crucially. The new Cartan matrix is hence
\[\big(\cm_{ij}^{\mathrm{II}}\big)= \begin{bmatrix}
\hphantom{-}2 & -1 & \hphantom{-}0 \\
-1 & \hphantom{-}2 & -1 \\
\hphantom{-}0 & -1 & \hphantom{-}2
\end{bmatrix}.
\]
Even though this particular Cartan matrix is of type $A_3$, the root system seems to be different: For the first $q$-diagram there is clearly a nontrivial element $x_{13}$ (an explicit braided commutator) in degree $\alpha_1+\alpha_3$. In different words, there is a root $\alpha_1+\alpha_3$ as preimage $s_1^{-1}(\alpha_3)=s_3^{-2}(\alpha_1)$. After reflection on $\alpha_2$ this is in the new basis
$$\alpha_1'+2\alpha_2'+\alpha_3'
=(\alpha_1+\alpha_2)+2(-\alpha_{2})
+(\alpha_2+\alpha_3)
=\alpha_1+\alpha_3
$$
and this is surely not a root in the usual $A_2$ root system. 
Intuitively (we make this precise below) this means that for the second $q$-diagram we expect an nonzero iterated commutator $[x_1',[x_3'[x_1',x_2']]]]$.

We now compute all reflections and arrange them for later use in the following picture. Each line is reflection on a hyperplane $\alpha^\perp$, as in the previous rank $2$ example. Note that where lines cross we see what is generated by  two (of the three) reflection, which only depends on the $q$-subdiagrams we have for this subset of simple roots and in particular the $2\times 2$ Cartan matrix between them. In the present example the only rank $2$ subdiagrams where appearing previously and have Cartan matrices and subsequently root systems of type  $A_1\times A_1$ or $A_2$.
%shift 
%text on left node, left edge, top node, right edge, right node, bottom edge
\newcommand{\qdiagD}[7]{
        \begin{scope}[shift=#1, scale=.9]
        \newcommand{\scale}{.85}
        \draw (210:1.1*\scale) circle[radius=0.1cm];    
        \draw (210:1.7*\scale) node{$#2$};
        \if\relax\detokenize{#3}\relax
            \else
            \draw (210:1.0*\scale)--(90:1.0*\scale);
            \draw (150:1.1*\scale) node{$#3$};
            \fi
        \draw (90:1.1*\scale) circle[radius=0.1cm];    
        \draw (90:1.7*\scale) node{$#4$};
        \if\relax\detokenize{#5}\relax
            \else
            \draw (90:1.0*\scale)--(-30:1.0*\scale);
            \draw (30:1.1*\scale) node{$#5$};
            \fi
        \draw (-30:1.1*\scale) circle[radius=0.1cm];    
        \draw (-30:1.7*\scale) node{$#6$};
        \if\relax\detokenize{#7}\relax
            \else
            \draw (-30:1.0*\scale)--(-150:1.0*\scale);
            \draw (-90:1.1*\scale) node{$#7$};
            \fi
        \end{scope}
}
\begin{center}
\begin{tikzpicture}[scale=.7]
\qdiagD{(0:0)}{-1}{a}{-1}{b}{-1}{c}
    \qdiagD{(150:7)}{-1}{a}{-1}{b}{-1}{c}
    \qdiagD{(270:7)}{-1}{a}{-1}{b}{-1}{c}
    \qdiagD{(30:7)}{-1}{a}{-1}{b}{-1}{c}
\qdiagD{(90:3.5)}{a}{a^{-1}}{-1}{b^{-1}}{b}{}
    \qdiagD{(90-30:6.5)}{a}{a^{-1}}{-1}{b^{-1}}{b}{}
    \qdiagD{(90+30:6.5)}{a}{a^{-1}}{-1}{b^{-1}}{b}{}
    \qdiagD{(90:8)}{a}{a^{-1}}{-1}{b^{-1}}{b}{}
\qdiagD{(210:3.5)}{-1}{a^{-1}}{a}{}{c}{c^{-1}}
    \qdiagD{(210-30:6.5)}{-1}{a^{-1}}{a}{}{c}{c^{-1}}
    \qdiagD{(210+30:6.5)}{-1}{a^{-1}}{a}{}{c}{c^{-1}}
    \qdiagD{(210:8)}{-1}{a^{-1}}{a}{}{c}{c^{-1}}
\qdiagD{(-30:3.5)}{c}{}{b}{b^{-1}}{-1}{c^{-1}}
    \qdiagD{(-30+30:6.5)}{c}{}{b}{b^{-1}}{-1}{c^{-1}}
    \qdiagD{(-30-30:6.5)}{c}{}{b}{b^{-1}}{-1}{c^{-1}}
    \qdiagD{(-30:8)}{c}{}{b}{b^{-1}}{-1}{c^{-1}}
%streight hyperplanes
\newcommand{\rad}{11.7}
\newcommand{\shift}{1.85}
\begin{scope}[shift=(210:\shift)]
\draw(120:\rad)--(300:\rad);
\draw (.2,-.2) node[fill=white] {\tiny $\alpha_1$};
\end{scope}
\begin{scope}[shift=(90:\shift)]
\draw(180:\rad)--(0:\rad);
\draw (0,0) node[fill=white] {\tiny $\alpha_2$};
\end{scope}
\begin{scope}[shift=(330:\shift)]
\draw(60:\rad)--(240:\rad);
\draw (0.0,-0.2) node[fill=white] {\tiny $\alpha_3$};
\end{scope}

%arc hyperplanes
\newcommand{\arcradius}{11.4}
\newcommand{\archalfangle}{50.60}
\begin{scope}[shift=(330:\arcradius)]
 %startpoint=arcstart
\draw (150-\archalfangle:\arcradius+2*\shift) arc (150-\archalfangle:150+\archalfangle:\arcradius+2*\shift);
\end{scope}
\begin{scope}[shift=(210:\arcradius)]
 %startpoint=arcstart
\draw (30-\archalfangle:\arcradius+2*\shift) arc (30-\archalfangle:30+\archalfangle:\arcradius+2*\shift);
\end{scope}
\begin{scope}[shift=(90:\arcradius)]
 %startpoint=arcstart
\draw (270-\archalfangle:\arcradius+2*\shift) arc (270-\archalfangle:270+\archalfangle:\arcradius+2*\shift);
\end{scope}

\draw (-1.9,4.0) node[fill=white] {\tiny $\alpha_{12}$};
\draw (1.9,4.0) node[fill=white] {\tiny $\alpha_{23}$};
\draw (7.5,-1.55) node[fill=white] {\tiny $\alpha_{13}$};

%final circle
\draw (0,0) circle (\rad+.125);
\draw (0,\rad+.125) node[fill=white] {\tiny $\alpha_{123}$};
\end{tikzpicture}
\end{center}

Note that reflection on the outmost circle maps to the antipodal side of the picture. Hence it turns out that the overall root system has seven positive roots. If $\lbrace\alpha_1, \alpha_2, \alpha_3\rbrace$ are the simple roots in the first $q$-diagram, then the positive roots in this basis are \[\{ \alpha_1, \alpha_2,  \alpha_3,  \alpha_{12}, \alpha_{23},  \alpha_{13},  \alpha_{123} \}\]
After reflection on $\alpha_2$, these roots are mapped to 
\[\{ \alpha_{12},  -\alpha_2,  \alpha_{23},  \alpha_{1},  \alpha_{3},  \alpha_{123},  \alpha_{13} \}\]
In the new basis this is
\[\{ \alpha_1', \alpha_2',  \alpha_3',  \alpha_{12}', \alpha_{23}',\alpha_{123}',  \alpha_{1223}' \}\]
So while they are ''globally'' the same hyperplanes, they are a different set of seven coordinate tuples in those different bases. 

Geometrically, this is a hyperplane arrangement related to the cuboctahedron, a platonic solid. We draw it as an affine resp.~projective picture, which is the picture we already drew above:
\begin{center}
	\includegraphics[width=0.49\textwidth]{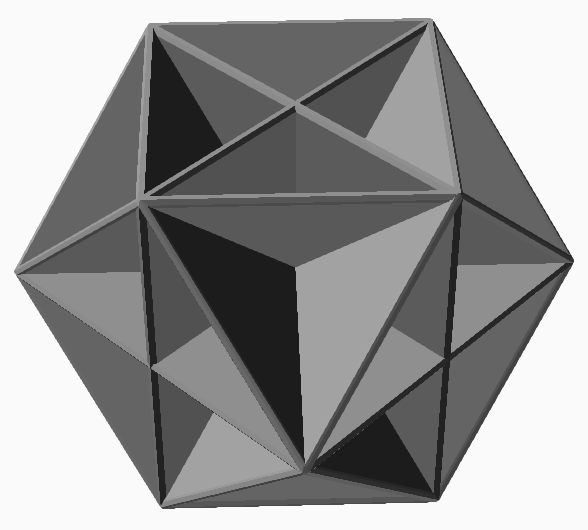}
	\hspace{1cm}
	\includegraphics[width=0.42\textwidth]{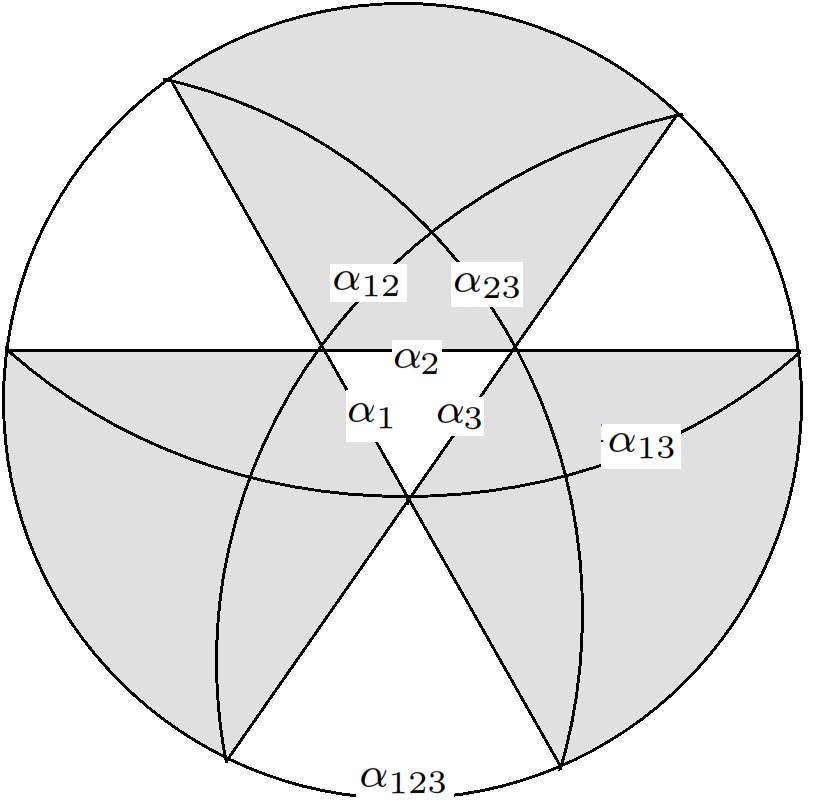}
\end{center}
There are $7$ hyperplanes through the origin (resp. projective lines) orthogonal to the $7$ roots. There are two types of roots according to the self-braiding (and hence practically the case of Rosso's lemma they lead to): $$\chi(\alpha_1,\alpha_1)=\chi(\alpha_2,\alpha_2)=\chi(\alpha_3,\alpha_3)=\chi(\alpha_{123},\alpha_{123})
=-1$$
$$\chi(\alpha_{12},\alpha_{12})
=a,\;\chi(\alpha_{23},\alpha_{23})=b,\;\chi(\alpha_{13},\alpha_{13})=c
$$
There are $16$ triangles corrresponding to Weyl chambers corresponding to $16$ bases of simple roots given by the three adjacient hyperplanes (with orientation sign). There are two types of triangles: Equilateral triangles (white) correspond to the Cartan matrix I and right triangles (grey) to the Cartan matrix~II.
\end{example}

\begin{exercise}
In the example above see if you could reproduce the picture with the different $q$-diagrams by applying reflections and write in some cases the basis of simple roots below the diagrams as in the rank $2$ example. Give a chain of reflections that make the root $\alpha_{123}$ visible.
\end{exercise}
\begin{exercise}
Draw the same picture for $\sl_3$ with root system $A_3$ (see preliminaries for a picture of the hyperplane arrangement). 
\end{exercise}

Heckenberger has classified in \cite{Heck09} all diagonally braided vector spaces over a field of characteristic zero, such that the root system attached to it is finite. In doing so, he has defined the reflections, arithmetic root systems etc. and aas we will later see, this is a necessary condition  for the Nichols algebra to be finite-dimensional. They are a subset of the classification of all finite root systems, which we discuss later. We present his list for rank $2$, see \cite{HS20} Section 15.2 for a revised proof.

\begin{center}
\label{fig_HeckenbergerRank2}
\includegraphics[scale=1.3]{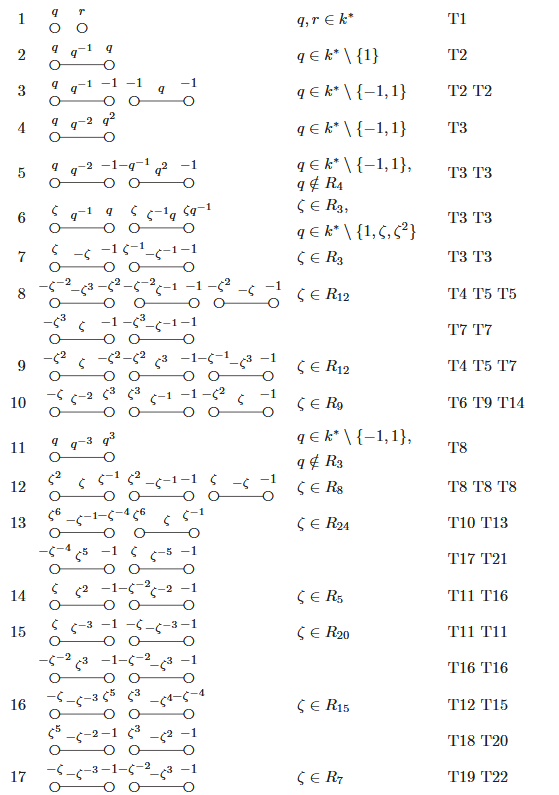}
\end{center}

Rows $1,2,4,12$ correspond to Lie algebras $A_1\times A_1,A_2,B_2,G_2$ and rows $3,5$ correspond to Lie superalgebras $A(1,0),B(1,1)$. The last column denotes binary trees describing the different Weyl chambers in Section \ref{sec_rank2}.   

\bigskip

Over fields of finite characteristic, the finite dimensional Nichols algebras are classified in \cites{HW15, Wa17, LYQW24}. The procedure is instructive: In the first paper all rank $2$ Nichols algebras are computed by hand, using new methods, in the second paper the rank $3$ Nichols algebras are computed using a-priori knowledge of the possible Weyl groupoids, the higher rank cases are then sorted out.  

\section{Axiomatization of generalized root systems}

Before giving the definitions of Weyl groupoids and root systems in \cite{HS20} Chapter 9, we give the corresponding geometric definition \cite{Cu11}, developed in hindsight:

\subsection{Hyperplane arrangements}

As discussed in the preliminaries, a root system in the usual sense is an arrangement of hyperplanes in some Euclidean vector space $\mathbb{R}^r$ and a choice of normal vectors thereof, called roots, which is stable under the reflections on these hyperplanes and fulfills an integrality condition. The connected components of the complement of all hyperplanes are called Weyl chambers, and the normal vectors of the walls of any fixed Weyl chamber give a basis of the vector space, a set of simple roots. The reflections act transitively on the Weyl chambers and every roots has integer coefficients with respect to every fixed set of simple roots.

In a generalized setting, called a \emph{crystallographic arrangement}, we drop the euclidean product of the ambient vector space (so the roots are a choice of linear functions), but we keep demanding that every root has integer coefficients with respect to every set of simple roots.

\begin{definition}[Crystallographic hyperplane arrangement, \cite{Cu11} Definition 2.3]
Let $V=\R^n$ and take a subset $\Phi\subset V^*$ called set of roots. Corresponding are the hyperplanes $H_\alpha=\ker(\alpha)$ for any $\pm \alpha\in V^*$ with $\alpha$ a chosen normal vector. Assume
\begin{itemize}
    \item Simplicial: $V\backslash \bigcup_{\pm \alpha\in R} H_\alpha$ is a union of simplicial cones, i.e. bounded by $n$ of the hyperplanes. We call the simplices $a$ the (Weyl-)chambers and the set of all chambers $A$\footnote{In the cited source, chambers are denoted $K\in\mathcal{K}$.}. For each chamber we denote the normal vectors of the $n$ adjacient hyperplanes $\alpha_1^a,\ldots,\alpha_n^a$. These form a basis of $V$.
    \item Reduced: For any root $\alpha$ the only multiples in $R$ are $\pm \alpha$. This is just for simplicity.
    \item Crystallographic: $\Phi\subset \alpha_1^a\Z\oplus\cdots \alpha_n^a\Z$ for any $a\in A$. 
\end{itemize}
\end{definition}
To a crystallographic arrangement we can attach the following 
\begin{itemize}
    \item We define an analog of a Cartan matrix $\cm_{ij}^a$ for every chamber $a\in A$
    $$\cm_{ii}=2,\qquad \cm_{ij}=-\max_m\{m\alpha_i+\alpha_j\in \Phi\}\;i\neq j$$
    \item We define reflections for each $a\in A$ and $i=1,\ldots,n$
    $$s_i:\alpha_i^a\mapsto \begin{cases}
    \alpha_j^a-\cm_{ij}\alpha_i^a,&\text{ for } i=j \\
    -\alpha_i^a,&\text{ for } i=j
    \end{cases}$$
    \SimonRem{revise}
    These are linear involutions of $V$. Moreover they preserve the sets of hyperplanes and map $\Phi$ to itself.
    \item In particular this produces a bijection $\rho_i:A\to A$. 
\end{itemize}

\subsection{Cartan graph and Weyl groupoid}

In contrast to usual root systems, the set of all roots, written in different bases of simple roots, may not always give the same set of coordinate tuples. Consequently the reflections generate a Weyl groupoid, whose objects are different types of Weyl chambers, and each object has attached its own Cartan matrix~$\cm_{ij}^a$ depicted by a Dynkin diagram. 

Such data has been axiomatized and classified under the name \emph{Cartan graph} \cites{AHS10, Heck06,  HY08}, which we discuss next.
Actually, this behaviour is already familiar from Lie superalgebras, where different sets of simple roots contain different parities, and in extreme cases such as $D(2,1;\alpha)$ even different Dynkin diagrams, see Example~\ref{ex:D21}.

In the following account we follow \cite{HS20} Chapters~9,~10 and~15, where more details can be found:

\newcommand{\ZZ}{\Z}
\newcommand{\Cm}{\mathsf{C}}
\renewcommand{\cm}{\mathsf{c}}
\newcommand{\rfl}{\rho}
\newcommand{\s}{s}
\renewcommand{\r}{\mathrm{r}}
\newcommand{\Cc}{\mathcal{C}}
\newcommand{\Wg}{\mathcal{W}}
\newcommand{\re }{^\mathrm{re}}
\newcommand{\rer }[1]{(R\re)^{#1}}
\newcommand{\rsC }{\mathcal{R}}
\begin{definition}[generalized Cartan matrix]
	Let $I:=\{1,\dots,r\}$, where $r$ is called {\it rank}, and
	$\{\alpha_i\,|\,i\in I\}$ the standard basis of $\ZZ ^I$.
	A {\it generalized Cartan matrix}
	$\Cm =(\cm _{ij})_{i,j\in I}$
	is a matrix in $\ZZ ^{I\times I}$ such that
	\begin{enumerate}\itemsep=0pt
		\item[(M1)] $\cm _{ii}=2$ and $\cm _{jk}\le 0$ for all $i,j,k\in I$ with
		$j\not=k$,
		\item[(M2)] if $i,j\in I$ and $\cm _{ij}=0$, then $\cm _{ji}=0$.
	\end{enumerate}
    We could equivalently encode this data in a \emph{Dynkin diagram},  each edge decorated by $(c_{ij},c_{ji})$.
\end{definition}
In the following definition we think on $A$ as a set of Weyl chambers.
\begin{definition}[Cartan graph]
	Let $A$ be a non-empty set, $\rfl _i\colon A \to A$ a map for all $i\in I$,
	and $\Cm ^a=(\cm ^a_{jk})_{j,k \in I}$ a generalized Cartan matrix
	in $\ZZ ^{I \times I}$ for all $a\in A$. The quadruple
	\[ \Cc = \Cc \big(I,A,(\rfl _i)_{i \in I}, (\Cm ^a)_{a \in A}\big)\]
	is called a \textit{Cartan graph} if
	\begin{enumerate}\itemsep=0pt
		\item[(C1)] $\rfl _i^2 = \mathrm{id}$ for all $i \in I$,
		\item[(C2)] $\cm ^a_{ij} = \cm ^{\rfl _i(a)}_{ij}$ for all $a\in A$ and
		$i,j\in I$.
	\end{enumerate}
\end{definition}

\begin{definition}[Weyl groupoid]
	Let $\Cc = \Cc \big(I,A,(\rfl _i)_{i \in I}, (\Cm ^a)_{a \in A}\big)$ be a
	Cartan graph. For all $i \in I$ and $a \in A$ define $\s _i^a \in
	\operatorname{Aut}(\Z ^I)$ by
	\begin{align*}
	\s _i^a (\alpha_j) = \alpha_j - \cm _{ij}^a \alpha_i \qquad
	\text{for all $j \in I$.}
	%\label{fwg_eq:sia}
	\end{align*}
	The \textit{Weyl groupoid of} $\Cc $
	is the category $\Wg (\Cc )$ such that $\operatorname{Ob} (\Wg (\Cc ))=A$ and
	the morphisms are compositions of maps
	$\s _i^a$ with $i\in I$ and $a\in A$,
	where $\s _i^a$ is considered as an element in $\operatorname{Hom} (a,\rfl _i(a))$.
	The cardinality of $I$ is called the \textit{rank of} $\Wg (\Cc )$.
\end{definition}
A Cartan graph axiomatizes a set of Cartan matrices, one for every Weyl chamber (or every type of Weyl chamber) $a\in A$, and reflections $\s_i^a$ on simple roots~$\alpha_i$ in the Weyl chamber~$a$, which are linear maps between a space $\Z^I$ attached to $a$ and to the Weyl chamber after reflection~$\rho_i(a)$. There is an abstract notion of Coxeter groupoid in this context, see \cite{HS20} Chapter 9.4.

Let $\Cc $ be a Cartan graph. For all $a\in A$ define the set of \emph{real roots} at $a$ by
\[ \rer a=\big\{ \s _{i_1}\cdots \s_{i_k}(\alpha_j)\,|\,
k\in \N _0,\,i_1,\dots,i_k,j\in I\big\}\subseteq \ZZ ^I.\]
A real root $\alpha\in \rer a$ is called positive if $\alpha\in \N _0^I$.

\begin{definition}[root system]
	Let $\Cc =\Cc \big(I,A,(\rfl _i)_{i\in I},(\Cm ^a)_{a\in A}\big)$ be a Cartan
	graph. For all $a\in A$ let $R^a\subseteq \ZZ ^I$, and define
	$m_{i,j}^a= |R^a \cap (\N_0 \alpha_i + \N_0 \alpha_j)|$ for all $i,j\in
	I$ and $a\in A$. We say that
	\[ \rsC = \rsC (\Cc, (R^a)_{a\in A}) \]
	is a \textit{root system of type} $\Cc $, if it satisfies the following
	axioms.
	\begin{enumerate}\itemsep=0pt
		\item[(R1)]
		$R^a=R^a_+\cup - R^a_+$, where $R^a_+=R^a\cap \N_0^I$, for all
		$a\in A$.
		\item[(R2)]
		$R^a\cap \ZZ\alpha_i=\{\alpha_i,-\alpha_i\}$ for all $i\in I$, $a\in A$.
		\item[(R3)]
		$\s _i^a(R^a) = R^{\rfl _i(a)}$ for all $i\in I$, $a\in A$.
		\item[(R4)]
		If $i,j\in I$ and $a\in A$ such that $i\not=j$ and $m_{i,j}^a$
		finite, then
		$(\rfl _i\rfl _j)^{m_{i,j}^a}(a)=a$.
	\end{enumerate}
\end{definition}
Note that (R3) is saying that there is essentially \emph{one} set of roots $\Phi$, written in different ways as tuples in $\Z^I$. Note that (R4) generalizes the geometric statement for reflection groups. 
\begin{lemma}
 Let $\Cc $ be a Cartan graph and $\rsC$ a root system of type $\Cc$. Let $a\in A$. Then 
 \[\cm^a_{ij}=-\max \big\{m\in\N_0\,|\,\alpha_j+m\alpha_i \in R^a_+\big\}.\]
\end{lemma}

%As a convention, we name a positive root $\alpha$ by indicating with which multiplicity each simple root $\alpha_i$ appears in $\alpha$ (in some fixed Weyl chamber), e.g.,
%\[\alpha_{12}=\alpha_1+\alpha_2,\qquad
%\alpha_{123}=\alpha_1+\alpha_2+%\alpha_3,\qquad
%\alpha_{12 2}=\alpha_{112}=2\alpha_1+\alpha_2, \qquad \dots.\]
For roots $\alpha,\beta\in R^a$ we can define a Cartan matrix entry independent of $a$
 \[-\cm_{\alpha,\beta}=\max \{m\in\N_0\,|\,\alpha+m\beta \in R^+\}.\]

The root system $\rsC $ is called \textit{finite} iff for all $a\in A$ the
set $R^a$ is finite. By~\cite{CH15} if $\rsC $ is a finite root system
of type $\Cc $, then $\rsC =\rsC \re $, and hence $\rsC \re $ is a root
system of type $\Cc$ in that case.

\begin{example}[Cartan type]Let $\g$ be a semisimple finite-dimensional complex Lie algebra. It is well-known that is uniquely determined (up to isomorphisms) by its root system, which is the root system of a finite Weyl group $W$. The corresponding Cartan graph has exactly one object~$a$ and~$C^a$ is the Cartan matrix of $W$. The set $R^a$ is the root system of~$W$. Alternatively, we can consider the Cartan graph with one object $a$ for each Weyl chamber and the Cartan matrices attached to all objects are equal.
\end{example}
\begin{problem}
    For infinite root systems there are imaginary roots, similar to affine Lie algebras, and we know only very few about those. See the outlook section in \cite{HS20}.  
\end{problem}

 It is proven in \cite{Cu11} Theorem~1.1 that the crystallographic arrangement mentioned in this section's introduction are in bijection to connected simply connected Cartan graphs for which the set of real roots is finite.

\begin{theorem}[\cite{Cu11} Theorem 5.4]
There is an equivalence between finite root systems and crystallographic arrangements.
\end{theorem}

The finite Weyl groupoids and root systems are classified in \cites{CH15}. 

\begin{theorem}\label{thm_ClassificationArrangments}
The following are all irreducible crystallographic arrangements:
\begin{enumerate}[a)]
\item The rank two cases are parametrized by triangulations of
convex $n$-gons by non-intersecting diagonals.
\item For each rank $r>2$, arrangements of type $A_r$, $B_r$, $C_r$
and $D_r$, and a further series of $r-1$ arrangements (familiar from Lie superalgebras and containing Example \ref{ex:D21}).
\item Further $74$ ``sporadic'' arrangements of rank $r$, $3\le r \le 8$.
\end{enumerate}
\end{theorem}

\begin{remark}
If the root system is not finite, then the geometry can be different. A notable class are the \emph{affine root systems}, where adding Weyl chambers does not close to a spherical picture, but to a tesselation of the full space. It would be very interesting to study Nichols algebras of this type \cites{Cu19}. \SimonRem{Heckenberger student?}
\end{remark}

\section{Classification in Rank 2}\label{sec_rank2}

We describe in more detail the connected root systems of rank $2$, which are the building block for higher rank, see the original sources \cite{CH11} and \cite{HS20} Section 9.1. and 9.2  Note that only finitely many of them appear as root systems of diagonally braided vector spaces of rank $2$ in the table above:

Any Cartan graph with a finite number of objects has necessarily the following form, in view of axiom (C2)
\begin{center}
\hspace*{-1cm}
\begin{tikzcd}
      \cdots a_{2n} \arrow[r,leftrightarrow, "\rho_2"] 
    &  a_1 \arrow[r, leftrightarrow, "\rho_1"] 
    & a_2 \arrow[r, leftrightarrow,"\rho_2"]
    & a_3 \arrow[r, leftrightarrow,"\rho_1"]
    & \cdots \\
    \hphantom{\cm^{a_1}=\begin{psmallmatrix}
    2 & -c_1 \\
    -c_n & 2
    \end{psmallmatrix}}
    &
    \cm^{a_1}=\begin{psmallmatrix}
    2 & -c_1 \\
    -c_n & 2
    \end{psmallmatrix}
    &
    \cm^{a_2}=\begin{psmallmatrix}
    2 & -c_1 \\
    -c_2 & 2
    \end{psmallmatrix}
    &
    \cm^{a_3}=\begin{psmallmatrix}
    2 & -c_3 \\
    -c_2 & 2
    \end{psmallmatrix}
    &
    \hphantom{\cm^{a_1}=\begin{psmallmatrix}
    2 & -c_1 \\
    -c_n & 2
    \end{psmallmatrix}}
\end{tikzcd}
\end{center}
Here $(c_1,\ldots,c_n)$ is some sequence of integers and the Cartan matrices are indexed by chambers $A=\{a_1,\ldots,a_{2n}\}$ and $a_{2n+1}=a_1$ (to avoid parity problems).

\bigskip

We now discuss the technical properties a sequence $(c_1,\ldots,c_n)$ must have to define a Cartan graph: Define $\eta(x)=\begin{psmallmatrix} x & -1 \\ 1 & 0\end{psmallmatrix}$, denote by $\mathcal{A}$ the set of integer sequences $(c_1,\ldots,c_n)$ such that $\eta(c_1)\cdots \eta(c_n)=-\mathrm{id}$, and consider $\mathcal{A}^+$ the subset of those sequences where $c_i\geq 1$ and where the entries on the first column of $\eta(c_1)\cdots \eta(c_i)$ are positive. 

There is an equivalence relation by shift and reflection. It is proven in \cite{CH09} Theorem 5.5 that all sequences in $\mathcal{A}^+$ arise from $(1,1,1)$ by iteratively increasing the length by $1$ using the following move: 
$$I:(\ldots, c_i,c_{i+1},\ldots)
\longmapsto 
(\ldots, c_i+1,1,c_{i+1}+1,\ldots)
$$
Further, it is proven in Theorem 3.4 that these sequences are in bijection with triangulations of an $n$-gon with $c_i$ triangles attached to the vertex $i$, so $(1,1,1)$ corresponds to the triangle itself and the iteration corresponds to attaching a triangle to the edge $(i,i+1)$.

As yet another characterization, these triangulations correspond to  binary trees: Take any edge of the $n$-gon (correspondingly, a chamber) and the first triangle attached to it, put a node on the edge and two branches to nodes on the other two edges. Now proceed iteratively through the adjacent triangles, producing a binary tree. We remark that these trees were the chronologically oldest incarnation of Nichols algebra root systems and are the closest to the actual commutator structure going back to Kharchenko, see \cite{Heck07} Section 5. Figure (A1) shows the trees numbered T1-T22 correpsonding to the Nichols algebra are in Figure A1.

\bigskip

It is proven in \cite{CH09} Theorem 5.5 that for each sequence in $\mathcal{A}^+$ there is a unique root system, and the main result is \SimonRem{where?}
\begin{theorem}
Every irreducible finite root system with is of the form above for a unique sequence $(c_1,\ldots,c_n)\in\mathcal{A}^+$ up to equivalence, where $n$ is the number of positive roots.  
\end{theorem}
The set of all real roots is obtained as follows, called an $\mathcal{F}$-sequence: 
$$\Phi^+(1,1,1)=\{(1,0),(1,1),(0,1)\}$$
$$\Phi^+(I(v_1,\ldots, v_n))
=(v_1,\ldots, v_i,v_i+v_{i+1},v_{i+1},\ldots, v_n)$$
(They are sorted by increasing quotient of the entries.) There are relations to continued fractions and Cluster algebras we do not discuss here. 
\begin{example} For small $n$ there is a unique sequence and hence a unique root system \SimonRem{triangles are roots?}
\begin{itemize}
    \item $n=3$ roots: The sequence $(1,1,1)$ corresponds to the $3$-gon being a single triangle and the following most basic binary tree:
    \begin{center}
    %\begin{minipage}[0.45\textwidth]
    \includegraphics[scale=.7]{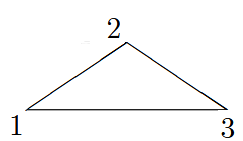}
    %\end{minipage}
    \hspace{3cm}
    %\begin{minipage}[0.45\textwidth]
    \includegraphics[scale=1]{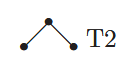}
    %\end{minipage}
    \end{center}
    \SimonRem{Istvan says triangles are roots}
     All Cartan matrices are equal, this is the  root system $A_2$. The set of positive roots is
    $$\big\{(1,0),(1,1),(0,1)\big\}$$
    
    \item $n=4$ roots: The sequence $(1,2,1,2)$ or equivalently $(2,1,2,1)$
    corresponds to the triangulation and binary tree
    \begin{center}
    %\begin{minipage}[0.45\textwidth]
    \includegraphics[scale=.7]{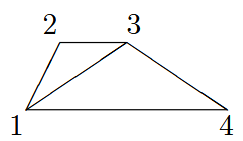}
    %\end{minipage}
    \hspace{3cm}
    %\begin{minipage}[0.45\textwidth]
    \includegraphics[scale=1]{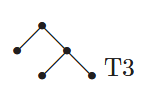}
    %\end{minipage}
    \end{center}
    All Cartan matrices are equal. This is the root system $B_2=C_2$. The set of positive roots is
    $$\big\{(0, 1), (1, 2), (1, 1), (1, 0)\big\}$$
    %$$\{ (0, 1), (1, 1), (2, 1), (1, 0)\}$$
    
    \item $n=5$ roots: The sequence $(3,1,2,2,1)$  corresponds to the triangulation and binary trees
     \begin{center}
    %\begin{minipage}[0.45\textwidth]
    \includegraphics[scale=.7]{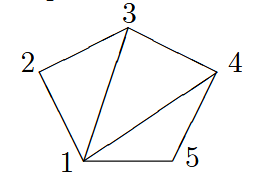}
    %\end{minipage}
    \hspace{1.5cm}
    %\begin{minipage}[0.45\textwidth]
    \includegraphics[scale=2]{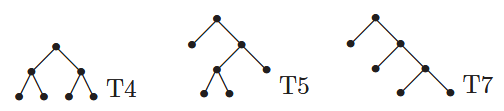}
    %\end{minipage}
    \end{center}
    depending on the starting edge. The set of positive roots in the different bases is
    $$\big\{(0, 1), (1, 3), (1, 2), (1, 1), (1, 0)\big\}$$
    $$\big\{(0, 1), (1, 1), (2, 1), (3, 1), (1, 0)\big\}$$ 
    $$\big\{(0, 1), (1, 2), (2, 3), (1, 1), (1, 0)\big\}$$
    $$\big\{(0, 1), (1, 2), (1, 1), (2, 1), (1, 0)\big\}$$
    $$\big\{(0, 1), (1, 1), (3, 2), (2, 1), (1, 0)\big\}$$
    In the chamber with $c_{ij}=-3$ this looks like the $G_2$ root system, but the nonobvious root $3\alpha_1+2\alpha_2$ is missing, and in other bases the picture look correspondingly different.
\item $n=6$ roots: There are $3$ different triangulations, one corresponding to $G_2$, one to the diagonal Nichols algebra in row 10 with trees T6, T9, T14, and one not to a diagonal Nichols algebra. 
    \end{itemize}

We consider the root system $n=5$ in more detail: We start with the initial set of roots, read off it the Cartan matrix and verify it is the one coming from the sequence $(3,1,2,2,1)$. Then we perform the reflection by choosing the given roots as new simple roots $(1,0),(0,1)$ in the next line and compute the images of the remaining roots by writing them as linear combinations of them (note that in general vector spaces the action on coordinate tuples is adjoint to the action on elements)\SimonRem{use involutiveness}, again sorted by increasing quotient:
\begin{center}
    \begin{tikzcd}
        \big\{(0,1),(1, 3),(1, 2),(1, 1),(1,0) \big\}
        \arrow[d,leftrightarrow, "\rho_2"]
        &
        \cm^{}=\begin{psmallmatrix}
         2 & -1 \\
         -3 & 2
         \end{psmallmatrix} 
         &
         \rho_2:(1,3),-(0,1)
         \mapsto (1,0),(0,1)
         %Remaining from below
         %(1,0),(1,1),(1,2)
         \\
         \big\{(0,1),(1,3),(1,2),(1,1),(1,0) \big\} 
        \arrow[d,leftrightarrow, "\rho_1"]
         &
        \cm^{}=\begin{psmallmatrix}
         2 & -1 \\
         -3 & 2
         \end{psmallmatrix} 
         &\rho_1:-(1,0),(1,1)
         \mapsto (1,0),(0,1)
         %Remaining from below
         %(0,1),(1,2),(1,3),
         \\
         \big\{(0,1),(1,2),(2,3),(1,1),(1,0) \big\}
        \arrow[d,leftrightarrow, "\rho_2"]
         &
         \cm^{}=\begin{psmallmatrix}
         2 & -1 \\
         -2 & 2
         \end{psmallmatrix} 
         &
         \rho_2: (1,2),-(0,1)
         \mapsto (1,0),(0,1)
         %Remaining from below
         %(1,0),(1,1),(23)
         \\
         \big\{(0,1),(1,2),(1,1),(2,1),(1,0) \big\}
         \arrow[d,leftrightarrow, "\rho_1"]
         &
          \cm^{}=\begin{psmallmatrix}
         2 & -2 \\
         -2 & 2
         \end{psmallmatrix} 
         &\rho_1:-(1,0),(2,1)
         \mapsto (1,0),(0,1)
         %Remaining from below
         %(1,1),(1,2),(0,1)
         \\
         \big\{(0,1),(1,1),(3,2),(2,1),(1,0) \big\}
        \arrow[d,leftrightarrow, "\rho_2"]
         &
         \cm^{}=\begin{psmallmatrix}
         2 & -2 \\
         -1 & 2
         \end{psmallmatrix}
         &
         \rho_2: (1,1),-(0,1)
         \mapsto (1,0),(0,1)
        %Remaining from below
        %(1,0),(2,1)
         \\
        \big\{(0,1),(1,1),(2,1),(3,1),(1,0) \big\}
        \arrow[d,leftrightarrow, "\rho_1"]
         &
         \cm^{}=\begin{psmallmatrix}
         2 & -3 \\
         -1 & 2
         \end{psmallmatrix}
         &\rho_1:-(1,0),(3,1)
         \mapsto (1,0),(0,1)
         \\
         \cdots
    \end{tikzcd}
\end{center}
where the last line is the first line with $1,2$ interchanged, so the pattern repeats once more, and overall we have $5+5$ lines.
\end{example}

\begin{exercise}
Check explicitly that the $q$-diagrams in Row 8 of Heckenberger's list (Figure \ref{fig_HeckenbergerRank2}) has the root system with $5$ roots above: Compute the Cartan matrices that Rossos' Lemma attaches to these $q$-diagrams. Verify then for some or all reflections, that the change in $q$-diagram corresponds to the change in Cartan matrix above.
\end{exercise}

\section{Reflection functors and PBW basis}\label{sec_parabolic}

Until now reflection was a procedure defining new braided objects $R_i(X)$ and root systems were combinatorial data attached to $X$ and all its reflections. Obviously for this to be useful, we want a relation between the Nichols algebras $\NicholsOf(X)$ and $\NicholsOf(R_i(X))$ and information about the Nichols algebra from the root system. 

\bigskip

We now sketch the two main results in this regard. We will only be able to discuss the prove much later in Section \ref{sec:PBW}.

\begin{lemma}[\cite{HS20} Section 13.5] 
The Nichols algebra of  $\NicholsOf(X)$ and $\NicholsOf(R_i(X))$ are related by a categorical operation we call \emph{partial dualization}. In particular:
\begin{itemize}
\item They have the same dimension if they are finite-dimensional. (but in general they are not isomorphic and do not have the same graded dimension).
\item Their representation categories are related by a invertible bimodule category. 
\item Their Drinfeld doubles are isomorphic, their Drinfeld centers are equivalent. 
\end{itemize}
\end{lemma}

As a consequence one can prove by writing the longest element in the Weyl groupoid as a reduced expression:

\begin{theorem}[\cite{AHS10}, \cite{HS20} Section 14.2]\label{thm_PBW}
In the setup of the book (Assumption \ref{ass_HS}), let $X=\bigoplus_{i\in I} X_i$ be a semisimple object with $X_i$ simple. Suppose the Nichols algebra $\NicholsOf(X)$ admits all reflections (meaning it is $i$-finite for any $i$ and the reflections are again) and the set of roots $\Phi^+$ is finite.

Then there is an isomorphism of objects (not of algebras)
$$\Nichols(X)\cong \bigotimes_{\alpha\in \Phi^+} \Nichols(X_\alpha)$$
where $X_{\alpha_i}=X_i$ and in general $X_\alpha$ are called \emph{root spaces}.
\end{theorem}

If the $X_\alpha=x_\alpha\C $ are $1$-dimensional, then this is a Poincare Birkhoff Witt type basis of sorted monomials in the $x_\alpha$, which are the root vectors defined by Lusztig. Note however that the root vector $x_\alpha$ is in general not simply defined as iterative commutator according to an additive decomposition of $\alpha$, rather to a choice of reduced expression to the longest element as in \cites{Lusz93, Jan96}.

\begin{example}[$\sl(2|1)$]
Consider the braiding matrix $(q_{ij})=\begin{pmatrix} q^2 & q^{-1} q^{-1} & -1 \end{pmatrix}$ with $q$ a primitive $\ell$-th root of unity. Then we have 
$$x_1,x_2,x_{12} \quad\text{with self braidings}\quad  q_{11}=q^2,q_{22}=-1,q_{11}q_{12}q_{21}q_{22}=-1$$
The Nichols algebra is generated by $x_1,x_2$ with relations $(x_1)^\ell=0, (x_2)^2=0$, which can be read as Nichols subalgebras $\NicholsOf(x_1\C),\NicholsOf(x_2\C)$. These generators have a $q$-commutator $x_{12}$ (which is not primitive) with relation $(x_{12})^2=0$. We have a PBW type basis $x_1^a x_2^b x_{12}^c$ with $a<\ell, b<2, c<2$, which can be written as an isomorphism 
$\Leftrightarrow 
\NicholsOf(X)\cong \C[x_1]/((x_1)^\ell)\otimes \C[x_2]/((x_2)^2) \otimes \C[x_{12}]/((x_{12})^2)
$
This is the PBW basis associated to the reduced expression of the longest element $s_1s_2s_1$. Had we chosen alternatively $s_2s_1s_2$, then our PBW basis would involve the non-proportional element $x_{21}=x_2x_1-q_{21}x_2x_1$.
\bigskip
After reflection $R_2$ we have generators $x_{21},x_2^*$ with self braiding $-1$ and a commutator with self-braiding $q$
$\Leftrightarrow 
\NicholsOf(R_2(X))\cong \C[x_{21}]/((x_{21})^2)\otimes \C[x_2^*]/((x_2^*)^2) \otimes \C[x_{1}]/((x_{1})^\ell)
$
\end{example}

\chapter{Category of representations}

\newcommand{\longhookrightarrow}{\lhook\joinrel\longrightarrow}
\DeclareRobustCommand\longtwoheadrightarrow
     {\relbar\joinrel\twoheadrightarrow}

\section{Representations of Nichols algebras}\label{sec_RepNichols}

Recall that the representations of any  Hopf algebra $K$ in any braided tensor category form a monoidal category.  

\SimonRem{What do we do: Vect, Rep, YetterDrinfeld?}

\begin{example}\label{exm_TaftRep}
Let $X=\C_g$ inside $\cC=\Vect_{\Z_n}^{\sigma,1}$ with $\Z_n=\langle g\rangle$, with a braiding $\sigma(g,g)=q$ a primitive $n$-th root of unity. Consider the Nichols algebra $K=\NicholsOf(X)=\C[x]/x^n$. 
The category of $K$-modules inside  $\Vect_{\Z_m}$ consists of $\Z_n$-graded vector spaces $V=\bigoplus_{g^k\in \Z_n} V_{g^k}$, which we place on a circle, and an additional endomorphism $x$ shifting the degree by $g$ and $x^n=0$.

\newcommand{\radiuscircle}{2cm}
\newcommand{\radiuscircleB}{1.85cm}

\noindent
\begin{minipage}{.5\textwidth}
\centering
\begin{tikzpicture}[->,scale=.7]
    %\node at (125:2.3cm) {$x$};
    %\node at (65:2.3cm) {$x$};

   \node at (210:\radiuscircle) {$V_{g^{-2}}$};
   \node at (150:\radiuscircle) {$V_{g^{-1}}$};
   \node at (90:\radiuscircle)  {$V_1$};
   \node at (30:\radiuscircle) {$V_g$};
   \node at (-30:\radiuscircle) {$V_{g^2}$};
   \node  at (-90:\radiuscircle) {$V_{g^3}$};
\end{tikzpicture}
\end{minipage}%
\begin{minipage}{.5\textwidth}
\centering
\begin{tikzpicture}[->,scale=.7]
    %\node at (125:2.3cm) {$x$};
    %\node at (65:2.3cm) {$x$};

   \node at (210:\radiuscircle) {$V_{g^{-2}}$};
   \node at (150:\radiuscircle) {$V_{g^{-1}}$};
   \node at (90:\radiuscircle)  {$V_1$};
   \node at (30:\radiuscircle) {$V_g$};
   \node at (-30:\radiuscircle) {$V_{g^2}$};
   \node  at (-90:\radiuscircle) {$V_{g^3}$};

   \draw (200:\radiuscircle)  arc (200:160:\radiuscircle);
   
    \node at (180:2.3cm) {$x$};
   \draw (140:\radiuscircle)  arc (140:100:\radiuscircle);
    \node at (120:2.3cm) {$x$};
   \draw (80:\radiuscircle)  arc (80:45:\radiuscircle);
    \node at (60:2.3cm) {$x$};
   \draw(15:\radiuscircle)  arc (15:-20:\radiuscircle);
    \node at (-3:2.3cm) {$x$};
   \draw (-40:\radiuscircle)  arc (-40:-80:\radiuscircle);
    \node at (-60:2.3cm) {$x$};
   \draw (-100:\radiuscircle)  arc (-100:-140:\radiuscircle);
    \node at (-120:2.3cm) {$x$};
\end{tikzpicture}
\end{minipage}

Acting with $x$ produces a submodule because $x^n=0$, hence any simple object must have zero action and is of the form $\C_{g^i}$, for example

\begin{center}
\centering
\begin{tikzpicture}[->,scale=.7]
    %\node at (125:2.3cm) {$x$};
    %\node at (65:2.3cm) {$x$};

   \node[gray]  at (210:\radiuscircle) {$0$};
   \node[gray] at (150:\radiuscircle) {$0$};
   \node[gray] at (90:\radiuscircle)  {$0$};
   \node at (30:\radiuscircle) {$\C_g$};
   \node[gray] at (-30:\radiuscircle) {$0$};
   \node[gray]  at (-90:\radiuscircle) {$0$};

   \draw[gray] (200:\radiuscircle)  arc (200:160:\radiuscircle);
   \draw[gray] (140:\radiuscircle)  arc (140:100:\radiuscircle);
    %\draw[gray] (100:\radiuscircleB)  arc (100:140:\radiuscircleB);
   \draw[gray] (80:\radiuscircle)  arc (80:45:\radiuscircle);
    %\draw (45:\radiuscircleB)  arc (45:80:\radiuscircleB);
   \draw[gray] (15:\radiuscircle)  arc (15:-20:\radiuscircle);
   \draw[gray] (-40:\radiuscircle)  arc (-40:-80:\radiuscircle);
   \draw[gray] (-100:\radiuscircle)  arc (-100:-140:\radiuscircle);
\end{tikzpicture}
\end{center}

There are also indecomposable extensions: Note that the $K$-module $\C_{g^i}$ can be mapped (under a morphism of $K$-modules)  to any element in some other $V_{g^i}$ as long as it is also in the kernel of $x$. Hence we have the following  exact sequences of morphisms of $K$-modules, nonsplit:

\noindent
\begin{minipage}{.3\textwidth}
\centering
\begin{tikzpicture}[->,scale=.7]
    %\node at (125:2.3cm) {$x$};
    %\node at (65:2.3cm) {$x$};

   \node[gray]  at (210:\radiuscircle) {$0$};
   \node[gray] at (150:\radiuscircle) {$0$};
   \node at (90:\radiuscircle)  {$\C_1$};
   \node[gray] at (30:\radiuscircle) {$0$};
   \node[gray] at (-30:\radiuscircle) {$0$};
   \node[gray]  at (-90:\radiuscircle) {$0$};

   \draw[gray] (200:\radiuscircle)  arc (200:160:\radiuscircle);
   \draw[gray] (140:\radiuscircle)  arc (140:100:\radiuscircle);
    %\draw[gray] (100:\radiuscircleB)  arc (100:140:\radiuscircleB);
   \draw[gray] (80:\radiuscircle)  arc (80:45:\radiuscircle);
    %\draw (45:\radiuscircleB)  arc (45:80:\radiuscircleB);
   \draw[gray] (15:\radiuscircle)  arc (15:-20:\radiuscircle);
   \draw[gray] (-40:\radiuscircle)  arc (-40:-80:\radiuscircle);
   \draw[gray] (-100:\radiuscircle)  arc (-100:-140:\radiuscircle);
\end{tikzpicture}
\end{minipage}%
$\longhookrightarrow$
\begin{minipage}{.3\textwidth}
\centering
\begin{tikzpicture}[->,scale=.7]
    %\node at (125:2.3cm) {$x$};
    %\node at (65:2.3cm) {$x$};

   \node[gray]  at (210:\radiuscircle) {$0$};
   \node[gray] at (150:\radiuscircle) {$0$};
   \node at (90:\radiuscircle)  {$\C_1$};
   \node at (30:\radiuscircle) {$\C_g$};
   \node[gray] at (-30:\radiuscircle) {$0$};
   \node[gray]  at (-90:\radiuscircle) {$0$};

   \draw[gray] (200:\radiuscircle)  arc (200:160:\radiuscircle);
   \draw[gray] (140:\radiuscircle)  arc (140:100:\radiuscircle);
    \draw (80:\radiuscircle)  arc (80:45:\radiuscircle);
    \node at (60:2.3cm) {$x$};
    %\draw (45:\radiuscircleB)  arc (45:80:\radiuscircleB);
   \draw[gray] (15:\radiuscircle)  arc (15:-20:\radiuscircle);
   \draw[gray] (-40:\radiuscircle)  arc (-40:-80:\radiuscircle);
   \draw[gray] (-100:\radiuscircle)  arc (-100:-140:\radiuscircle);
\end{tikzpicture}
\end{minipage}
$\longtwoheadrightarrow$
\begin{minipage}{.3\textwidth}
\centering
\begin{tikzpicture}[->,scale=.7]
    %\node at (125:2.3cm) {$x$};
    %\node at (65:2.3cm) {$x$};

   \node[gray]  at (210:\radiuscircle) {$0$};
   \node[gray] at (150:\radiuscircle) {$0$};
   \node[gray] at (90:\radiuscircle)  {$0$};
   \node at (30:\radiuscircle) {$\C_g$};
   \node[gray] at (-30:\radiuscircle) {$0$};
   \node[gray]  at (-90:\radiuscircle) {$0$};

   \draw[gray] (200:\radiuscircle)  arc (200:160:\radiuscircle);
   \draw[gray] (140:\radiuscircle)  arc (140:100:\radiuscircle);
    %\draw[gray] (100:\radiuscircleB)  arc (100:140:\radiuscircleB);
   \draw[gray] (80:\radiuscircle)  arc (80:45:\radiuscircle);
    %\draw (45:\radiuscircleB)  arc (45:80:\radiuscircleB);
   \draw[gray] (15:\radiuscircle)  arc (15:-20:\radiuscircle);
   \draw[gray] (-40:\radiuscircle)  arc (-40:-80:\radiuscircle);
   \draw[gray] (-100:\radiuscircle)  arc (-100:-140:\radiuscircle);
\end{tikzpicture}
\end{minipage}%

The projective modules are the  indecomposables of maximal length $n$, in view of $x^n=0$.
\end{example}

Compare this to the modules of $\C[x]$ in the category of vector spaces in Exercise \ref{exercise_RepPolynomialring}.

\begin{exercise}
Let $\C_{g^k,l}$ be the indecomposable $K$-module inside $\Vect_{\Z_m}$ which is as graded vector space $\C_{g^k}\oplus \C_{g^{k+1}}\oplus\cdots \oplus \C_{g^{k+l-1}}$. Compute the dimension of the space of homomorphism between $\C_{g^k,l}$ and $\C_{g^{k'},{l'}}$ (Answer: it is one if $k\in[k',k'+l']$ and $k+l>k'+l'$ and zero otherwise).
\end{exercise}

Nichols algebras are nilpotent in the following sense, which follows solely from the grading
\begin{lemma}
Any simple representation of $\NicholsOf(X) $ is a simple object of the base category with trivial action.
\end{lemma}

\section{Drinfeld center}\label{sec_DrinfeldCenters}

\begin{definition}
For $\cD$ a given monoidal category, one can construct a universal braided monoidal category $\cZ(\cD)$ called \emph{Drinfeld center}, consisting of pairs $(X,c_{X,-})$ of objects $X\in\cC$ together with half-braidings $c_{X,Y}:X\otimes Y\to Y\otimes X$ for any $Y\in\cD$. As coherence condition we ask for the hexagonal identity in the second argument (see Definition \ref{def_BTC}).
\end{definition}

\begin{example}
Consider $\cC=\Vect_G$ with trivial associator, and assume $G$ abelian, so the tensor product is a-priori commutative. Let $V\in \Vect_G$ be an object, we are interested in the possible half-braidings $c_{V,-}$. It is obviously enough to define them on simple objects, that is, for any $h\in G
$ a morphism 

\[V\otimes \C_h\stackrel{\rho(h)}{\longrightarrow} \C_h\otimes V\] 

where both sides can be identified with $v$ as vector spaces. The hexagon identity in the right argument implies $\rho(h')\rho(h'')=\rho(h'h'')$ and unitality implies $\rho(e)=\id$, so $\rho=c_{V,\C_{(-)}}$ defines a representation of $G$ on $V$. Since $\rho$ should be a morphism in the category, it has to preserve the grading.

\begin{lemma}
The Drinfeld center of $\Vect_G$ consist of objects $V$ together with representations of $G$ respecting the $G$-grading.
\end{lemma}

In particular, every such object is a direct sum of objects $V=V_g$ in a fixed homogeneous degree $g$, then we denote it by $\O_g^V$. The tensor product and braiding in the center is 
\[\O_{g}^{V}\otimes \O_{h}^{W}=\O_{gh}^{VW}\]
\[
\O_{g}^{V}\otimes \O_{h}^{W}
\to \O_{h}^{W}\otimes \O_{g}^{V}
\]
\[
v\otimes w\mapsto 
g.w\otimes w
\] Note that  the object $\O_g^V$ is simple, indecomposable or projective iff $V$ is, depending on the base field. In particular for $\K=\C$ the center is again semisimple and the simple objects are 1-dimensional $\O_g^\rho$ for $\rho:G\to \C^\times$. 

For nonabelian groups $G$, not every object will be in the center, only direct sums over conjugacy classes, we will do this in section \ref{sec_nonabelian}.
\end{example}
\begin{example}\label{exm_centerofRepG}
Consider $\cC=\Rep(G)$ with trivial associator. Every object will be in the center, but it is tricky to determine the possible half-braidings. It follows from a general principle  Theorem  \ref{thm_DualizeYD} that we get the same center center 
$$\cZ(\Rep(G))\cong \cZ(\Vect_G)$$
For $G$ abelian and base field $\C$ this equivalence is very clear, because $\Rep(G)=\Vect_{\hat{G}}$ for $\hat{G}:\Hom(G,\C^\times)$ the group of $1$-dimensional characters, and $\smash{\hat{\hat{G}}}=G$ by evaluating a character at a given group element. Hence the objects are of the form $\C_\chi^g$ and can be identified roughly with $\C_g^\chi$. Note however the nontrivial tensor structure in  Theorem  \ref{thm_DualizeYD}.
\end{example}

Note that for a Hopf algebra $H$ in vector spaces there is a quasi-triangular Hopf algebra $D(H)$ called Drinfeld double such that $\Rep(D(H))=\cZ(\Rep(H))$, a notion that predates the Drinfeld center. This appears in the earliest construction of the quantum groups. In older texts there is also a frequent "identification of Cartan parts" that is categorically the following:

\begin{definition}
For $\cD$ a given monoidal category with a central subcategory $\cD\hookleftarrow \cC$ (that is, $\cC$ is braided the embedding upgrades to a braided monoidal functor to the center of $\cD$, that is, its objects come with a distinguished half-braiding with all of $\cD$). Then  one can construct a universal braided monoidal category $\cZ_\cC(\cD)$ inside $\cZ(\cD)$ called \emph{relative Drinfeld center}, consisting of pairs $(X,c_{X,-})$ of objects $X\in\cC$ together with half-braidings $c_{X,Y}:X\otimes Y\to Y\otimes X$ that coincide with the given half-braiding of the central embedding.
\end{definition}

We now discuss a special case, that predates the categorical definition and is used extensively in \cite{HS20}: Let $\cD=\Rep(H)$ be the categories of representations of a Hopf algebra in some ambient braided monoidal category $\cC$ (classically the category of vector spaces). 

Then a half-braiding $c_{X,-}$  can be described as an additional coaction $\delta_X:X\to H\otimes X$, in Sweedler notation $\delta(x)=x^{(0)}\otimes x^{(1)}$. \footnote{Note that for $\Rep(H)$ the set of natural transformations of the identity functor can be identified with elements of $H$ acting simultaneously in a way that by definition commutes with all morphisms.}  Explicitly  
$$c_{X,Y}:X\otimes Y\longrightarrow Y\otimes X$$
$$x\otimes y\mapsto x^{(0)}.y\otimes x$$
using silently also braiding of the underlying category $\cC$. The necessarily compatibility of action and coaction leads to the following definition 

They have been generalized in \cite{Besp95} to Hopf algebras $\Nichols$ inside a braided tensor category $\cC$, see also \cite{BLS15} Section 2 and 3: 
\begin{definition}\label{def_YD}
Let $H$ be a Hopf algebra in a braided category $\cC$. A $H$-$H$-\emph{Yetter-Drinfeld module} $X$ is an object $X$ in $\cC$ with the structure of a $H$-module and $H$-comodule in $\cC$, such that the following compatibility condition holds
\begin{center}
                \begin{grform}
                        \begin{scope}[scale = 0.5]
                                \dMult{0.5}{1.5}{1}{-1.5}{\grau}
                                \dMult{0.5}{2.5}{1}{1.5}{\grau}
                                \dAction{2.5}{1.5}{0.5}{-1.5}{\grau}{black}
                                \dAction{2.5}{2.5}{0.5}{1.5}{\grau}{black}
                                \vLine{0.5}{1.5}{0.5}{2.5}{\grau}
                                \vLine{2.5}{1.5}{1.5}{2.5}{\grau}
                                \vLineO{1.5}{1.5}{2.5}{2.5}{\grau}
                                \vLine{3}{1.5}{3}{2.5}{black}
                        \end{scope}
                        \draw (0.5 , -0.3) node {$H$};
                        \draw (0.5 , 2.3) node {$H$};
                        \draw (1.5 , -0.3) node {$X$};
                        \draw (1.5 , 2.3) node {$X$};
                \end{grform}
                %%%%%%%%%
                =
                %%%%%%%%%
                \begin{grform}
                        \begin{scope}[scale = 0.5]
                                \vLine{2}{0}{1}{1}{black}
                                \dMultO{0}{1}{1.5}{-1}{\grau}
                                \dAction{0}{1}{1}{1}{\grau}{black}
                                \dAction{0}{3}{1}{-1}{\grau}{black}
                                \dMult{0}{3}{1.5}{1}{\grau}
                                \vLineO{2}{4}{1}{3}{black}
                                \vLine{1.5}{1}{1.5}{3}{\grau}
                        \end{scope}
                        \draw (0.45 , -0.3) node {$H$};
                        \draw (0.45 , 2.3) node {$H$};
                        \draw (1 , -0.3) node {$X$};
                        \draw (1 , 2.3) node {$X$};
                \end{grform}
\end{center}
 
\end{definition}

The category $\YD{H}(\cC)$ consists of  Yetter-Drinfeld modules and of $H$-linear and $H$-colinear morphisms. It becomes a tensor category with the usual tensor product of $H$-modules and $H$-comodules $X\otimes Y$ (see Section \ref{sec_RepNichols}).

\begin{lemma}[\cite{Lau20} Prop.  3.36]\label{lm_relcenterSplitting}
    Suppose $\cB=\Rep(H)(\cC)$, then the relative center is
    $$\cZ_\cC(\cB)=\YD{H}(\cC)$$
\end{lemma}

\begin{theorem}[\cite{BLS15}, \cite{HS20} Theorem 12.3.2]\label{thm_DualizeYD}
For $H$ a Hopf algebra in $\cC$, which has a dual object $H^*$ (then automatically a Hopf algebra) there is an equivalence of braided categories 
$$\Omega:\;\YD{H}(\cC)\stackrel{\sim}{\longrightarrow}
\YD{H^*}(\cC)$$
\end{theorem}
Note that in the second source there is a more general statement in the infinite case using rational Yetter-Drinfeld modules. There is a more categorical picture where $\cB,\cB^*$ are dual over a module category $\cC$.

\begin{example}
For $H=\C[G]$ and $\cC=\mathrm{Vect}$ the Yetter-Drinfeld modules are $G$-modules with $G$-grading and a nontrivial compatibility condition for nonabelian group, which we will study later. For $G$ abelian we recover the center discussed above. The equivalence to the dual picture, the center of $\Vect_G$, is also as discussed above, and already slightly nontrivial.
\end{example}

\begin{exercise}
    Check using braid diagrams that the tensor product of two Yetter-Drinfeld modules fulfills again the compatibility condition of a Yetter-Drinfeld module. Optinally, you may also check the hexagon identity for the braiding.
\end{exercise}
\begin{exercise}
    Spell out the compatibility condition for Yetter-Drinfeld modules over a group.
\end{exercise}

\section{Radford biproduct}\label{sec:RadfordBiproduct}

 Suppose $K$ is a bialgebra in the Drinfeld center $\cZ(\Rep(L))$ of the category of representations of a Hopf algebra $L$. Typically $L$ is a Hopf algebra in the category of vector spaces, but it can be in any nontrivial braided monoidal category $\cC$. Explicitly, this means the following data and compatibilities: 
\begin{itemize}
\item $K$ carries an action of $L$.
\item The algebra and coalgebra structure are compatible with this action in the sense of the tensor category $\Rep(L)$, explicitly as (co)product rule
\begin{align*}
h.(b_1b_2)&=\mu_K(\Delta_L(h)(b_1\otimes b_2))=(h^{(1).b_1})(h^{(2).b_2})
\\
\Delta_K(h.b)&=\Delta_L(h).\Delta_K(b)
=h^{(1).b^{(1)}}\otimes h^{(2)}.b^{(2)}
\end{align*}
\item $K$ is endowed with a half-braiding $c_{K,X}$ for any $L$-module $X$, or equivalently it is a $L$-$L$-Yetter Drinfeld module, see Section \ref{sec_DrinfeldCenters}.  
\item The algebra and coalgebra structure are similarly compatible with the half-braiding. 
\item The bialgebra axiom is fullfilled with respect to the half-braiding $c_{K,K}$.
\end{itemize}

\begin{theorem}[Radford biproduct, categorical version \cite{Besp95}, in the book \cite{HS20} Section 3.8]
The monoidal category of representations of Hopf-/bialgebra $K$ inside $\Rep(L)$ inside $\cC$ is equivalent to the monoidal category of representations of a Hopf-/bialgebra $K\# L$ or as I like to write $K\rtimes L$ inside $\cC$. 

$$\Rep(K\rtimes L)(\cC)\cong \Rep(K)\Rep(L)(\cC)$$

More precisely, as a vector space we take $K\otimes H$ and modify the multiplication of $K$ by the action of $H$ (so the new adjoint action is the given action) and the comultiplication of $K$ by the coaction of $H$.
\end{theorem}

This generalizes the notion of a semidirect product of groups or of Lie algebras, where the coaction is trivial. For the quantum group, we have $u_q(\g)^{\geq 0}=u_q(\g)^+\rtimes u_q(\g)^0$ and the first factor is not an ordinary Hopf algebra, but a Nichols algebra over the second factor. 

\section{Radford projection theorem}\label{sec:RadfordProjection}

There is also a vice-versa decomposition theorem, as for split exact sequences of groups:

\begin{theorem}[Radford projection theorem \cite{Rad85}, 
categorical version \cites{Besp95,AF00} in the book \cite{HS20} Section 3.9]\label{thm_RadfordProjection}
Let $\pi:H\twoheadrightarrow L$ be a surjection of Hopf algebras in a braided monoidal category $\cC$, which admits a section of Hopf algebra $\iota: H\hookleftarrow L$. Consider the space of coinvariants 
$$H^{\coin(\pi)}
:=\{h\in H\mid (\id\otimes \pi)\Delta(h)=h\otimes 1_L\}$$
of course in general categories defined as an equalizer. Then $H^{\coin(\pi)}$ carries an action of $L$ by adjoint action, a coaction via $(\pi\otimes \id)\Delta$, and a modified coproduct
$$\Delta_K(k)=(\mu_J\otimes \id)(\id\otimes \pi\circ S_H\otimes \id)(\Delta\otimes \id)\Delta_H (k)$$
which lands again in $K\otimes K$ and turns $K$ into a Hopf algebra in the Drinfeld center of $\Rep(L)$ or equivalently in the category of $L$-$L$-Yetter Drinfeld modules. Moreover
$$\Rep(H)(\cC)\cong \Rep(H^{\coin(\pi)})\Rep(L)(\cC)$$
\end{theorem}

\begin{example}
Take any homomorphism of groups $f:G\to Q$ and extend it to a homomorphism of group rings $f':\K[G]\to \K[Q]$. The coinvariants $\K[G]^{\coin(\pi)}$ are precisely the span of the kernel of $f$ (note that the literal kernel of $f'$ is different and much larger).

\bigskip

Assume now that $f$ splits via some $s$. Then it is a standard result of group theory that this splits as a semidirect product $\ker(f)\rtimes Q$. The action is the adjoint action of $Q$ on the kernel, the coaction is trivial, hence also the braiding is trivial. This is why in this case the kernel is again an ordinary group, but with an additional action datum. Similarly for Lie algebras, which has to do in both cases with cocommutativity. For general Hopf algebras, this is not so and in general the kernel lives in a Yetter-Drinfeld category.
\end{example}

\begin{example}\label{exm_Taft}
The \emph{Taft algebra} $T$ is a Hopf algebra in the category of vector spaces defined as follows: It has generators $g,x$ with algebra relations 
$$g^n=1,\quad x^n=0,\quad  gx=qxg$$
for $q$ a primitive  $n$-th root of unity. It has coproduct
$$\Delta(g)=g\otimes g,\qquad
\Delta(x)=g\otimes x + x\otimes 1
$$
and antipode $S(g)=g^{-1}$ and $S(x)=-g^{-1}x$. That is, $g$ is a group element generating the group $\Z_n$ and $x$ is \emph{skew-primitive}.

\bigskip

Then there is a Hopf algebra projection $\pi$ to the group ring $\C[\Z_n]$ by setting $\pi(x)=0$. we compute the space of coinvariants
$$K=T^{\coin(\pi)}
=\C[x]/x^n$$
The adjoint action is $g.x=qx$ and the coadjoint action is $x\mapsto g\otimes x$. This defines a Yetter-Drinfeld module and thus a braiding with any $\C[\Z_n]$-module $V$ by $x\otimes v\mapsto g.v\otimes x$. In particular the self braiding of $x\C$ is $x\otimes x\mapsto q(x\otimes x)$. The modified coproduct of $K$ is $\Delta(x)=1\otimes x+x\otimes 1$. Hence $T^{\coin(\pi)}$ is the Nichols algebra we encountered above, and the diagonally braided vector space realized as $\Z_n$-modules.

\bigskip

Conversely, the Taft algebra can be built as a Radford biproduct of this Nichols algebra and the group ring $\C[\Z_n]$.

To connect to our study of the representations of $\NicholsOf(X)$ in $\Vect_\Gamma$ in Section \ref{sec_RepNichols}, we can identify representations of $\Z_n$ by vector spaces graded by $1$-dimensional characters in $\hat{\Z_n}$ as in Example~\ref{exm_centerofRepG}. Then $x$ is an element in degree $\chi$, where $\chi(g)=q$ and the half-braiding is given by $g$.
\end{example}

\begin{exercise}
\begin{enumerate}[(a)]
\item 
Check the Hopf algebra axioms for the Taft algebra. Hint: Check first coassociativity on the generators. Then \emph{define} $\Delta$ as an algebra morphism on the free algebra in the generators as product of $\Delta$ given on the generators. Then it remains to check that $\Delta$ factors over the algebra relations, so you actually have to compute $\Delta$ of the algebra relations. The most difficult one is $x^n$, and here we encounter the same calculations as in Example \ref{exm_rankOneHopfAlg}.
\item 
Check that $\pi$ defines a homomorphism of Hopf algebras. 
\item
Check explicitly the statement of the Radford projection theorem, that is, a module over the Taft algebra precisely defines a module over the Nichols algebra over the group ring $\C[\Z_m]$. 
\item 
Consider a version of the Taft algebra where $x^m=0$ and $g^m=1$ and $\Delta(x)=g^k\otimes x+x\otimes 1$. What are the precise conditions relating $n,m,k$ and the order of $q$ such that this defines a Hopf algebra? \item Define an even more general version of the Taft algebra with group ring $\C[G]$ and $\Delta(x)=g^k\otimes x+x\otimes 1$ for some fixed $g\in G$.
\end{enumerate}
\end{exercise}

The Radford projection theorem has a categorical version \cite{LM25}:

\begin{theorem}\label{thm_Mombelli}
Let $\cB$ be a finite and rigid monoidal category and $\cC\subset \cB$ be a central braided monoidal category.  Assume this inclusion \emph{splits} by an exact faithful tensor functor $F$
 $$\begin{tikzcd}[row sep=10ex, column sep=15ex]
   {\cC}\; 
   \arrow[shift left=1, hookrightarrow]{r}
   %\arrow[darkgreen]{dr}{}
   & \arrow[shift left=1, darkgreen]{l}{F}
   %\arrow[shift left=2,darkgreen]{d}{\text{split}} 
   \cB
    \end{tikzcd}$$
Then there exists a Hopf algebra $\Nichols\in\cC$ such that  
$$\cB\cong\CoRep(\Nichols)$$ 
In this correspondence, the functor $F$ is the functor forgetting the $\Nichols$-action and the inclusion of $\cC$ is given by the trivial $\Nichols$-action.
\end{theorem}

\section{Quantum groups}

Let $\cC$ be a braided monoidal category and $K$ a braided Hopf algebra in $\cC$, typically  $K=\NicholsOf(X)$ a Nichols algebra.

\begin{definition}
As \emph{generalized quantum groups} we consider the braided monoidal category
$$\cZ_\cC(\Rep(K))$$
It has an exact faithful  functor forgetting the half-braiding to 
$\Rep(K)$, which we consider to be the generalized Borel part, and we call $\cC$ the generalized Cartan part. 
\end{definition}

This way of constructing a quantum group as a Drinfeld double has been there more or less from the beginning, and the language to discribe it has developed over the year the the current point \cites{Drin89, Maj95, AG03, AS10,Lau20}.

\begin{example}
Let $\g$ be a semisimple complex finite-dimensional Lie algebra. Let  $\cC=\Vect_\Gamma^{\sigma,1}$ in such a way that $\cC=\Rep(u_q(\g)^0)$ is the Cartan part. Take the object
$$X= \C_{\alpha_1}\oplus\cdots\oplus \C_{\alpha_n}$$
and assume $\sigma(\alpha_i,\alpha_j)=q^{(\alpha_i,\alpha_j)}$. Then we already mentioned that  that $\NicholsOf(X)$ is isomorphic to $u_q(\g)^+$. Further, the Radford biproduct is $u_q(\g)^{\geq 0}$. Finally, the relative Drinfeld center is equivalent to representations of the full quantum group $u_q(\g)$.

We remark that in this approach we can see many possible modifications already of this type of quantum group, namely by changing the Borel. They have appeared in literature under different names:
\begin{itemize}
    \item Taking a different braiding matrix with same $q$-diagram gives the multiparameter quantum groups. 
    \item Taking $\Vect_{\mathfrak{h}}$ or $\Rep(\mathfrak{h})$ gives an unrolled quantum group.
    \item Taking $\Vect_\Gamma^{\sigma,\omega}$ for nontrivial associator gives the quasi quantum groups. In particular this is necessary if we want to have a nondegenerate braiding at an even order root of unity.
\end{itemize}
\end{example}

\begin{remark}
For any relative Drinfeld center $\cZ_\cC(\cB)$ there is an exact faithful monoidal functor $\cZ_\cC(\cB)\to \cB$ forgetting the half-braiding and under favorable assumptions there is a right adjoint (lax monoidal) functor $\forget^\ra:\cZ_\cC(\cB)\leftarrow \cB$. In particular in our case where $\cB$ is the category of representations of a Hopf algebra there is an explicit right adjoint 
$$V\mapsto\Nichols\otimes_\cC V$$ 
with the regular coaction on the first factor and an explicit somewhat involved action on the tensor product involving double braidings, which turns the result into a Yetter-Drinfeld module, see  \cite{LW21} Theorem 3.13 and following. In particular the image of an object in $\cC$ endowed with the trivial action is comparable to a coVerma module in a quantum group. The image of the unit is $K$ with the coregular coaction and the adjoint action and it is a commutative algebra in $\cU^K$
\end{remark}

Understanding the structure of the category from this construction is not easy, apart from the fact that all objects in the center are in particular representations of the Cartan part (weight spaces) and the Nichols algebra (nilpotent part). One needs to write down the relations in the quantum group between resp. the relation between action and coaction explicitly. In cases beyond $\sl_2$ one also has to use the root system theory and reflection functors we give in Section \ref{sec_parabolic}.We first give the result for $\sl_2$: 

\begin{example}
In the example $\NicholsOf(X)=\C[x]/x^n$ in $\Rep(\Z_n)$ the category of representations was discussed in Example \ref{exm_Taft}. It is equivalent to representations of the  Radford biproduct, which is the Taft algebra in Example \ref{exm_Taft}. The relative Drinfeld center, or equivalently the category of Yetter-Drinfeld modules, gives a systematic description of the braided monoidal category of representations of the small quantum group $u_q(\sl_2)$ in Example \ref{sec_quantumgroups}.

This category turns out to have the following form: The simple objects  are $\NicholsOf(X)$-indecomposables 
%on $\C_{q^{-\lambda}}\oplus\cdots\oplus \C_{q^{\lambda}}$.\\
with symmetry $q\leftrightarrow q^{-1}$ (Weyl reflection). The indecomposable extensions of length $2$ are extensions of $q^\lambda,q^{-\lambda-2}$ with parameter $(a:b)$

\newcommand{\radiuscircle}{2cm}
\newcommand{\radiuscircleB}{1.85cm}

\noindent
\begin{minipage}{.5\textwidth}
\centering
\begin{tikzpicture}[->,scale=.7]
    %\node at (125:2.3cm) {$x$};
    %\node at (65:2.3cm) {$x$};

   \node[gray]  at (210:\radiuscircle) {$0$};
   \node at (150:\radiuscircle) {$\C_{q^{-2}}$};
   \node at (90:\radiuscircle)  {$\C_{1}$};
   \node at (30:\radiuscircle) {$\C_{q^{2}}$};
   \node[gray] at (-30:\radiuscircle) {$0$};
   \node[gray]  at (-90:\radiuscircle) {$0$};

   \draw[gray] (200:\radiuscircle)  arc (200:160:\radiuscircle);
   \draw (140:\radiuscircle)  arc (140:100:\radiuscircle);
    \draw (100:\radiuscircleB)  arc (100:140:\radiuscircleB);
   \draw (80:\radiuscircle)  arc (80:45:\radiuscircle);
    \draw (45:\radiuscircleB)  arc (45:80:\radiuscircleB);
   \draw[gray] (15:\radiuscircle)  arc (15:-20:\radiuscircle);
   \draw[gray] (-40:\radiuscircle)  arc (-40:-80:\radiuscircle);
   \draw[gray] (-100:\radiuscircle)  arc (-100:-140:\radiuscircle);
\end{tikzpicture}
\end{minipage}%
\begin{minipage}{.5\textwidth}
\centering
\begin{tikzpicture}[->,scale=.7]
       %\node at (180:2.3cm) {$x^*$};
              \node at (180:1.7cm) {$a$};

   %\node at (125:2.3cm) {$x$};
   %\node at (65:2.3cm) {$x$};
    %\node at (0:2.3cm) {$x$};
            \node at (0:1.6cm) {$b$};
    %\node at (-55:2.3cm) {$x$};
    %\node at (-115:2.3cm) {$x$};

   \node at (210:\radiuscircle) {$\C_{q^{8}}$};
   \node at (150:\radiuscircle) {$\C_{q^{-2}}$};
   \node at (90:\radiuscircle)  {$\C_{1}$};
   \node at (30:\radiuscircle) {$\C_{q^{2}}$};
   \node at (-30:\radiuscircle) {$\C_{q^{4}}$};
   \node  at (-90:\radiuscircle) {$\C_{q^{6}}$};

   %\draw[dashed] (200:\radiuscircle)  arc (200:160:\radiuscircle);
     \draw[dashed] (160:\radiuscircle)  arc (160:195:\radiuscircle);
   \draw (140:\radiuscircle)  arc (140:100:\radiuscircle);
      \draw (100:\radiuscircleB)  arc (100:140:\radiuscircleB);
   \draw (80:\radiuscircle)  arc (80:45:\radiuscircle);
      \draw (45:\radiuscircleB)  arc (45:80:\radiuscircleB);
   \draw[dashed] (15:\radiuscircle)  arc (15:-20:\radiuscircle);
      %\draw (-20:\radiuscircleB)  arc (-20:15:\radiuscircleB);
   \draw (-45:\radiuscircle)  arc (-45:-75:\radiuscircle);
      \draw (-75:\radiuscircleB)  arc (-75:-45:\radiuscircleB);
   \draw (-105:\radiuscircle)  arc (-105:-140:\radiuscircle);
      \draw (-135:\radiuscircleB)  arc (-135:-105:\radiuscircleB);
\end{tikzpicture}
\end{minipage}

The case $(1:0)$ and $(0:1)$ correspond to Verma modules. 

The projective cover contain the indecomposables $(1:0)$ and $(0:1)$, two each.

\bigskip

\begin{center}
\begin{tikzpicture}[->,scale=.7]

   %top
    \newcommand{\shiftx}{-1.7cm}
   \renewcommand{\radiuscircle}{8cm}
    \renewcommand{\radiuscircleB}{7.85cm}
   
   \node at ([shift=(150:\radiuscircle)]0,\shiftx) {$\C_{q^{-2}}$};
   \node at ([shift=(90:\radiuscircle)]0,\shiftx)  {$\C_{1}$};
   \node at ([shift=(30:\radiuscircle)]0,\shiftx) {$\C_{q^{2}}$};
   
   \draw ([shift=(145:\radiuscircle)]0,\shiftx)  arc (145:95:\radiuscircle);
      \draw ([shift=(95:\radiuscircleB)]0,\shiftx)  arc (95:145:\radiuscircleB);
   \draw([shift=(85:\radiuscircle)]0,\shiftx)  arc (85:35:\radiuscircle);
      \draw ([shift=(35:\radiuscircleB)]0,\shiftx)  arc (35:85:\radiuscircleB);
   
  \node at (-7,0.5) {$x^*$};
  \draw[dashed] (-7,1.7)--(-5.6,-0.8);
\node at (-3.6,0.5) {$x$};
  \draw[dashed] (-5.3,-0.8)--(-3.2,1.7);
 
 \node at (6.8,0.5) {$x$};
  \draw[dashed] (6.8,1.7)--(5.3,-0.8);
\node at (3.3,0.5) {$x^*$};
  \draw[dashed] (5.1,-0.8)--(3.0,1.7);

   \renewcommand{\radiuscircle}{2.5cm}
    \renewcommand{\radiuscircleB}{2.35cm}
   
   %left
   \newcommand{\shifty}{3.2cm}
   \node at ([shift=(-150:\radiuscircle)]\shifty,0) {$\C_{q^{8}}$};
   \node at ([shift=(-30:\radiuscircle)]\shifty,0) {$\C_{q^{4}}$};
   \node  at ([shift=(-90:\radiuscircle)]\shifty,0) {$\C_{q^{6}}$};

   \draw ([shift=(-45:\radiuscircle)]\shifty,0)  arc (-45:-75:\radiuscircle);
      \draw ([shift=(-75:\radiuscircleB)]\shifty,0)  arc (-75:-45:\radiuscircleB);
   \draw ([shift=(-105:\radiuscircle)]\shifty,0)  arc (-105:-140:\radiuscircle);
      \draw ([shift=(-135:\radiuscircleB)]\shifty,0)  arc (-135:-105:\radiuscircleB);
      
      %left
    \renewcommand{\shifty}{-3.2cm}
   \node at ([shift=(-150:\radiuscircle)]\shifty,0) {$\C_{q^{8}}$};
   \node at ([shift=(-30:\radiuscircle)]\shifty,0) {$\C_{q^{4}}$};
   \node  at ([shift=(-90:\radiuscircle)]\shifty,0) {$\C_{q^{6}}$};

   \draw ([shift=(-45:\radiuscircle)]\shifty,0)  arc (-45:-75:\radiuscircle);
      \draw ([shift=(-75:\radiuscircleB)]\shifty,0)  arc (-75:-45:\radiuscircleB);
   \draw ([shift=(-105:\radiuscircle)]\shifty,0)  arc (-105:-140:\radiuscircle);
      \draw ([shift=(-135:\radiuscircleB)]\shifty,0)  arc (-135:-105:\radiuscircleB);

      %bottom
   \renewcommand{\radiuscircle}{3cm}
    \renewcommand{\radiuscircleB}{2.85cm}
    \renewcommand{\shiftx}{0.7cm}

   \node at ([shift=(150:\radiuscircle)]0,\shiftx) {$\C_{q^{-2}}$};
   \node at ([shift=(90:\radiuscircle)]0,\shiftx)  {$\C_{1}$};
   \node at ([shift=(30:\radiuscircle)]0,\shiftx) {$\C_{q^{2}}$};
   
   \draw ([shift=(140:\radiuscircle)]0,\shiftx)  arc (140:100:\radiuscircle);
      \draw ([shift=(100:\radiuscircleB)]0,\shiftx)  arc (100:140:\radiuscircleB);
   \draw([shift=(80:\radiuscircle)]0,\shiftx)  arc (80:50:\radiuscircle);
      \draw ([shift=(50:\radiuscircleB)]0,\shiftx)  arc (50:80:\radiuscircleB);
   
\end{tikzpicture}
\end{center} 
\end{example}

\section{Appendix: Generators and relations}

For the convenience of the reader, we give an explicit realizing Hopf algebra including several cases above (but trivial associativity) which applies to arbitrary diagonal Nichols algebras. The calculation is taken from \cite{CLR23} Section 6.2:

\bigskip

Assume the braided tensor category $\cC$ is realized as finite-dimensional semisimple modules $\Rep^{wt}(C)$ of a commutative cocommutative Hopf algebra $C$, so $\cC=\Vect_\Gamma$ for a possibly infinite abelian group $\Gamma=\mathrm{Spec}(C)$ with a braiding $\sigma$.
    
    Let $\Nichols=\NicholsOf(X)$ be a finite dimensional Nichols algebra for an object $X=\bigoplus_{k=1}^n \C_{\gamma_k}$ in $\cC$ spanned by $x_1,\ldots, x_n$ in degrees $\gamma_1,\ldots,\gamma_n\in \Gamma$, so $g.x_i=\gamma_i(g)x_i$ for all $g\in C$. Assume there are corresponding grouplike elements $g_i\in C$ with $\sigma(\gamma_i,\gamma)=\gamma(g_i)$ and $\bar{g}_i\in C$ with $\sigma(\gamma,\gamma_i)=\gamma(\bar{g}_i)$. \\

    We define a Hopf algebra $B=\NicholsOf(X)\otimes C$ generated by $C$ and $\NicholsOf(X)$ and together with the commutation relation and modified coproduct and antipode
    \begin{align*}
        gx_i &= \gamma_i(g^{(1)})\cdot x_i g^{(2)}\\
        \Delta(x_i) &= g_i\otimes x_i + x_i\otimes 1\\
        S(x_i)&=-g^{-1}x_i
    \end{align*}
    Here $g^{(1)}$ and $g^{(2)}$ are defined via the co-product, that is $\Delta(g)= g^{(1)}\otimes g^{(2)}$ in Sweedler notation.
    We also define a Hopf algebra $B^*=\NicholsOf(X^*)\otimes C$ where $X^*=X$ with generators $x_i^*$ in degree $\gamma_i(S(-))$ and modified coproduct and antipode
     \begin{align*}
        gx_i^* &= \gamma_i(S(g^{(1)}))\cdot x_i^* g^{(2)}\\
        \Delta(x_i^*) &= \bar{g}_i\otimes x_i^* + x_i^*\otimes 1\\
        S(x_i^*)&=-\bar{g}^{-1}x_i^*
    \end{align*}
    %We substitute $y_i=x_i g_i^{-1}  \gamma_i(g_i)$, then 
    % \begin{align*}
    %    gy_i &= \gamma_i(S(g^{(1)}))\cdot y_i g^{(2)}\\
    %    \Delta(y_i) &= g_i^{-1}\otimes y_i + y_i\otimes 1\\
    %    S(y_i)&=-g_i y_i
    %\end{align*}   
    We finally define the Hopf algebra $U=\NicholsOf(X)\otimes \NicholsOf(X^*)\otimes C $ with all previous relations and the {linking relation} 
    \begin{align*}
    x_j^*x_i - \gamma_j(g_i) x_ix_j^* 
    &= \delta_{ij}\big(1- \bar{g}_i g_i\big)\\
    %[x_i,y_j]&= \delta_{ij} \big(\bar{g}_i-g_i^{-1}\big)
    \end{align*}

\begin{lemma}\label{lm_realizedQuantumGroup}
    The relations above define a Hopf algebra and we have an equivalence of braided tensor categories with a nondegenerate braiding
    $$\Rep^{wt}(U)\cong \YD{\Nichols}(\cC)$$
\end{lemma}
\begin{proof}
    We first check that $B, B^*$ and $U$ are Hopf algebras: For $B$ we have to check coassociativity and antipode condition on the generators $x_i$, which clearly holds, and check that $\Delta$ is compatible with the commutation relation, i.e.~$\Delta$ extended by the bialgebra axiom applied to both sides of the relation is equal up to the relation (here using commutativity and cocommutativity)
    \begin{align*}
    \Delta(gx_i)&=g^{(1)}g_i\otimes g^{(2)}x_i+g^{(1)}x_i\otimes g^{(2)} \\
    \Delta(\gamma_i(g^{(1)})x_i g^{(2)}) 
    &=\gamma_i(g^{(1)})g_i g^{(2)}\otimes x_ig^{(3)}
    +\gamma_i(g^{(1)})x_ig^{(2)}\otimes g^{(3)}
    \end{align*}
    Since the commutation relation and the coproduct relation on a product on $1$-dimensional $C$-modules in degrees $\mu,\nu$ reads 
    $$g.(x_i.m_\mu)=(g^{(1)}.x_i)(g^{(2)}.m_\mu) \qquad 
    \Delta(x_i)(m_\mu\otimes n_\nu)=(x_i.m_\mu)\otimes n_\nu+\sigma(\gamma_i,\mu) m_\mu\otimes (x_i.n_\nu)$$
    As a consequence, we have almost by definition an equivalence of tensor categories 
    $$\Rep(B)\cong \Rep(\Nichols)(\cC)$$
    The computation for $B^*$ is completely analogous. 

    For $U$ we have to check that $\Delta$ is compatible with the linking relation, because extending $\Delta$ by the bialgebra axiom to the left-hand side of the linking relation matches the general coproduct for the right hand side (a so-called trivial skew-primitive for $g=\bar{g}_jg_i$):
    \begin{align*}
    \Delta(x_j^* x_i-\gamma_j(g_i) x_ix_j^*)
    &=(\bar{g}_j \otimes x_j^*+x_j^*\otimes 1) 
    (g_i \otimes x_i+x_i\otimes 1)\\
    &-\gamma_j(g_i)(g_i \otimes x_i+x_i\otimes 1)
    (\bar{g}_j \otimes x_j^*+x_j^*\otimes 1)  \\
    &=\bar{g}_jg_i \otimes (x_j^* x_i-\gamma_j(g_i) x_ix_j^*)
    +(x_j^* x_i-\gamma_j(g_i) x_ix_j^*)\otimes 1 \\
    &+(\bar{g}_jx_i-\gamma_j(g_i)x_i\bar{g}_j)\otimes x_j^*
    +(x_j^*g_i-\gamma_j(g_i) g_ix_j^*)\otimes x_i
    \\
    \Delta(1-g)&=g\otimes (1-g)+ (1-g)\otimes 1
    \end{align*}
    
    We now turn to the asserted category equivalence: Define a functor $\YD{\Nichols}(\cC)\to \Rep(U)$ by considering $M\in \Rep(\Nichols)(\cC)$ as a $B$-module and turning the $\Nichols$-coaction $m\mapsto m^{(-1)} \otimes m^{(0)}$ into a $\NicholsOf(X^*)$-action on $M$ using the nondegenerate Hopf pairing $\langle-,-\rangle:\NicholsOf(X^*)\otimes \NicholsOf(X)\to \C$ uniquely determined by $ \langle x_j^*,x_i\rangle=\delta_{i,j}$, whose existence is one of the equivalent characterizations of a Nichols algebra \cite{Lusz93}\cite{HS20}. So in particular on generators 
    $$x_j^*.m =\langle x_j^*,m^{(0)}\rangle m^{(1)}$$
    The coaction being inside the category $\cC$ i.e.~$g.m\mapsto g^{(1)}m^{(-1)} \otimes g^{(2)}m^{(0)}$ is equivalent to the relation between $g,x_j^*$ since
    \begin{align*}
    x_j^*.g.m &=\langle x_j^*,g^{(1)}.m^{(-1)}\rangle g^{(2)}m^{(0)}\\
    &=\langle S(g^{(1)})x_j^*,m^{(-1)}\rangle g^{(2)}.m^{(0)}\\
    &=\langle \gamma_i(g^{(1)})x_j^*,m^{(-1)}\rangle g^{(2)}.m^{(0)}\\
    &=\gamma_j(g^{(1)})g^{(2)}.x_j^*.m
    \end{align*}
    The braided Yetter-Drinfeld condition in Definition \ref{def_YD} is equivalent to the linking relation between $x_i,y_j$ because the two sides of the Yetter-Drinfeld condition evaluated on $x_i\otimes m$ for an element $m$ having degree $\lambda$ and a  generator $x_i$ having the coproduct $\Delta_\NicholsOf(x_i)=1\otimes x_i+x_i\otimes 1$, and the result paired with another generator $x_j^*$, has on both sides of the equation only two contributions for $i=j$  
    \begin{align*}
        \langle x_j^*,x_i\rangle m 
        + \sigma(\gamma_i,\gamma_j) \langle x_j^*,m^{(-1)}\rangle x_i.m^{(0)}
        &= \langle x_j^*,x_i\rangle \sigma(\lambda,\gamma_i)\sigma(\gamma_i,\lambda)m 
        +  \langle x_j^*,(x_i.m)^{(-1)}\rangle (x_i.m)^{(0)}
    \end{align*}
    for braiding factors $\sigma(\gamma_i,\gamma_j)=\gamma_j(g_i)$ and $\sigma(\gamma_i,\lambda)=\lambda(g_i)$ resp.  $\sigma(\lambda,\gamma_i)=\lambda(\bar{g}_i)$ for $m$ in degree $\lambda$. Rewritten as action of $x_j^*$ this is
    \begin{align*}
        \delta_{ij} m 
        + \gamma_j(g_i) x_i.x_j^*.m
        &= \delta_{ij} \lambda(\bar{g}_i)\lambda(g_i) m 
        + x_j^*.x_i.m \\
   \end{align*}
   and this gives the asserted linking relation in $U$. We also check that this functor extends to a tensor functor: For $\Rep(\Nichols)(\cC)\to\Rep(B)$ we have already checked this. The  tensor product of $\Nichols$-comodules is defined using the algebra structure in $\Nichols$  
   $$(m_\mu\otimes n_\nu)^{(-1)}\otimes (m_\mu\otimes n_\nu)^{(0)}
   =
   m_{\mu'}^{(-1)}n_{\nu'}^{(-1)}\sigma(\mu'',\nu')\otimes (m_{\mu''}^{(0)}
   \otimes n_{\nu''}^{(0)})
   $$
   so after applying the Hopf pairing with a generator with primitive coproduct $\Delta_\NicholsOf(x_j^*)=1\otimes x_j^*+x_j^*\otimes 1$
   \begin{align*}
    x_j^* (m\otimes n)
   &=\langle x_j^*, m_{\mu'}^{(-1)}n_{\nu'}^{(-1)}\sigma(\mu'',\nu')\rangle \; (m_{\mu''}^{(0)}
   \otimes n_{\nu''}^{(0)})\\
   &=\big(\langle 1, m_{\mu'}^{(-1)}\rangle\langle x_j^*, n_{\nu'}^{(-1)}\rangle \sigma(\mu'',\nu') 
  +\langle x_j^*, m_{\mu'}^{(-1)}\rangle\langle 1, n_{\nu'}^{(-1)}\rangle \sigma(\mu'',\nu')
  \big)
     \; (m_{\mu''}^{(0)}
   \otimes n_{\nu''}^{(0)})\\
   &=  \sigma(\mu,\gamma_j) \; (m_{\mu}
   \otimes x_j^*.n_{\nu}) +\sigma(\mu-\gamma_j,0) \; (x_j^*.m_{\mu}
   \otimes n_{\nu}) \\
   &= (\bar{g}_j\otimes x_j^*+x_j^*\otimes 1 ).(m\otimes n)
   \end{align*}
    
    %Rewritten as action of $y_j$, using $g_jx_i=\gamma_i(g_j)x_ig_j=\sigma(\gamma_i,\gamma_j)x_ig_j$ and $g_j.m=\lambda(g_j)$ 
    %\begin{align*}
    %    \delta_{ij} m 
    %    + \gamma_i(g_j)\gamma_j(g_j)^{-1} x_i.y_j\lambda(g_j).m
    %    &= \delta_{ij}\lambda(g_i)\lambda(\bar{g}_i)m 
    %    %+\gamma_j(g_j)^{-1} y_jg_j.x_i.m 
    %    +\gamma_i(g_j)\gamma_j(g_j)^{-1}y_j.x_i.\lambda(g_j).m
    %\end{align*}
    % and rearranging it gives precisely the linking relation above 
    %\begin{align*}
    %[x_i,y_j]m&= \delta_{ij} \gamma_i(g_j)^{-1} \gamma_j(g_j)%\big(\lambda(g_i)\lambda(\bar{g}_i)-1\big)\lambda(g_i)^{-1}m \\
    %&= \delta_{ij} \big(\lambda(\bar{g}_i)-\lambda(g_i)^{-1}\big)m
    %\end{align*}\\

    The same reasoning gives an inverse functor $\Rep(U)\to \YD{\Nichols}(\cC)$, because the pairing $\langle -,-\rangle$ is invertible for $\Nichols$ finite-dimensional, and the relations between $x_i^ *,g$ and $x_i,x_j^*$ precisely show on generators the coaction being in the category and the Yetter-Drinfeld condition. We thus finally conclude that we have an equivalence of tensor categories.

 A concrete $R$-matrix in some closure of $U$ making this a braided tensor category can easily be read off from the braiding in $\YD{\Nichols}$.
    \end{proof}

\begin{exercise}
Apply the calculations above to the Nichols algebra $\C[x]/x^n$, realized over $C=\mathfrak{h}$ the weight lattice with braiding $q^{(\lambda,\mu)}$. This constructs the unrolled quantum group $u_q(\sl_2)$, which in particular contains the usual small quantum group. 
\end{exercise}

\chapter{Nichols algebras beyond the diagonal case}

\section{Over nonabelian groups}\label{sec_nonabelian}

We now discuss the center of $\Vect_G$ (or equivlently $\Rep(G)$) for a nonabelian group $G$, and then summarize briefly what is known about Nichols algebras in this category. 

Take an arbitrary $G$-graded vector space $V$, then a half-braiding for $V$ is a collection of morphisms
$$c_{V,-}:\; V\otimes \K_g \to \K_g \otimes V$$
We can identify this with a linear map $\rho(g):V\to V$, but in order for the morphism to preserve $G$-grading, it has to send an element in degree $h$ to an element in degree $ghg^{-1}$. The hexagon identity amounts to the fact that $\rho$ is a representation of $G$. Note that this set of data coincides with the description of a Yetter-Drinfeld module over the group algebra of $G$.

\begin{definition}[Yetter-Drinfeld modules over groups]
	For $G$ any group and $\K$ any base field, we have a braided tensor category
    $$\YD{G}=\cZ(\Vect_G)=\cZ(\Rep(G))$$
    consisting of $G$-graded vector spaces $V=\bigoplus_h V_h$ together with a $G$-action such that 
    $$g.(V_h)=V_{ghg^{-1}}$$
    The braiding $V\otimes W\to W\otimes V$ is as follows, if $v_g\in V_g$ and $w_h\in W_h$:
    $$v_g\otimes w_h \mapsto g.w_h \otimes v_g $$
\end{definition}
How do such an objects look explicitly? Surely, if it contains some $v_h\in V_h$, then it also contains elements $v_{h'}\in V_{h'}$ for all $h'$ in the conjugacy class $[h]$. On the other hand, for some fixed $h$  the subgroup $\Cent(h)=\{g\in G|ghg^{-1}=h\}$ of elements commuting with $h$ acts on $V_h$ individually. It turns out that these two ingrediences suffice to describe the situation.

\begin{lemma}\label{lm_Ogchi}
Any object in $\YD{G}$ is a direct sum of objects of the form  $\mathcal{O}_{[h]}^V$ parametrized by conjugacy classes $[h]$ and representations $V$ of the centralizer subgroup $\mathrm{Cent}(h)$. As $G$-graded vector space it is $\bigoplus_{h'\in [h]} V_{h'}$ with all $V_{h'}=V$ as vector spaces.
\end{lemma}

\begin{remark}
These objects are simple, indecomposable or projective iff the corresponding centralizer representation $V$ is simple, indecomposable or projective. In particular if $\Rep(G)$ is semisimple over the base field $\K$, then $\YD{G}$ is semisimple.
\end{remark}

\begin{proof}[Proof of Lemma \ref{lm_Ogchi}]
Suppose we have a fixed conjugacy class $[h]$ and a fixed representative $h$ and fixed $V_h$ with action of $\mathrm{Cent}(h)$. 

We recall a tool from the permutation theory of groups: Let $G$ act on a set $X$ transitively (that is, $X$ is one single orbit, that is, for every $x,y\in X$ there exists $g\in G$ with $g.x=y$). Fix some $x\in X$ and consider the subgroup $U$ of elements fixing $x$. Then there exist a set of elements $t_1,\ldots, t_n$ called coset representatives, in bijection with $X=\{x=x_1,\ldots,x_n\}$, such that any $g\in G$ can be uniquely written as $g=t_iu$ for some $i$ and $u\in U$. Explicitly, we pick elements $t_i$ sending $x=x_1$ to $x_i$ respectively, then for any given element $g$ we look to which $x_i$ it maps $x$, so it has on $x$ the same effect as our representative $t_i$. Accordingly $t_i^{-1}g.x=x$ so $u:=t_i^{-1}g$ is in the stabilizer $U$ and $g=t_iu$ as intended. Accordingly, there is a decomposition of $G$ into $U$ cosets  
$$G=\bigcup_{i=1}^n t_i U$$
More generally, if we have $g.x_i=x_j$, then $t_j^{-1}gt_i$ is in the stabilizer subgroup $U$ of $x=x_1$ by construction.

\begin{example}
An illustrative standard example is the following (for conjugacy classes see below): The group of rotations $SO(3)$ acts on the points of the sphere. The stabilizer subgroup of a point fixed point $x$ on the sphere is the subgroup of rotation around the corresponding axis through $x$. We can pick (non-canonically) for every point $x_i$ a rotation $t_i.x=x_i$, then every rotation $g$ can be written as a product $g=t_iu$ for some rotation $u$ around $x$.   
\end{example}

Now we apply this to the action of $G$ on the conjugacy class $X=[h]=\{h_1,\ldots,h_n\}$ with conjugacy class representative $h=h_1$. We pick coset representatives $t_i$ with $t_i h t_i^{-1}=h_i$. Now for a given Yetter-Drinfeld module $V=\bigoplus_{h_i\in [h]} V_{h_i}$ we use the action of $t_i$ to identify $V:=V_{h_1}$ with $V_{h_i}$, that is, we identify all $V_{h_i}=V$ such that $t_i.v_h=v_{h_i}$. Now for an arbitrary fixed $g$ and $h_i$ and $gh_ig^{-1}=h_j$ the element $u:=t_j^{-1}g t_i$ is in the centralizer of $h=h_1$ and we define the action in terms of the chosen centralizer action on $V=V_{h_1}$ by
\[ g.v_{h_i}=(u.v)_{h_j} \qedhere\]

\end{proof}

\begin{comment}
\begin{remark}\label{rem_explicitAction}
For calculations, it is important to have the action explicitly, and this depends on choices: Fix a representative $h$ of the conjugacy class $[h]$. Fix a set of coset representatives 
$$G/\Cent(h)=\cup_i t_i\Cent(h)$$
As for any permutation representation, there is a bijection between $[g]$ and $G/\Cent(h)$, explicitly $t_iht_i^{-1}$ runs through all conjugates $h_i$ of $h$. Now these $t_i$ are used to identify all $V_{h_i}$ with $V_h=V$. The action of an arbitrary element $g$ acting on an element $v_{h_i}$ in arbitrary degree $V_{h_i}$ produces an element in degree $gh_ig^{-1}=h_j$ for some $h_j$, and it is explicitly
$$g.v_{h_i}=\rho( t_j^{-1}g t_i)v_{h_j}$$
where $(V,\rho)$ is the given centralizer representation and by construction $t_jg t_i^{-1}$ is in this centralizer of $h$. Indeed $$(t_j^{-1}g t_i)h(t_j^{-1}g t_i)^{-1}
=(t_j^{-1}g)h_i(t_j^{-1}g)^{-1}
=(t_j^{-1})h_j(t_j^{-1})^{-1}
=h$$
\end{remark}
\end{comment}

\begin{example}
The symmetric group $\mathbb{S}_3$ over $\K=\C$ has simple Yetter-Drinfeld modules
of dimensions $1,1,2,3,3,2,2,2$.
\vspace{-.1cm}
$$\O_e^{\1},\;\O_e^{\mathtt{sgn}},\;\O_e^{\mathtt{std}},\;
\O_{[(12)]}^{\pm 1},\;\O_{[(123)]}^{q_3^k}$$ 

We take a closer look at $\O_{(12)}^\pm$: It is spanned by three base elements 
$$x_{(12)},\quad x_{(23)},\quad x_{(13)}$$
The centralizer subgroup of $(12)$ is $\{e,(12)\}$, so these element map $x_{(12)}$ to itself up to a factor 
$$(12).x_{(12)}=\pm x_{(12)}$$
This factor $\pm$ is the chosen centralizer representation indicated in the notation  $\O_{(12)}^\pm$. On the other hand the subsets of elements $\{(13),(123)\}$ and $\{(23),(321)\}$ conjugate $(12)$ to $(23)$ or $(13)$ respectively.\footnote{Recall that the conjugation of a conjugation $\pi$ on another conjugation $\sigma$ is simply given by applying $\pi$ to the entries of $\sigma$ in cycle notation. Otherwise, check this explicitly in these examples by multiplying permutations.} Let us choose coset representatives 
$$t_1=e,\quad t_2=(12),\quad t_3=(23)$$
then we have (better: we define our basis in a way such that)
$$(13).x_{(12)}=:x_{(23)},\quad (23).x_{(12)}=:x_{(13)}$$
The other elements in the cosets can be written with these representatives (this follows, but check explicitly) $(123)=(13)(12)$ and $(321)=(23)(12)$, hence we have 
$$(321).x_{(12)}=\pm x_{(23)},\quad (123).x_{(12)}=\pm x_{(13)}$$
The action of $g$ on an arbitrary base element  can be computed by computing  $t_j g t_i^{-1}$ (exercise).

%In characteristic $2$ or $3$ the category is nonsemisimple, with simples
%$$\O_e^{\1},\;\O_e^{\mathtt{std}},\;
%\O_{[(12)]}^{1},\;\O_{[(123)]}^{q_3^k}$$
%$$\O_e^{\1},\;\O_e^{\mathtt{sgn}},\;
%\O_{[(12)]}^{\pm 1},\;\O_{[(123)]}^{1},$$   
\end{example}
\begin{exercise}
Consider $G=\S_3$ over the field $\C$.
\begin{itemize}
\item Compute explicitly the action of $(12)$ on $\O_{[(12)]}^{\pm 1}$ with basis $x_{12},x_{23},x_{13}$. Compute explicitly the braiding between two such modules in a basis.
\item Compute the decomposition of the tensor product 
$$\O_{[(12)]}^{\pm 1}\otimes \O_{[(12)]}^{\pm 1}$$
 \item The braiding as a natural transformation acts by a scalar on each direct summand of the tensor product: compute these scalars in the previous tensor product. 
 \item If you like nonsemisimple representation theory, then determine the irreducible, indecomposable and projective Yetter Drinfeld modules over $\S_3$ in characteristic $2$ and $3$.
 \item Convince yourselves for general $G$ that the sum of squares of dimensions of $\O_{[g]}^V$ is $|G|^2$. Check for $\S_3$ in characteristic $2,3$ that the sum of the dimensions of the projectives times the dimension of the simples again gives $|G|^2$   .
\end{itemize}

\end{exercise}

Computing the Nichols algebra even of a simple Yetter-Drinfeld module over a nonabelian group is hard. The first example is contained in in \cite{MS00}: 
\begin{example}
We consider over the symmetric group $\S_n$  the Yetter Drinfeld modules $\O_{[(12)]}^{\rho}$ for $\rho$ a $1$-dimensional character of the centralizer subgroup. It has a basis $x_{(ij)}$ indexed by transpositions. The centralizer of a fixed transposition $(12)$ is  $\Z_2\times \S_{n-2}$, where $\S_{n-2}\subset \S_n$ means the stabilizer of $1,2$. Accordingly, we have four $1$-dimensional characters $(\epsilon,\mathtt{triv})$ and $(\epsilon,\mathtt{sgn})$ with $\rho((12))=\epsilon=\pm 1$. 

\bigskip

We compute the action of an $g\in G$ explicitly as in the proof of Lemma \ref{lm_Ogchi}. We clearly have 
$$g.x_{(ij)}=\chi(g,(ij))\, x_{(g(i)g(j))}$$
To compute the factor $\chi(g,(ij))$ we pick the element $(1i)(2j)$ as the coset representatives  that conjugates $(12)$ to $(ij)$ for $i<j$. Then the factor will depend on if $g(i)<g(j)$ or $g(i)>g(j)$ and is accordingly
\begin{align*}
\chi(g,(ij))&=\rho\big(((1g(i))(1g(j)))^{-1} g ((1i)(1j))\big)\\
\chi(g,(ij))&=\rho\big(((1g(j))(1g(i)))^{-1} g ((1i)(1j))\big)
\end{align*}
For the character $(1,\mathtt{triv})$ the factor is trivial. For the character $(-1,\mathtt{triv})$ the result only depends on what the argument of $\rho$ does to $1,2$, so it is $\pm 1$ depending on the case above. For the character $(-1,\mathtt{sgn})$, which is the restriction of the sign character on $\S_n$, we have $\chi(g,(ij))=\mathtt{sgn}(g)$. For the character $(1,\mathtt{sgn})$ it is the product of the two. The two middle cases are most interesting to us in what follows, so we summarize
\begin{align*}
\O^{(-1,\mathtt{triv})}_{(12)}:\qquad
\chi(g,(ij))&=\begin{cases}
+1, g(i)>g(j)\\
-1,g(i)<g(j)
\end{cases}\\
\O^{(-1,\mathtt{sgn})}_{(12)}:\qquad
\chi(g,(ij))&=\mathtt{sgn}(g)
\end{align*}
We compute the nondiagonal braiding on this vector space  explicitly, with $i,j,k,l$ distinct 
\begin{align*}
x_{(ij)}\otimes x_{(ij)}
&\mapsto 
(ij).x_{(ij)} \otimes x_{(ij)}
=-x_{(ij)} \otimes x_{(ij)}\\
x_{(ij)} \otimes x_{(kl)}
&\mapsto 
(ij).x_{(kl)} \otimes x_{(ij)}
=\pm x_{(kl)} \otimes x_{(ij)}\\
x_{(ij)} \otimes x_{(jk)}
&\mapsto 
(ij).x_{(jk)} \otimes x_{(ij)} = \pm x_{(ik)} \otimes x_{(ij)}
\end{align*}

Here the first sign is $+1$ for $\O^{(-1,\mathtt{triv})}_{(12)}$ and $-1$ for $\O^{(-1,\mathtt{sgn})}_{(12)}$. The second sign is for $\O^{(-1,\mathtt{triv})}_{(12)}$ by the formula above $+1$ for $j<k$ and $i<k$ and $-1$ for $j<k$ and $i>k$, and the second sign is for $\O^{(-1,\mathtt{sgn})}_{(12)}$ always $-1$.

\bigskip

After this preparation we can rather easily prove that the following relations hold for the Nichols algebra $\NicholsOf(\O^{(-1,\mathtt{triv})}_{(12)})$ (upper sign) and $\NicholsOf(\O^{(-1,\mathtt{sgn})}_{(12)})$ (lower sign) as the next exercise:
\begin{align*}
x_{(ij)}^2 &=0 \\
x_{(ij)}x_{(kl)} \mp x_{(kl)}x_{(ij)} &=0,\qquad \text{$i,j,k,l$ distinct}\\
x_{(ij)}x_{(jk)}\mp x_{(jk)}x_{(ik)}\mp x_{(ik)}x_{(ij)}
&=0,\qquad \text{$i<j<k$}\\
x_{(jk)}x_{(ij)}\mp x_{(ik)}x_{(jk)}\mp x_{(ij)}x_{(ik)}
&=0,\qquad \text{$i<j<k$}\\\\
\end{align*}
Note that it is not clear if this is a defining set of relations (this is known for $n=3,4,5$ by now). The algebra defined only by the quadratic relations is the \emph{Fomin-Kirillov algebra} \cite{FK99}, which has many applications. It is conjecturally infinite dimensional for $n>6$ and supposedly related to to several other phenomena, for example the preprojective algebra, Lusztis's canonical basis etc, see \cite{Ven12}.

\bigskip

For $G=\S_3$, one can check explicitly that an algebra with these relations has graded dimensions $1+3+4+3+1$ and an explicit basis\SimonRem{find good basis to see this}
$$1,\quad 
x_{(12)},x_{(13)},x_{(23)},\quad
x_{(12)}x_{(23)},x_{(13)}x_{(12)},
x_{(23)}x_{(12)},x_{(13)}x_{(23)},\quad
\cdots ,\quad
x_{(12)}x_{(23)}x_{(12)}x_{(13)}
$$
and it is indeed the Nichols algebra. The subspaces in each degree of course still carry the $G$-grading as well as the $G$-action and it makes sense to write them as a sum of $G$-Yetter Drinfeld modules \SimonRem{Unsure about middle term}
$$1\oplus \O_{[(12)]}^{-1} \oplus 2\O_{[(123)]}^{-1} \oplus 
\O_{[(12)]}^{-1} \oplus 1$$
Note the symmetry called \emph{Poincare duality}, which holds for arbitrary Hopf algebras. For all known Nichols algebras over nonabelian groups, there seems to be a curious factorization of the Hilbert series into $q$-polynomials, for example in this case 
$$1+3t+4t^2+3t^3+t^4
=(1+t)(1+t)(1+t+t^2)$$ 
\end{example}
\begin{exercise}
Prove the quadratic relations above relations by quantum symmetrizer, and by skew derivations, and by computing the coproduct.
\end{exercise}
The paper gives similar results for arbitrary Coxeter groups, below we will discuss the case of the dihedral group $\D_4$, which is the Coxeter group of type $B_2$ and also known to admit a finite Nichols algebra. 
\begin{problem}
Probably the hardest problem in the field which resist any attempts and developments since 25 years:
Decide if the Nichols algebra of $\O_{[(12)]}^{-1}$ over $\S_6$ is infinite dimensional.
\end{problem}

In general, the classification of finite-dimensional Nichols algebras of irreducible Yetter-Drinfeld modules (rank $1$) over nonabelian group is a notoriously difficult subject. There are only a handfull conjugacy classes known with a finite-dimensional Nichols algebra, and over the last $\sim 20$ years there have been powerful methods developed to rule out such conjugacy classes over most conjugacy classes in simple groups. 
A main technique is the identification of subsets of conjugates that lead to diagonal Nichols-subalgebras, which can be decided using Heckenbergers list. We will not discuss this important work here, see \cite{ACG15} and references therein. A very different development that could recently decide cases resistent to the methods mentioned is the reduction of Nichols algebras modulo $p$ \cite{HMV24}.

\bigskip

We now present an example in rank $2$: This is essentially from \cite{MS00}, but we present it in the generality in which is appears in the classification result in \cite{HS10} Section 4 where this type of group is called $G_2$ and in the classification in \cite{HV17} \SimonRem{HeckVendraminRank2 (which one?) Example 1.2}, where this type of group is called $Z^{2,2}_2$. 

\begin{example}
Consider a group $G$ generated by $g,h,\varepsilon$ with $\varepsilon$ central of order $2$ and $gh=\varepsilon hg$. Then we have conjugacy classes
$$[g]=\{g,\varepsilon g\},\qquad [h]=\{h, \varepsilon h\},\qquad
[gh]=\{gh,\varepsilon gh\}
$$
Let $\rho$ be a character of the centralizer subgroup of $g$, which is generated by $\varepsilon, g, h^2$ and $\sigma$ be a character of the centralizer subgroup of $h$, which is  generated by $\varepsilon,h,g^2$. 
\bigskip

We want to compute the Nichols algebra 
$$\NicholsOf(\O_{[g]}^\rho\oplus \O_{[h]}^\sigma)$$
\begin{exercise}
    Compute that the braiding on $\O_{[g]}^\rho$ with basis $x_g$ and $x_{\varepsilon g}:=h.x_g$ is is diagonal with the following braiding matrix
    $$\begin{pmatrix}
    \rho(g) & \rho(\varepsilon)\rho(g) \\
    \rho(\varepsilon)\rho(g) & \rho(g) \\
    \end{pmatrix}$$
    In particular if $\rho(g)=\sigma(h)=-1$, then $\NicholsOf(\O_{[g]}^\rho)$ is $4$-dimensional. Similarly for $\NicholsOf(\O_{[h]}^\sigma)$
    \bigskip
    Check if there are quadratic relations between $\NicholsOf(\O_{[g]}^\rho)$ and $\NicholsOf(\O_{[h]}^\sigma)$, now non-diagonal.
\end{exercise}

\begin{theorem}
If $\rho(g)=\sigma(h)=-1$ and $\rho(\varepsilon h^2)\sigma(\varepsilon g^2)=1$, then the Nichols algebra 
$$\NicholsOf(\O_{[g]}^\rho\oplus \O_{[h]}^\sigma)$$
has finite dimension $64=1+4+7+8+8+8+7+4+1$.

\bigskip
The Nichols algebras $\NicholsOf(\O_{[g]}^\rho)$ and  $\NicholsOf(\O_{[h]}^\sigma)$ are in the exercise above. The space of braided commutators $[\O_{[g]}^\rho,\O_{[h]}^\sigma]$ is the simple Yetter-Drinfeld module $\O_{[gh]}^{\psi}$ with 
$$\psi(\varepsilon)=\rho(\varepsilon)\sigma(\varepsilon),\qquad
\psi(gh)=-1,\qquad
\psi((gh)^2)=\rho(\varepsilon h^2)\sigma(\varepsilon g^2)
%\psi(g^2)=\sigma(g^2)
$$
\SimonRem{check last relation}
The root system is of type $A_2$ and the roots $\alpha_1,\alpha_2,\alpha_{12}$ are assigned to the three Yetter-Drinfeld modules  $\O_{[g]}^\rho,\,\NicholsOf(\O_{[h]}^\sigma),\,\O_{[gh]}^{\psi}$. 
\bigskip
There is a PBW basis in these generators. This can be phrased more categorically: As $\N$-graded Yetter Drinfeld modules we have an isomorphism
$$\NicholsOf(\O_{[g]}^\rho\oplus \O_{[h]}^\sigma)
\cong \NicholsOf(\O_{[g]}^\rho)
\otimes 
\NicholsOf(\O_{[h]}^\sigma)
\otimes 
\NicholsOf(\O_{[gh]}^{\psi})
$$
\end{theorem}
\end{example}

\begin{exercise}~
\begin{itemize}
\item Compute from the data given the action of $G$ and the diagonal braiding on $\O_{[g]}^\rho$ and on $\O_{[h]}^\sigma$
\item Compute from the data given the tensor product of $\O_{[g]}^\rho$ and $\O_{[g]}^\sigma$ and decompose it. 
\item Compute the four braided commutators in $[\O_g^\rho,\O_h^\sigma]$ and try to determine which are zero in the Nichols algebra.
\item Assume in addition that the order of $G$ is $8$, so $g^2,h^2,(gh)^2$ are either $1$ or $\varepsilon$. Identify which choices lead to the  the dihedral group and which to the quaternion group and give all possible choices for characters. 
\end{itemize}
\end{exercise}

There is a large family of examples of this type, with nonabelian group a central extension of an abelian, derived from diagonal Nichols algebras together with a diagram automorphisms, through a process called folding \cite{Len14}. This idea leads to Nichols algebras with root system $A_n,D_n,E_6,E_7,E_8$ of dimension $2^{2|\Phi^+|}$ from folding two disjoint copies of diagonal Nichols algebras with roots $\Phi^+$ and the automorphisms switching them, and to Nichols algebras with root systems $B_n,C_n,F_4,G_2$ with dimension $2^{|'\Phi^+|}$ where $'\Phi$ is the root system $D_{n+1},A_{2n-1},E_6,D_4$ with nontrivial diagram automorphisms. The case $A_2$ recovers the previous example above. In the first cases, all Yetter-Drinfeld modules attached to roots are $2$-dimensional, in the second cases, those attached to short roots correspond to orbits of length $2$ and are $2$-dimensional and those attached to long roots are $1$-dimensional
  
\SimonRem{Exercise on decomposables}

\bigskip

Going back the the example of rank $2$, now with a classificatory view, it turns out that the finiteness of the root system (to be precise: already the fact that a reflection is irreducible) is extremely restrictive in the nonabelian setting, leading to a complete classification of finite-dimensional Nichols algebras of rank $2$ in \cite{HS10} and of rank $\geq 2$ in a series of subsequent papers culminating in \cite{HV17}. Apart from the central extensions there are a few more exceptional case uncovered during the classification. 

\bigskip

\begin{example}\label{exm_newNicholsAlgebra}
	In rank 3, apart from $A_3,C_3$, there is only one finite dimensional exceptional Nichols algebra over a 
	nonabelian group (defined for any field with a third 
	root of unity or any field characteristic $3$). It is \emph{not} of 
	Lie type, but has the root system rank 3 no.\ 9 with 13 roots and root space dimensions (in the first object $a$ with $B_3$ 
	Cartan matrix, see \cite{HV17} Lemma 8.8:
	
	\begin{center} 
		\includegraphics[scale=.6]{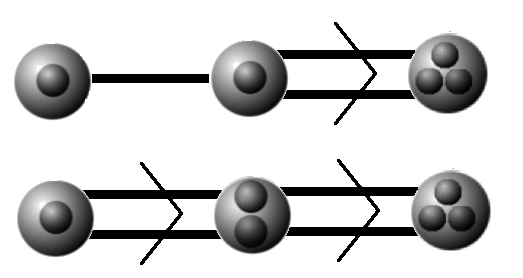}
	\end{center}
	\begin{center}
	\begin{tabular}{cc}
	$\alpha \in R^a_+$ & $\dim(M_\alpha)$\\
    
	\hline
	$1,2, 12,12^23^4,12^33^4, 1^22^33^4$ & 1 \\
	$23^2,123^2,12^23^2$ & 2\\
	$3,23,123,12^23^3$ & 3\\ 
	\end{tabular}
	\end{center}
	The rank 2 parabolic generated by $\alpha_2,\alpha_3$ is an exceptional Nichols algebra over a 
	nonabelian group and has a standard root system of type $B_2$. The three-dimensional rank 1 
	parabolics are the Nichols algebras of dimension $12$ over the nonabelian group $\S_3$.
    \end{example}

\section{Over other Nichols algebras}\label{sec:PBW}

Recall that we had introduced in Chapter \ref{chp_rootsystem} a reflection of Nichols algebras and attached a root systems. But we had not explain if and how the corresponding Nichols algebras are actually related, nor have we proven the PBW theorem. We will now at least sketch this proof, referring to \cite{HS20} Chapter 12-14. 

The main idea is in my opinion interesting in many situation beyond this particular proof: We can treat Nichols algebras in parabolic situations as Nichols algebras over other Nichols algebras. In particular it requires the Hopf algebra definitions from the previous chapter and looks somewhat similar to the Nichols algebras over nonabelian groups in the previous section.\footnote{This may be more than a parallelism...}

\bigskip

Consider some object $V=U\oplus W$ in some braided monoidal category $\cC$. Denote by $\pi_U,\pi_W$ the projections to $U,W$ and by $\iota_U,\iota_W$ the inclusions of $U,W$. Then there is a graded projection of Hopf algebras 
$$\NicholsOf(U\oplus W)\stackrel{\pi_U}{\leftrightarrows} \NicholsOf(U)$$
by the universal property. We consider the coinvariants 
$$K=\Nichols(U\oplus W)^{\coin(\pi_U)}
=\{x\in \Nichols(U\oplus W) \mid  (\id\otimes \pi_U)\Delta(x)=x\otimes 1\}$$
By the Radford projection theorem\footnote{which we invoke in $\cC$, while in \cite{HS20} it is invoked after taking a smash product} \ref{thm_RadfordProjection} we have 
$$\NicholsOf(U\oplus W) = K\rtimes \NicholsOf(U)$$
and $K$ carries the adjoint action and a coaction, and becomes with the modified coproduct a Hopf algebra in the category of Yetter-Drinfeld modules $\YD{\Nichols(U)}(\cC)$. A beautiful result is now 

\begin{theorem}[\cite{HS20} Theorem 13.2.8]
In the setup of the book (Assumption \ref{ass_HS}) 
$$K\cong \NicholsOf(\ad_{\NicholsOf(U)}(W))$$
is the Nichols algebra in $\YD{\Nichols(U)}(\cC)$ of the generated module $\ad_{\NicholsOf(U)}(W)$ under the adjoint action. 
\end{theorem}
\begin{remark}
It seems this proof is again true in any braided tensor category. 
\end{remark}

We slightly rewrite this to highlight the parabolic situation: Let $X=\bigoplus_{i\in I} X_i$ with $X_i$ simple. Let $J\subset \{1,\ldots,n \}$ be a subset of simple roots and in the previous notation take $U=X_J=\bigoplus_{j\in J} X_j$ and $W=\bigoplus_{i\not\in J}$. Then our previous assertion says 
        $$\NicholsOf(X)\cong \NicholsOf(X_J)\rtimes \NicholsOf(X_{I\backslash J}),\qquad
       X_{I\backslash J}:=\bigoplus_{i\not\in J} X_{i+J},\qquad  X_{i+J}=\ad_{\NicholsOf(X_J)}(X_i) 
        $$

\begin{example}[$A_2$ over $A_1$]
Consider the diagonal Nichols algebra $\NicholsOf(x_1\C\oplus x_2\C)$ with the following braiding matrix, which we recognize to be $u_q^+(\sl_3)$:  
$$(q_{ij})_{ij}=\begin{pmatrix}
    q^2 & q^{-1} \\q^{-1} & q^2
\end{pmatrix}$$
As we saw in Rosso's lemma \ref{lm_Rosso}, this Nichols algebra fulfills the quantum Serre relations
\begin{align*}
(\ad_{x_1})^2(x_2) &=0 \\
(\ad_{x_2})^2(x_1) &=0 \\
\end{align*}
Let $\pi_1=\pi_U$ be the projection with setting $x_2$ to zero, the kernel is the ideal generated by $x_2$. The coinvariants $K$ contain besides $x_2$ by the adjoint action also the  element 
$$x_{12}:=\ad_{x_1}(x_2)=x_1x_2-q_{12}x_2x_1$$
This fact can be also checked directly:
\begin{align*}
\Delta(x_1x_2-q_{12}x_2x_1)
&=x_1x_2\otimes 1+ x_1\otimes x_2+q_{12} x_2\otimes x_1 + 1\otimes x_1x_2 \\
&-q_{12}\big(x_2x_1\otimes 1+ x_2\otimes x_1+q_{21} x_1\otimes x_2 + 1\otimes x_2x_1\big)\\
&=\ad_{x_1}(x_2) \otimes 1 + (1-q_{12}q_{21})(x_1\otimes x_2) + 1\otimes \ad_{x_1}(x_2) 
\end{align*}
as the middle term (and of course the last term) vanishes if $\pi_1$ is applied to the second tensor factor. The coaction of $x_{12}$ in the Radford projection theorem is given by applying $\pi$ on the first factor, so 
$$\delta(x_{12})=(1-q_{12}q_{21})(x_1\otimes x_2) + 1\otimes \ad_{x_1}(x_2) $$

We now compute the nondiagonal braiding matrix on this $2$-dimensional object in $\YD{\Nichols(U)}(\cC)$ from action and coaction
\begin{align*}
\begin{pmatrix}
x_2\otimes x_2 & x_2 \otimes x_{12} \\
 &\\
x_{12} \otimes x_2 & x_{12}\otimes x_{12}
\end{pmatrix}
\longmapsto 
\begin{pmatrix}
q_{22}(x_2\otimes x_2) & q_{21}q_{22}(x_{12} \otimes x_{2})\\
 &\\
q_{12}q_{22} (x_{2} \otimes x_{12})
& q_{11}q_{12}q_{21}q_{22}(x_{12}\otimes x_{12})\\
+q_{22}(1-q_{12}q_{21})(\ad_{x_1}(x_2) \otimes x_2)
&
+q_{21}q_{22}\cancel{(\ad_{x_1}(x_{12})}x_{12}\otimes x_{2})
\end{pmatrix}
\end{align*}
A nontrivial check of the assertion above is that this Nichols algebra reproduces the second quantum Serre relation (note that it contains the $q$-commutator $x_{21}=[x_2,x_1]_q$ which is not proportional to $x_{12}$)). We can check by matching coefficients that it can be rewritten in the generators $x_2,x_{12}$ by $q_{12}q_{21}q_{22}=1$ as follows
\begin{align*}
    [x_2,[x_2,x_1]_q]_q 
    &=x_2 (x_2x_1-q_{21}x_1x_2)-q_{22}q_{21} (x_2x_1-q_{21}x_1x_2)x_2\\
    &=
    -q_{12}^{-1}x_2(x_1x_2-q_{12}x_2x_1)
    +q_{22}q_{21}^2(x_1x_2-q_{12}x_2x_1)x_2\\
    &=-q_{22}q_{21}\big(x_2x_{12}-q_{21}x_{12}x_2\big)
\end{align*}
\SimonRem{find relation directly?}
\end{example}

\begin{remark}
We can ask what the root system of $\NicholsOf(X_{I\backslash J})$ is. In \cite{CL17} we have shown that this can be described in terms of restriction of hyperplane arrangements. 

Combinatorially it says plainly: Take the set of all roots $\Phi^+$ for $X$, remove all roots only containing simple roots $\alpha_j,j\in J$, and consider on the other roots an equivalent relation given removing the simple roots $\alpha_j,j\in J$ from the linear combinations. The size of an equivalence classes $\alpha+J$ is the size of the corresponding root space $(X_{I\backslash J})_{\alpha+J}$ in the PBW decomposition. Sometimes the resulting root system is not reduced, that is, $2(\alpha+J)$ can appear as a root. 

In our initial example, the root system of type $A_2$ restricts to a root system of type $A_1$ with a $2$-dimensional root space
\begin{center}
    \includegraphics[scale=.7]{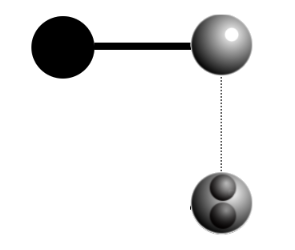}
\end{center}
In general the resulting root system can be a generalized root system, even if the initial root system is from a Lie algebra, compare \cite{DF24} For example in \cite{CL17} Example 4.19 we compute the restriction of $E_7$ with respect to a parabolic subalgebra $A_1\times A_1\times A_2$:
\begin{center}
    \includegraphics[scale=.7]{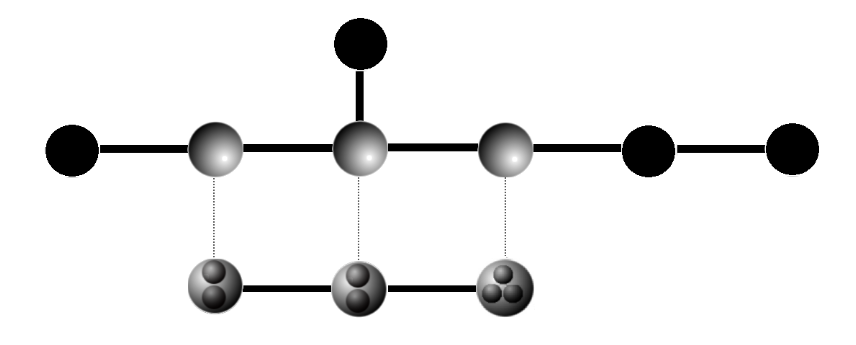}
\end{center}
Altogether, the restriction of the Weyl group arrangement $E_7$ with $63$ roots has a Weyl groupoid of sporadic type $7$ and $13$ roots (one of which is nonreduced). Namely, taking the roots of $E_7$ and only considering the entries $\alpha_3,\alpha_4,\alpha_5$ we get besides $(0,0,0)$\\

\begin{center}
\begin{tabular}{l|lllllllllllll}
$\Phi^+ +J$ 				&	
$(1,0,0)$ & $(0,1,0)$	&$(0,0,1)$ &	$(1,1,0)$ & 
$(0,1,1)$	&
$(1,1,1)$ & $+2(1,1,1)$	\\
multiplicity 		& $2$							
& $2$							& $3$				
			&	$4$				& $6$ 			
& 
$12$ & +3\\
\hline
$\Phi^+ +J$ 				& $(1,2,1)$	& $(2,2,1)$	& 
$(1,2,2)$	& $(2,3,2)$	& $(2,3,3)$	& $(2,4,3)$&$(3,4,3)$\\
multiplicity 		&	$6$				& $3$			
	& $6$				& $6$				& $2$		
		& $1$ 			& $2$\\
\end{tabular}
\end{center}~\\
And we can clearly find that the Dynkin diagram in this basis is type $A_3$.
After the reflections $s_1s_2$, the Dynkin diagram is different, because the underlying parabolic has changed  
\begin{center}
    \includegraphics[scale=.7]{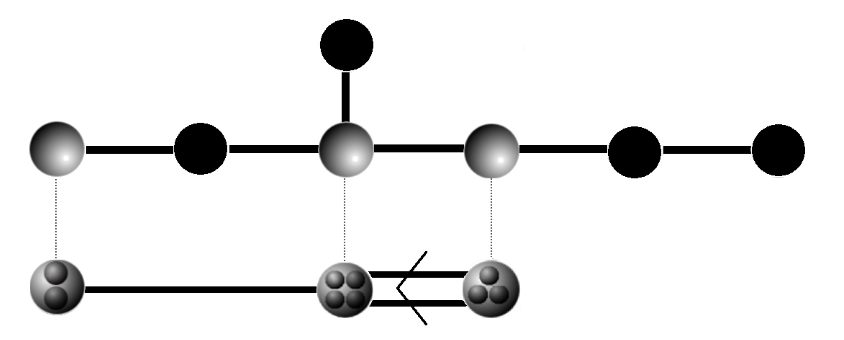}
\end{center}
We have checked in this article that in fact all generalized root systems of rank $\geq 3$ appear in this way.

\end{remark}
\begin{exercise}
Consider the root system $B_2,G_2$ and $J=\{\alpha_1\}$ or $J=\{\alpha_2\}$. Compute the restricted root system. 

Consider the root system $A_{n+1+m}$ and $J=\{n+1\}$. Compute the restricted root system. 
\end{exercise}

We finally discuss the significance of this construction to the general theory of root system and PBW basis in \cite{HS13}. 
\begin{itemize}
    \item Consider the parabolic situation above for $J=\{j\}$, that is
    $$\NicholsOf(X)\cong  \NicholsOf(X_{I\backslash J})\rtimes
    \NicholsOf(X_j),\qquad
       X_{I\backslash J}:=\bigoplus_{i\neq j} X_{i+J}
       \;\in\; \YD{\NicholsOf(X_j)}(\cC)
        $$
    \item Recall that by one of the equivalent definitions, there is a nondegenerate Hopf pairing between $\NicholsOf(X_j)$ and $\Nichols(X_j^*)$. Using the dualization functor $\Omega$ for Yetter-Drinfeld modules in Theorem \ref{thm_DualizeYD}, we can map each summand  $X_{i+J}$ and the entire $X_{I\backslash J}$ to Yetter Drinfeld modules $\Omega(X_{i+J}),\Omega(X_{I\backslash J})$ over $\Nichols(X_j^*)$.
    Since this is a braided equivalence of categories, it maps the rank 1 Nichols algebras $\NicholsOf(X_{i+J}$ and the rank $n-1$ Nichols algebra $\NicholsOf(X_{I\backslash J})$ to the corresponding Nichols algebras. 

    Glueing back together we define the \emph{reflected Nichols algebra}
    $$R_j(\NicholsOf(X))=\Omega(\NicholsOf(X_{I\backslash J}))\rtimes \NicholsOf(X_j^*).$$
    Since the dualization functor switches action and coaction, it switches highest and lowest weight vectors. Hence indeed this is the Nichols algebra of the reflected Yetter Drinfeld module  defined previously
$$R_j(X):= \ad^{-\cm_{i1}}(X_j)(X_1)\oplus \cdots \oplus
X_j^*\oplus \ad^{-\cm_{in}}(X_j)(X_n)$$  
\item We now get a proven relation between this two Nichols algebras. For example, they have the same dimension if finite. Moreover, by construction and the categorical view, their Drinfeld centers are equivalent, 
$$\YD{\Nichols(X)}
\cong \YD{\NicholsOf(X_{I\backslash J})\rtimes \NicholsOf(X_j)}
\cong \YD{\NicholsOf(X_{I\backslash J})}\YD{\NicholsOf(X_j)}
\cong \YD{\Omega(\NicholsOf(X_{I\backslash J}))}\YD{\NicholsOf(X_j^*)}
\cong \YD{\Nichols(R_j(X))}
$$
Remark: This is a $2$-categorical version of a Morita equivalence, and accordingly there is an invertible bimodule category and an equivalence of module categories over the categories of representations. 
\end{itemize}

A main technical assertions is that $\ad^{-\cm_{i1}}(X_j)(X_1)$ is again irreducible for finite $\cm_{i1}$. Again, this should be true in arbitrary categories, and improved arguments can be found in \cite{Wo23}.

\bigskip

Finally, the decomposition can be iterated to get a flag of coideal subalgebras in $\NicholsOf(X)$ and prove the PBW theorem \ref{thm_PBW}. In this way \cite{HS13} also proved a substantial conjecture, already for quantum groups, that the set of homogeneous coideal subalgebras is in bijection to the elements of the Weyl group.

\section{Towards new examples}

We now list two cases beyond the scope of \cite{HS20} that we find  interesting to study in the future. We try to give the natural background as well as specific calculable examples. Both types are very interesting from a CFT perspective:

\bigskip

The following is called a $G$-graded extension of $\cC$: Let $\cC$ be a braided tensor category with action of a group $G$, then there exists by \cite{ENOM10} (up to a certain obstruction and a certain cohomological choice) a $G$-graded tensor category $\cC=\bigoplus_{g\in G} \cC_g$ with $\cC_e=\cC$ with $G$-action such that $g.\cC_g=\cC_{hgh^{-1}}$ (this is a $2$-categorical Yetter-Drinfeld condition) and a crossed braiding, and the category of $G$-equivariant objects is a braided tensor category. 

\begin{example}[see e.g. \cite{GLM25} Section 3]
Let $\Gamma$ be a finite abelian group of odd order and $G=\Z_2$ acting by $-1$. Then the extension of $\cC_e=\Vect_\Gamma^Q$ has $\cC_g=\Vect$ with a unique simple object $X$ such that $\C_a\otimes X=X$ for any $a\in \Gamma$ and $X\otimes X=\bigoplus_{g\in G} \C_g$. There are nontrivial associators present.   
\end{example}

If we compute the Nichols algebra of an object already present in $\cC_e$, then the Nichols algebra as an algebra does not change, it is only upgraded to a Hopf algebra in the equivariantization and we get essentially the folding construction briefly mentioned in Section \ref{sec_nonabelian}.

\begin{problem}
Compute the Nichols algebra of objects in $\cC_g$. For example, compute the Nichols algebra of $X$ and more generally of any $X\boxtimes \C_x$ in $\cC\boxtimes\Vect_{\Gamma'}$ for some other finite group (to be more flexible in choices of braidings). In general: Classify Nichols algebras in the twisted sectors of $G$-graded extensions
\end{problem}

Another very important source of semisimple finite braided monoidal categories without any forgetful functor are the representations of affine Lie algebras $\hat{\g}_\kappa$ at positive integer level. Algebraically, these categories can be obtained from $\Rep(u_q(\g))$ by a process called semisimplification. \SimonRem{Source}

\begin{example}[$(\hat{\sl}_2)_\kappa$]
A braided fusion category $\cC$ is said to be of type $A(\ell)$ if it has $\ell$ inequivalent simple objects $X_0=1$ and $X_1,\dots, X_{\ell-1}$ with 
$$X_1 \otimes X_i \cong X_{i-1} \oplus X_{i+1} \qquad \text{for}\qquad  1 \leq  i  \leq \ell-2$$
$$X_1 \otimes X_{\ell-1} \cong X_{\ell-2}$$
Note that these are $\sl_2$ fusion rules truncated at $n$. The associators and braiding are explicitly classified in \cites{FK93,CM}.
\end{example}
\begin{problem}
Compute the Nichols algebras over objects in such categories. An obvious source of examples is to take the parabolic situation $\g\subset \g'$ in Section \ref{sec_parabolic}, then $(u_q(\g')^+)^{\coin(\pi)}$ is a Nichols algebra over $u_q(\g)$. After applying semisimplification to $u_q(\g)$ this gives a braided Hopf algebra in $\cC$, usually the corresponding Nichols algebra (give better criteria when this is the case). The problem is if there are more finite Nichols algebras in $\cC$. 

For example, applying this to $\sl_3$ or $\sl(2|1)$  with $q$ an $\ell$-th root of unity, with parabolic $\sl_2$, leads to Nichols algebras of the image of the standard representation $X_1$ in the category of type $A(\ell)$.
\end{problem}
    
\chapter{Reconstruction results}

At this point I also want to talk about classification and reconstruction results for Hopf algebras and tensor categories from Nichols algebras. This is not the main topic of this course, but I want to at least summarize, the reader is advised to the survey \cite{AG19}. I add an own categorical version of these results in Theorem \ref{thm_AScat}. After some digression on commutative algebras in categories I also add some own reconstruction result on tensor categories which are braided, in terms of Drinfeld centers of Nichols algebras.

%\begin{slogan}
%    In good cases, every Hopf algebra is a deformation of a Nichols algebra over its (co)semisimple part,\\
%    and every braided tensor category is a  
%\end{slogan}

\section{The Andruskiewitsch Schneider program}

A major achievement and still today the main reason to use Nichols algebras is the Andruskiewitsch-Schneider  program for the classification of finite-dimensional pointed Hopf algebras over abelian groups (in the category of vector spaces) in \cite{AS10}. 

The idea is as follows:
\begin{itemize}
\item To every Hopf algebra $H$ there is associated a \emph{coradical filtration}, dual to the Jacobson radical filtration of algebras 
$$H_0\subset H_1\subset \cdots$$
Here $H_0$ is the coradical, the direct sum of all simple subcoalgebras, and inductively 
$$H_i=\Delta^{-1}(H_1\otimes H_{i-1}
+H_{i-1}\otimes H_1)$$
\item \emph{Pointed} means $H_0$ is a group ring $\K[\Gamma]$, i.e. all simple subcoalgebras are $1$-dimensional $g\K$. Then $H_1$ is the space of skew-primitive elements. In this case the coradical filtration is a Hopf algebra filtration. 
\item We consider the \emph{graded Hopf algebra} $\mathrm{gr}(H)$ with respect to the filtration, i.e. remove from multiplication $H_i\cdot H_j$ all terms in lower degree $<i+j$. 
\item The Hopf subalgebra generated as an algebra by $H_0$ and $H_1$ in $\mathrm{gr}(H)$  is then by construction the Radford biproduct of $\K[G]$ and the Nichols algebra of $H_1/H_0$ (as explicit representatives, take all skew primitives $\Delta(x)=g\otimes x+$).
\end{itemize}
Thus the classification problem turns into  three separate problems:
\begin{itemize}
\item What are are finite-dimensional Nichols algebras over $G$? 
\item Can we exclude the possibility of further generators beyond $H_0,H_1$?
\item What are the possible deformations or liftings, that is, what are the possible $H$ for $\mathrm{gr}(H)$ given as above as a Radford biproduct of a Nichols algebra with $\K[G]$ above. These are two separate points: Are there deformations of the Nichols algebra itself as a Hopf algebra in the braided tensor category, and are there deformations of the Radford biproduct.  
\end{itemize}

\begin{example}
The small quantum group $H=u_q(\g)$ has generators $E_i,F_i,K_i$, compare Section \ref{sec_quantumgroups}. The coradical $H_0$ is generated by the group elements $K_i$, the first degree component $H_1$ is generated by the $E_i,F_i$. The relation which is not graded is $[E,F]=(K-K^{-1})/(q-q^{-1})$. The graded algebra is then the Radford biproduct of $H_0$ with two commuting Nichols algebras $u_q^+$ and $u_q^-$.
\end{example}
\begin{example}
There are more exotic linkings of several Borel parts.
\SimonRem{details}
\end{example}
Note that by construction as a Drinfeld double/center, the particular form of the quantum group containing two copies of a Borel and a nongraded relation comes with a braiding on its category of representations. The classification we discuss right now ignores this additional wish. We will discuss reconstruction results involving the braiding in Section \ref{sec_Schauenburg}
\begin{problem}
    Which of the pointed Hopf algebra produced by the Andruskiewitsch-Schneider program admits a nondegenerate braiding? (one might expect: essentially only Drinfeld double)
\end{problem}

For $G$ abelian not involving small primes $2,3,5,7$ this program has been completed in \cite{AS10}. The reason in hindsight behind the primes restriction is that in these cases the only Nichols algebras are those familiar from quantum groups. 

%, for a current state see the survey \cite{AG19} and \cites{An13,AG11}

To deal with the small primes, Heckenberger started his classification of arbitrary diagonal Nichols algebras \cite{Heck09}. Angiono has computed relations, and it turns out there are exceptions for small primes, even for usual  quantum group, see \cite{An13}. See also the comprehensive survey \cite{AA17}.

Angiono gave in \cite{AKM15} a useful criterion for the deformability of the Nichols algebra and of the Radford biproduct: A relation $R$ can be deformed if there is a skew-primitive element $x$ which is isomorphic as an object (i.e. same $G$-grading and $G$-action). By checking all relations of all Nichols algebras it is shown in this article that there are no deformations of the Nichols algebra itself.   
Interestingly, the problem of having additional generators in higher degrees is almost equivalent, see \cite{AG11}.

\begin{example}
Take $X=\C_g^\chi$ with generator $x$ realized over the group $G=\Z_n$ with generator $g$. Suppose $q=\chi(g)$ is a primitive $n$-th root of unity, then the Nichols algebra is $\C[x]/x^n$. More precisely, the relation $x^n$ is a primitive element in degree $\C_{g^n}^{\chi^ n}$, meaning in the Radford biproduct (of the tensor algebra with the group algebra) a relation 
$$\Delta(x^n)=g^n\otimes x^n+x^n\otimes 1$$
This means as a deformed relation we may identify this element with a multiple of $g^n-1$, which has the same coproduct, iff $\chi^n=1$.
\end{example}
\begin{example}
There are examples of Hopf algebras with deformations of the Nichols algebra by itself, and examples with generators in higher degree, for infinite dimensional Hopf algebras: Let us again consider the smallest case $\cC=\Vect_\Gamma^Q$ and $X=x\C_a$ with self-braiding $x\otimes x\to -(x\otimes x)$, so the Nichols algebra is $\C[x]/x^2$.
\begin{itemize}
\item $H=\C[x]$ is also a Hopf algebra, but it is not the Nichols algebra, since it has the additional primitive element $y=x^2$. We now apply the Andruskiewitsch-Schneider program to this Hopf algebra: In the coradical filtration both $x,y$ have degree $1$. Hence the graded algebra is $$\gr(H)=\C[x,y]/x^2$$
This is the Nichols algebra of $x,y$ with braiding matrix $\begin{psmallmatrix}
-1& 1 \\ 1 & 1 
\end{psmallmatrix}$. The deformation is changing the relation to $x^2=y$.
\item Let us extend $\C[x]/x^2$ to $\C[x,z]/x^2$ by a generator $z$ with coproduct 
$$\Delta(z)=z\otimes 1+x\otimes x+1\otimes z$$
This generator is in the coradical filtration is degree $2$, but it is not a square of $x$.
\end{itemize}
Both of these constructions exist for more general Nichols algebras and also for arbitrary quantum groups $u_q(\g)$ and in this case they are known as the \emph{quantum group with big center} (or unrestricted specialization or Kac-DeConcini-Procesi quantum group) respectively as the \emph{quantum group of divided powers} (or restricted specialization or Lusztig quantum group). In the first case, the powers $E_i^n,F_i^n,K_i^n-1$ (or similar, depending on the situation) are nonzero new generators that generate a big central subalgebra, which can be identified with the ring of functions of the Poisson dual Lie group, so this can be understood as fibring over this Lie group, with the zero fibre being the small quantum group. In the second case, we leave the truncated $E_i^n,F_i^n,K_i^n-1=0$ but add additional "renormalized" divided powers $E_i^{(n)}=E^n/[n]_q$\footnote{Mathematically spoken, we compute in the quantum group with generic $q$, where such expressions are rigorous, and then prove that the divided powers form a subalgebra with structure coefficients only polynomials in $q$, so an integral form, so $q$ can be specialized.}. In the example above $z=E^{(2)}$ and when we compute the coproduct of $z$, the $[2]_q$ in the denominator cancels exactly the vanishing coefficient of the mixed term $$\Delta(E^2)=E^2\otimes 1+[2]_q(E\otimes E)+1\otimes E^2$$
 \end{example}

In the \emph{lifting program}, see the survey \cite{AG19}, all possible liftings of all diagonal Nichols algebras were computed. There are quite exceptional cases according to exceptional defining relations in the Nichols algebra 
\begin{example}
In $u_q(A_3)^+$ at $q^2=-1$ the quantum Serre relations are implied but there is a new relation $$[x_2,[x_1,[x_3,x_2]_q]_q]_q=0$$ 
Correspondingly, there is a nontrivial deformation of this relation, see \cite{AAG17} Proposition 4.8.
\end{example}

\begin{problem}
The queer and periplectic Lie super algebra seem not to correspond to Nichols algebras. However, the concrete presentations of explicit corresponding quantum super groups in literature has to fall into the classification program. I would expect they correspond to a single Borel part and an exotic lifting. For example, a Borel part of type $A(n|n)$ with linking relations $[x_i,x_{\sigma(i)}]$ according to the diagram symmetry, respectively in the second case with a nontrivial lift of the relation $[x_2,[x_1,[x_3,x_2]]]$   
\end{problem}

The case of a nonabelian group will be discussed in Section \ref{sec_nonabelian}

\SimonRem{Question of cocycle deformation}

%\section{and a categorical version}

%Tannaka-Krein reconstruction, which was originally devised for algebraic groups, states in the version in \cite{Schau91} that any tensor category $\cB$ with a faithful exact tensor functor $F:\cB\to \Vect$ is equivalent to the category of representations of a Hopf algebra $\Nichols\in\Vect$ and $F$ is the functor forgetting the action. In \cite{LM24} we give the following relative version for fibre functors $F:\cB\to \cC$.  In a Hopf algebra setting it reduces to a version of the Radford projection theorem 

In view of the categorical version of the Radford projection theorem \ref{thm_Mombelli}, let me also mention a categorical version that uses precisely the output of the Andruskiewisch-Schneider to determine the representing Hopf algebra: 

%From there it is not difficult to see that $\Nichols$ contains primitive elements for every $X_i$ with $\Ext^1_\cB(1,X_i)\neq 0$. Using the results from the Andruskiewitsch Schneider program it can even be shown that (for the part of $\cB$ generated by $\cC$ as abelian category) this Hopf algebra  $\Nichols$ is actually the Nichols algebra $\NicholsOf(X)$

\begin{theorem}[\cite{Len25}]\label{thm_AScat}
Let $\cB$ be a finite and rigid monoidal category and $\cC\subset \cB$ be a central braided monoidal category, equivalent to $\mathrm{Vect}_\Gamma^Q$,
such that all simple object in $\cB$ are in $\cC$. Then $\cB\cong \CoRep(\Nichols)$ with $\Nichols$ the Nichols algebra of the following $\cC$-object $$\bigoplus_{a\in\Gamma}\mathrm{Ext}^1_{\cB}(\C_a,1)\C_a$$
\end{theorem}

\section{Categorical commutative algebras}\label{sec_CommutativeAlgebra}

 Let $A$ be an algebra in a tensor category $\cD$, then there is a straightforward notion of a category of $A$-modules, which we call $\cD_A$, and a tensor category of $A$-$A$-bimodules ${_A}\cD_A$ with the tensor product $\otimes_A$. If $A$ is a commutative algebra in a braided tensor category $\cD$, then as in classical algebra every $A$-module can be turned into an $A$-$A$-bimodule, and $\cD_A$ becomes a tensor category with $\otimes_A$. 
 However, in the braided case there are two inequivalent choices for such an identification, using over- or underbraiding. In turn, the tensor category $\cD_A$ itself is not braided, but there is a braided tensor subcategory $\cD_A^\loc$ consisting of \emph{local $A$-modules} $M$, which are modules such that the following diagram commutes:
\begin{equation*}
\begin{split}
\xymatrix{
A \otimes  M    \ar[rd]_{  \rho_M }  \ar[rr]^{ c_{M, A} \,\circ\,  c_{A, M}} && A\otimes M \ar[ld]^{\rho_M} \\
& M &  \\
} 
\end{split}
\end{equation*}
This appears early in \cites{Par95,KO02}, see \cites{FFRS06,SY24} for key properties of these categories.

We have an adjoint pair of functors: Left-adjoint the right-exact induction functor $A\otimes(-):\,\cD\to \cD_A$, which is a monoidal functor, and right-adjoint the left-exact functor forgetting the $A$-action $\mathrm{forget}_A:\,\cD_A\to \cD$, which is a lax monoidal functor. %All claims clear
We summarize all structures in the following diagram:
$$
%\hspace{0.8cm}
 \begin{tikzcd}[row sep=10ex, column sep=15ex]
   {\cD} \arrow[shift left=1]{r}{A\otimes(-)}  
   & 
   \arrow[shift left=1]{l}{\mathrm{forget}_A} 
   \cB  
   %\arrow[equal]{r} &     
   =\cD_A \\
   &\arrow[hookrightarrow, shift left=6]{u}{}
   %\arrow[below, dotted, shift left=2]{ul}{} %{\coVerma}
   \cC 
   %\arrow[equal]{r}& 
   =\cD_A^\loc
\end{tikzcd}
$$

 \begin{example}
     In any $\cZ_\cC(\cB)$ there is a commutative algebra $A$  (the image of the tensor unit of $\cB$ under the induction functor in Section \ref{sec_quantumgroups}), such that the category of $A$-modules is  $\cB$ and the category of local $A$-modules is $\cC$. 

$$
%\hspace{0.8cm}
 \begin{tikzcd}[row sep=10ex, column sep=15ex]
   \cZ_\cC(\cB) \arrow[shift left=1]{r}{\mathrm{forget}_{c_{X,-}}}  
   & 
   \arrow[shift left=1]{l}{} 
   %\arrow[equal]{r} &     
   \cB\\
   &\arrow[hookrightarrow, shift left=0]{u}{}
   %\arrow[below, dotted, shift left=2]{ul}{} %{\coVerma}
   %\arrow[equal]{r}& 
   \cC
\end{tikzcd}
$$

For the small quantum group we have the following picture. Note that induction and restriction are reversed compared to commutative algebras, and that the identification of $\Rep(A),\otimes_A$ with representations over a Hopf subalgebra $u_q(\g)$ is not obvious at all, but holds in a general Hopf algebras setting \cites{Tak79, Skry07}

$$
%\hspace{0.8cm}
 \begin{tikzcd}[row sep=10ex, column sep=15ex]
   \Rep(u_q(\g))
   \arrow[shift left=1]{r}{\mathrm{res}}  
   & 
   \arrow[shift left=1]{l}{\mathrm{ind}} 
   %\arrow[equal]{r} &     
   \Rep(u_q(\g)^{\geq 0})\\
   &\arrow[hookrightarrow, shift left=0]{u}{}
   %\arrow[below, dotted, shift left=2]{ul}{} %{\coVerma}
   %\arrow[equal]{r}& 
   \Rep(u_q(\g)^{0})
\end{tikzcd}
$$

\end{example}

\section{Schauenburg functor}\label{sec_Schauenburg}

As a converse of this construction, we reconstructing the horizontal arrow for an arbitrary commutative algebra, by further developing in \cite{CLR23} Section 3 a functor devised by P. Schauenburg \cite{Sch01}.

Consider the induction functor 
$$\cD \stackrel{A\otimes(-)}{\longrightarrow}\cD_A $$
It is an important observation that the image actually comes with a half-braiding, essentially because
$$
X\otimes_A (A\otimes V)
\cong X\otimes V
\qquad
(A\otimes V)\otimes_A X
\cong V\otimes X
$$
with $A$ acting just on the first resp. second factor, see \cite{CLR23}. Hence this functor upgrades to a functor to the center, and in fact to the relative center
$$\cD \stackrel{A\otimes(-)}{\longrightarrow}\cZ_\cD(\cD_A)$$
Quite surprisingly, this often gives a reconstruction of $\cD$ in terms of $\cD_A$:

\begin{theorem}\label{thm_Schauenburg}
Let $\cD$ be a finite rigid braided tensor 
category and  $A$ be a haploid commutative algebra. Then the upgraded induction functor  is an equivalence of braided monoidal categories
$$\cD \cong\cZ_\cD(\cD_A)$$
\end{theorem}

As a consequence of this statement together with the relative Tannaka-Krein reconstruction in Theorem \ref{thm_Mombelli}, the category of representations of a small quantum group representation can be uniquely characterized by the fact that they contain a commutative algebra $A$ (the dual Verma module of weight zero) and the category of $A$-modules contains as simple objects only local $A$-modules. Here, the category of local $A$-modules serves as the Cartan part $\cC$ and the $\mathrm{Ext}^1(1,X)$ determine the braided object generating the Nichols algebra.

\chapter{Application to analysis and conformal field theory}

I now want to give a very brief sketch regarding the appearance of quantum groups on the analysis and conformal field theory side. 

To make this more digestible, I first discuss appearance of Nichols algebras in complex analysis and partial differential equations (PDE). The main general idea is that solutions of PDEs are typically multivalued around singularities, and analytic continuation produces another solution, so there is an action of the braid group on the space of solutions of the PDE. I focus on a particular PDE called Knizhnik-Zamolochikov equation for a given Lie algebra $\g$, where the action of the braid group is related to the braiding of the quantum group $U_q(\g)$ by the celebrated Drinfeld-Kohno theorem. I also discuss the toy case of an abelian Lie algebra $\h$, in which everything is very explicit and the braiding is diagonal.

Secondly, we discuss how Nichols algebras appear  on this context. Our main point is that they can be used to pass from the abelian case $\h$ to the semisimple $\g$ with Cartan subalgebra $\h$. This becomes clear by work of Varchenko, and I also mention own work in this matter. Explicitly, there are certain integrals whose relations match the Nichols algebra relations.

Third, I briefly sketch conformal field theory and vertex algebras: In a nutshell, a quantum field theory assigns to a suitable manifold a vector space of states, typically a space of solutions of a differential equation. There is also an associated structure called vertex algebra, which is in some sense a commutative ring in an analytic setting, and which has a braided tensor category of representations. There is also a topological field theory picture completing this.

Fourth, I want to make two conjectures about the appearance of Nichols algebras. They are now theorems our motivating example, the logarithmic Kazhdan-Lusztig conjecture, in which one wants to construct vertex algebras equivalent to the nonsemisimple category of representations of the small quantum group. but Nichols algebras point us to a zoo of new examples. 
\begin{itemize}
\item Screening operators over a vertex algebra generate precisely the Nichols algebra in the corresponding representation category. 
\item The kernel of screening operators give a new vertex algebra, whose category or representations should be the generalized quantum group $\YD{\NicholsOf}(\Rep(\cV))$. 
\end{itemize}
I would not expect one can grasp all the different theories touched in this chapter at once and the cross connections between them (some are indeed unclear I think). Rather one should get an impression what the main players in these theories are, and see on our repeated examples how every time the same special functions, the same braiding, the same categories etc. appear.

\section{Knizhnik Zamolodchikov equation and quantum groups}

As main source I want to point the reader to the excellent MIT lecture notes of Varchenko 
\cite{Var98}, some of the material is treated in more detail in \cite{Kas97}. 

Let $V_{1},\ldots,V_{n}$ be irreducible representations of $\g$. Let $\Sigma$ be a complex curve or Riemann surface, with distinct punctures $z_1,\ldots,z_n$.

\begin{remark}
As a spoiler, the reader may have in mind a particular conformal field theory (Wess-Zumino-Witten model) based on a gauge group $G$ on $\Sigma$ and in each puncture $z_i$ some particle in the multiplet $V_i$. 
\end{remark}

\begin{center}
\tikzset{pics/.cd,
	handle/.style={code={
			\draw[fill=gray!10]  (-2,0) coordinate (-left) 
			to [out=260, in=60] (-3,-2) 
			to [out=240, in=110] (-3,-4) coordinate (-k)
			to [out=290,in=180] (0,-6) coordinate (-num)
			to [out=0,in=250] (3,-4) 
			to [out=70,in=300] (3,-2) 
			to [out=120,in=280] (2,0)  coordinate (-right);
			\pgfgettransformentries{\tmpa}{\tmpb}{\tmp}{\tmp}{\tmp}{\tmp}
			\pgfmathsetmacro{\myrot}{-atan2(\tmpb,\tmpa)}
			\draw[rotate around={\myrot:(0,-2.5)}] (-1.2,-2.4) to[bend right]  (1.2,-2.4);
			\draw[fill=white,rotate around={\myrot:(0,-2.5)}] (-1,-2.5) to[bend right]coordinate[pos=0.5] (-B) coordinate[pos=1] (-D) (1,-2.5) 
			to[bend right] coordinate [pos=1] (-E) coordinate[pos=0.7] (-A) coordinate[pos=0.4] (-C) (-1,-2.5);
}}}

	\fontsize{8}{8}\selectfont
	\begin{tikzpicture}[scale=0.7]
	\newcommand\shift{0}
	
	\pic[rotate=-90,scale=0.4] (tr) at (0-\shift,0) {handle};
	\pic[rotate=90,scale=0.4] (tl) at (0-\shift,0) {handle};
	\draw[blue,rotate=-270, %rho's
	fill=lightgray] (1+\shift,-2.5)  circle (10pt) node {$V_1$} ;
	\draw[blue,rotate=-270, %rho's
	fill=lightgray] (0+\shift,-3)  circle (10pt) node {$V_2$} ;
\end{tikzpicture}
\normalsize
\end{center}

We now consider functions $\Psi$ on the moduli space of such $\Sigma$, taking values in $V_1\otimes\cdots\otimes V_n$. For example, in genus $0$ we may simply take $\Psi(z_1,\ldots,z_n)$ to be a function in $n$ complex variables taking values in $V_1\otimes\cdots\otimes V_n$ with possible singularities in $z_i=z_j$. In genus $1$ the function $\Psi$ also depends on the complex structure on the torus $\C/(\Z+\tau\Z)$ and is hence a vector valued modular form or Jacobi form.  

\bigskip

For any choice $\kappa\in\R$ called \emph{shifted level} we consider 
the \emph{Knishnik-Zomolochikov equation}:
\begin{align}\label{exm_KZequation}
    \left(\kappa\frac{\partial}{\partial z_i} 
 - \sum_{j\neq i}\frac{1}{z_i-z_j}\Omega_{ij}\right)
 \Psi(z_1,\ldots,z_n) =0,
\end{align}
where $\Omega_{ij}$ is the Casimir element of $\g$ acting on the tensor factors $V_i$ and $V_j$. In the example $\sl_2$ it is $\frac12 H\otimes H+E\otimes F+F\otimes E$.

Then the space of solutions of this differential equation has a nontrivial monodromy around the singularities $z_i=z_j$, which can be turned into a nontrivial \emph{braiding morphism}
$$V_i\otimes V_j \longrightarrow V_j\otimes V_i,$$
\begin{theorem}[Drinfeld-Kohno 1987, \cites{Koh88,Drin89}]
For generic $\kappa$, the monodromy of the KZ differential equation around he singularity $z_i=z_j$ 
coincides with the braiding of the quantum group $U_q(\sl_2)$ at $q=e^{\frac{\pi\i}{\kappa}}$.
\end{theorem}

(note that for generic $\kappa$ the representations of $U_q(\g)$ coincides with the representations of $\g$ as abelian category). Actually, Drinfeld invented quantum groups precisely for this purpose. 
%
%Drinfeld actually defined the \emph{quantum group} $U_q(\g)$ as a deformation of $U(\g)$ by a formal parameter $\smash{q=e^{\frac{\pi\i}{\kappa}}}$. Its representation theory coincides with the representation theory of $\g$, and it also admits a commutative tensor product $V\otimes W$, but the braiding morphism is nontrivial.
%This leads to the axiomatization of a \emph{braided tensor category}, a representation theory with a commutaive tensor product and good properties, even if there is no underlying (symmetry-) group anymore. 

\begin{example}[Abelian Knizhnik Zamolochikov equation]
Let $\h=\C^n$ with trivial Lie bracket and a fixed inner product. The irreducible representations $V_i$ are $1$-dimensional and given by weights $\lambda_i:\h\to \C$. The function $\Psi$ takes also values in a $1$-dimensional vector space, so it is a scalar function. In genus $0$ the KZ equation reads:
\begin{align}\label{exm_abKZequation}
    \left(\kappa\frac{\partial}{\partial z_i} 
 - \sum_{j\neq i}\frac{(\lambda_i,\lambda_j)}{z_i-z_j}\right)
 \Psi(z_1,\ldots,z_n) =0,
\end{align}
The space of solutions is $1$-dimensional and clearly spanned by 
\begin{align*}
\Psi(z_1,\ldots,z_n)
&=\prod_{i<j} (z_i-z_j)^{(\lambda_i,\lambda_j)/\kappa}
\end{align*}
Nevertheless, since this function is multivalued in general, this $1$-dimensional space carries and braid group representation, which is obviousely given by the diagonal braiding
$$V_i\otimes V_j\stackrel{q_{ij}}{\longrightarrow}V_j\otimes V_i,\qquad 
q=e^{\pi\i(\lambda_i,\lambda_j)}$$
Let us also make some remarks on the genus $1$ case, which should related to modular forms: Indeed, solutions for the KZ equation are here given by shifted inverse Dedekind eta functions, which in some sense can be pieced together to a vector valued modular form with continuously many entries. \SimonRem{Reference} With a slight variation in the setup (so-called lattice compactification), the solution space is spanned by Jacobi theta functions associated to the lattice cosets, which can be pieced together to a vector valued modular form with respect to the Weil representation of the modular group. With punctures, we get the classical Jacobi forms for a lattice depending on $z_1,\ldots,z_n$.  
\end{example}
\SimonRem{Wikipedia: BPZ equation for Gaudin model}
\SimonRem{Quote: Finkelberg Schechtmann / Kapranov Schechtmann}

\section{Nichols algebra and Varchenko-Selberg integrals}\label{sec_Sel}

\newcommand{\Fp}{\mathrm{Fp}}
\newcommand{\rFp}{\mathrm{rFp}}
\newcommand{\Beta}{\mathrm{B}}

\begin{slogan}
Can we solve the KZ equation for a finite-dimensional semisimple Lie algebra $\g$ using the explcit solutions of the KZ equation for $\h$ being the Cartan subalgebra?     
\end{slogan}

This is explained in \cite{Var98} starting Section 1. We first observe in general that acting with $x\in \g$ on a solution produces a new solution, because $[\Omega,(x\otimes 1+1\otimes x)]=0$.

We want to solve the KZ equation on the highest weight vectors in the decomposition of $V_1\otimes\cdots\otimes V_n$ into irreducible modules. Of course these are not preserved by $\Omega_{ij}$, but one can proceed inductively: 

Consider first the "highest" highest weight vector $v_{\lambda_1}\otimes \cdots \otimes v_{\lambda_n}$. Here $\Omega_{ij}$ acts indeed simply by the scalar $\lambda_i/2$, hence there is a solution 

$$\Psi(z_1,\ldots,z_n)
=(v_{\lambda_1}\otimes \cdots \otimes v_{\lambda_n})\prod_{i<j} (z_i-z_j)^{(\lambda_i,\lambda_j)/2\kappa}$$

Now for $\g=\sl_2$ the "next" highest weight vectors are of the form 
$$\sum_k c_k(v_1\otimes\cdots \otimes F.v_k\otimes \cdots \otimes v_n)$$
with $\sum_k c_k\lambda_k=0$ (Exercise). Acting on this with $\Omega_{ij}=\frac12 H\otimes H+E\otimes F+F\otimes E$ produces some terms involving the previous highest weight vector. As discussed in \cite{Var98} Section 2, this leads to solutions of the KZ equation:

\begin{align*}
\Phi(z_1,\ldots,z_n)
&=\sum_k c_k(v_1\otimes\cdots \otimes F.v_k\otimes \cdots \otimes v_n)\int_\gamma \frac{\d t}{t-z_k}\prod_{i<j} (z_i-z_j)^{(\lambda_i,\lambda_j)/2\kappa}\prod_i (t-z_i)^{-\lambda_i}
\end{align*}

Again we can use the action of $\g$ to obtain more solutions. Finally, there is a third case we not discuss here, which shows further hyperplane arrangements.   \\

\bigskip

I now want to add the following formal statement from my work involving directly Nichols algebras, also in cases where $\kappa$ is not generic. It it surely closely related to the observations above, as the same integrals appear, but the explicit connections and consequences have not been worked out. I will use this statement in the next section.

\begin{problem}
    Use the result below in the case of a quantum group of arbitrary rank with generic $q$ to produce an action of the Nichols algebra on the solution spaces. It would be beautiful if this could be further used to simplify the proof of the Drinfeld-Kohno theorem. 
\end{problem}

\bigskip

Given a set of of complex parameters $m_i,m_{ij}\in \C$ for $1\leq i<j\leq n$ and .\\
Consider this $n$-fold contour integral of a multivalued analytic function	
$$\Fp((m_{ij}),(m_i)):=\int\cdots\int_{[e^0,e^{2\pi\i}]^n} \prod_i t_i^{m_i}\prod_{i<j}(t_i-t_j)^{m_{ij}}\; \mathrm{d} t_1\ldots \mathrm{d} t_n$$
Convergence if data $(m_{ij})$ is \emph{subpolar}: For all  $J\subseteq \{1,\ldots,n\}$ with $|J|\geq 2$
$$\sum_{1\leq i<j\leq n} \mathfrak{Re}(m_{ij}) > -|J|+1$$

\begin{theorem}[L. 2017]
A linear combination in $\Fp(m_{ij},m_i)$ is zero if the resp. element in the Nichols algebra $\B(V)$ is zero, where $V=\langle x_1,\ldots,x_n\rangle$ and $q_{ij}=e^{\pi\i\;m_{ij}}$.
\end{theorem}
The idea is to decompose $[0,2\pi\i]^n$ into $n!$ simplices with a fixed order $t_{\sigma(1)}<\cdots <t_{\sigma(n)}$ and use analytic continuation to rewrite the integrals over each of these simplices  in terms of the integral $\rFp$ over the standard simplex $t_1<\cdots < t_n$. The analytic continuation turns out to produce precisely the braiding terms of $q_{ij}$ in the Matsumoto section and hence this decomposition is a quantum symmetrizer formula for $\Fp$ in terms of $\rFp$. Note that as for Nichols algebras the point is \emph{not} that $\Fp$ itself is braided symmetric, this is not true to choices involved in lifting the integration contours.

\begin{example}[Quadratic relations] In degree $2$ we can explicitly calculate $\Fp$ using the  Euler Beta function $\Beta$: 
\begin{align*}
\Fp(m_1,m_2,m_{12})
&=\frac{e^{2\pi\i m_2}-1}{2\pi\i}\frac{e^{2\pi\i m_1+2\pi\i m_{12}}-1}{2\pi\i}
\;\frac{1}{m_1+m_2+m_{12}+2}\cdot\\
&\hspace{-1cm}\cdot \left(\Beta(m_2+1,m_{12}+1)+\frac{\sin\pi m_1}{\sin\pi(m_1+m_{12})}\Beta(m_1+1,m_{12}+1)\right)\\
\vphantom{\sum^{x^{x^{x^x}}}}\Fp(m_1,m_2,m_{12})
&=\rFp(m_1,m_2,m_{12})+e^{\pi\i m_{12}}\cdot \rFp(m_2,m_1,m_{12})
\end{align*}
We can see the two types of quadratic Nichols algebra relations:
\begin{itemize}
\item for $m_{12}$ an integer:  
$\Fp(m_1,m_2,m_{12})-(-1)^{m_{12}}\Fp(m_2,m_1,m_{12})=0$
%\zem_{\alpha_1}\zem_{\alpha_2}\mp 
%\zem_{\alpha_2}\zem_{\alpha_1}=0$.
\item for $m_1=m_2$ and for $m_{12}$ an odd integer: $\Fp(m_1,m_1,m_{12})=0$. 
\end{itemize}
\medskip
The argument fails, where the formula has poles:
\begin{itemize}
\item $m_{12}\in-\mathbb{N}$, then the Nichols relation fails, we get an extension by a Lie algebra. 
\item $m_1+m_2+m_{12}+2=0$, corresponding to Nichols algebra reflection.
\end{itemize}
\end{example}

\section{Conformal field theory and vertex algebras}

Any kind of quantum field theory (and I am not sure there is a completely satisfying axiomatization) has to assign to any suitable manifold $\Sigma$ with possible punctures some vector space $\cZ(\Sigma)$. For a conformal $2$-dimensional quantum field theory, our input are surfaces with conformal structure. As discussed in the first section, this means the elements in $\cZ(\Sigma)$ are for $\Sigma$ of fixed genus $g$ and varying punctures $z_1,\ldots,z_n$ functions in these variables. Similary for $\Sigma$ of fixed genus $1$ they are Jacobi forms.

\bigskip

We have seen in the KZ equation some $\cZ(\Sigma)$ defined as explicit solutions spaces of a PDE. Let us give a very rough idea how elements in these spaces concretely arise in physics. We play a game called path integral, despite severe issues of  mathematical rigor:

Let $\varphi$ be a random function from $\Sigma$ to some $V$ (imagine a path, a field, or a string). Random means with respect to a probability measure $P(\phi)=|A(\phi)|^2$ via the complex amplitude 
$$A(\phi)=e^{i/\hbar \int_\Sigma \cL(\phi(z),\phi_i(z),\ldots) \d z}$$
where we "judge" a function $\phi$ by some given local property $\cL$ called \emph{Langrangian} depending on its value $\phi(z)$ and its partial derivations 
$\phi_i(z)=\frac{\partial}{\partial z_i} \phi$,  summed over all points - think about the integral in the exponential rewritten as a continuous product over all point, expressing independent  probabilities in each point.

For $\hbar\to 0$ this game becomes extreme: Only the "best" $\phi$, those with a minimal value of the exponent, have a nonzero probability of appearing. This leads to a variational problem for $\phi$, which is solved by the Euler-Langrange equation differential equation. Recall that in the 18th century it was realized that classical mechanics and field theory have a variational formulation. Hence, using the a classical Langrangian $L$ we can play a game that for $\hbar$ reduces to the deterministic choice of $\phi$ prescribed by classical physics. For finite $\hbar$ on the other hand we imagine that all $\phi$ are happening simultaneously, but they are weighted by $P$, meaning how far they are away from the classical solution. In this prohabilistic game we can ask questions like: What is the expectation value $\langle \phi(z)\rangle$ for fixed $z$? What is the correlation $\langle \phi(z_1)\phi(z_2)\rangle$ for fixed $z_1,z_2$?

\begin{center}
\includegraphics[scale=.3]{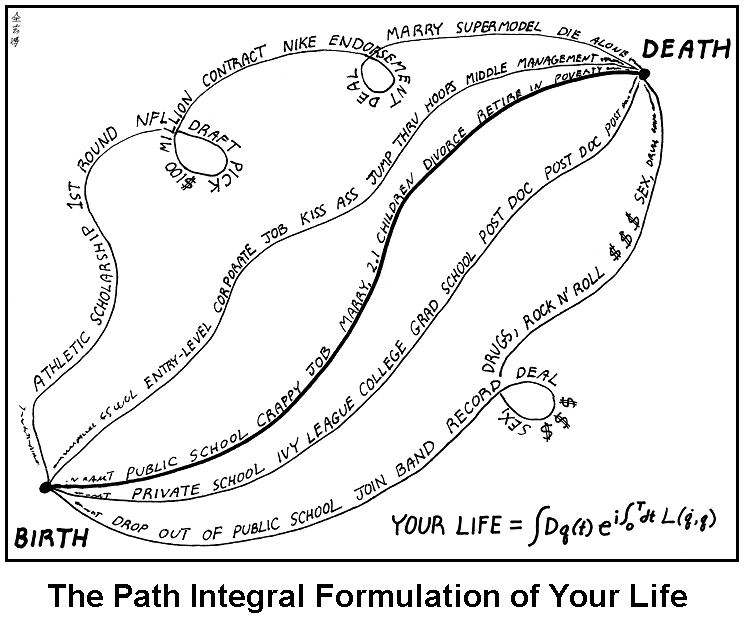}
\end{center}

\begin{example}[Free scalar field]
Consider a scalar field $\varphi:\R^d\to \R$ with the following Langrangian called free: $\int_{\R^d} |\nabla\varphi(z)|^2$. The corresponding Euler-Langrange equation is 
$$\Delta \varphi=0$$
(depending on the context this means: in mechanics free motion $x(t)''=0$, in stationary field theory the potential equation $\Delta \varphi=0$, in dynamical field theory the wave equation $\Box \varphi=0$. We could then add potential to $\cL$.).
\medskip
In quantum field theory some $n$-point correlators can be computed explicitly explicitly:
\begin{align*}
    \langle \varphi(z) \rangle &=0 \\
\langle \varphi(z_1)\varphi(z_2) \rangle  &\sim 
\begin{cases}
|z_1-z_2|^{2-d},\; & {d>2} \\
\mathrm{log}|z_1-z_2|,\; & {d=2} 
\end{cases}\\
\langle \varphi(z_1)\cdots \varphi(z_n) \rangle  &= \text{products of $2$-point correlators}
\intertext{For later use and to connect to previous formulas also we give for $d=2$}
\langle e^{\lambda_1\varphi(z_1)}\cdots e^{\lambda_n\varphi(z_n)} \rangle  &= \prod_{i<j} |z_i-z_j|^{\lambda_i\lambda_j}
\end{align*}
\footnote{There is also a mathematical rigourous formulation in the free case: Make an ansatz for $\phi$ in Fourier expansion and consider the Fourier modes as independent random variables.}  
\end{example}
\SimonRem{BPS Equation}

To arrive at representation theory we have to go some more steps: 
\begin{itemize}
\item 
For $2$-dimensional field theories, we may identify $\R^2=\C$ and write all functions as sums of products of holomorphic and antiholomorphic functions, for example $|z_i-z_j|^{\lambda_i\lambda_j}=(z_i-z_j)^{\lambda_i\lambda_j}(\bar{z}_i-\bar{z}_j)^{\lambda_i\lambda_j}$. In this way, any $2$-dimensional conformal field theory is glued from two \emph{chiral} (or holomorphic) conformal field theories\SimonRem{This glueing construction can be done purely in categorical terms \cite{Schweigert}. Glueing two different theories is called \emph{heterotic}, eg. heterotic string theories.}.
The price is that the chiral correlation functions are multivalued functions. Hence a chiral conformal field theory is something assigning to a complex curve a \emph{space of conformal blocks} consisting of multivalued functions and an action of the mapping class group on it, for example in genus $0$ an action of the braid group. 
\item 
For a $2$-dimensional chiral conformal field theory we can look at the vector space $V$ assigned to the punctured disc and the chiral correlation functions on the $3$-punctured sphere or pair-of-pants with insertions $0,z,\infty$, which gives a map 
$$\Y:\; \cV\otimes \cV\to \cV[[z]]$$
The tuple $(\cV,\Y)$ is called a \emph{vertex algebra}, and it turns out it is sufficient to recover the chiral conformal blocks of any surface by glueing. The additional conformal symmetry we assume corresponds to an action of the Virasoro Lie algebra on $\cV$. 

As an algebraist, I imagine $\Y$ as a multiplication depending on $z$. Physical locality means the multiplication is commutative  up to terms looking like formal delta-functions. Similarly one could define modules 
$$\cV\otimes M\to M[[z]]$$
and intertwiners (algebraically: balanced bilinear maps)
$$\cV\otimes M\to M\{z\}[\log(z)]$$
depending on $z$ as a multivalued function with a regular singularity in $z=0$. With the algebra intuition, it is quite suggestive that the notion of a balanced bilinear map implies a notion of a tensor product $M\otimes_\cV N$.

The main difference to the algebraic picture comes with the braiding: While for a commutative ring simply swapping gives another balanced bilinear map, for the vertex algebra we have to replace $z\mapsto -z$ and get a nontrivial braiding. Since the function is multivalued we have to specify how we do this (say counterclockwise) and the braiding is non-symmetric.

Altogether, the category of representations of a suitable vertex algebra gives a braided tensor category \cites{MS88, HLZ06}. 
\end{itemize}

\begin{remark}
Topological field theories (if we restrict ourselves to the $2$-dimensional part, also called modular functor) differ from chiral conformal field theory in that they only depend on $\Sigma$ as a topological space, but they retain the information about the abstract vector space and the mapping class group action. The Reshetikhin-Turaev construction produces a TFT from a modular tensor category, and more generally Lyubaschenko'S construction produces from a nonsemisimple modular tensor category a modular functor. Supposedly, starting with a conformal field theory and forgetting to the topological field theory (better: modular functor) gives the same result as going to the vertex algebra, then to the category of representations, and from there via Lyubaschenko. 
\end{remark}

\begin{example}[Free field theory, Heisenberg vertex algebra]
In a free field theory discussed above everything are functionals (or: observables) depending on the random scalar field $\phi$.
\begin{itemize}
\item The vertex algebra $V$ consists of 
polynomials in $\partial^n\varphi$, which have single-valued chiral correlators among themselves.
For example we already saw $\langle 1(z_1)\partial\varphi(z_2)\partial \varphi(z_3)\rangle = (z_2-z_3)^{-2}$
\item For any fixed $\lambda\in \C^n$ we have a module $\cV_\lambda$ consisting of $e^{a\varphi}$ times polynomials in $\partial^n\varphi$, which still have single-valued chiral correlators with $V$
\begin{align*}
\cV\otimes \cV_a &\longrightarrow   \cV_a[[z]]\\
\partial\varphi \otimes e^{\lambda\varphi} &\longmapsto \lambda z^{-1}e^{\lambda\varphi} +\cdots 
\end{align*}
\item Among each other, these elements have multivalued chiral correlators already appearing several times above for $z_1-z_2=z$
\begin{align*}
\cV_\lambda\otimes \cV_\mu &\longrightarrow   \cV_{\lambda+\mu}\{z\}[\log(z)]\\
e^{\lambda\varphi} \otimes e^{\mu\varphi}&\longmapsto z^{\lambda\mu}e^{\lambda+\mu}+\cdots
\end{align*}
This gives (as universal object admitting such an intertwiner) a tensor product
$$\cV_\lambda\otimes_\cV \cV_\mu=\cV_{\lambda+\mu}$$
\item The braiding  based on replacing $z$ by $-z$ and thus the multivaluedness around $z=0$ is then clearly given by $$\cV_\lambda\otimes \cV_\mu \stackrel{e^{\i\pi{\lambda\mu}}}{\longrightarrow} \cV_\mu\otimes \cV_\lambda $$
\item The associated braided tensor category of representations is $\Vect_{\C^n}^\sigma$ with trivial associator and $\sigma(\lambda,\mu)=e^{\pi\i(\lambda,\mu)}$.
\end{itemize}
\end{example}

\begin{example}[Wess-Zumino-Witten Model, Affine Lie algebra]

There is a conformal field theory built on functions $\varphi:\Sigma\to G$ for some Lie group and some level $\kappa$, with a certain Langrangian. The associated vertex algebra is the affine Lie algebra $\hat{g}_\kappa$. Concretely $\cV$ is the Verma module of weight zero and $\Y(a_{-1},z)=\sum_{n} a_n z^{-n-1}$ where $a\in \g,a_n\in \hat{g}$. In a different interpretation, the $z$-dependence is related to the underlying commutative ring of functions on the imaginary roots. The category of representations of this vertex algebra is essentially the category of representations of the affine Lie algebra, which in turn is by the celebrated Kazhdan-Lusztig correspondence the category of representations of $U_q(\g)$ at generic $q$. This is in some sense the categorical version of the Drinfeld-Kohno theorem. As a topological field theory there is a $3$-dimensional part which physically is the Chern-Simon theory.  
\end{example}

\section{Nichols algebra and screening operators}

A reasonable question would be, in alignment with the corresponding question for the KZ equation: 
    Can we produce vertex algebras associated to $\g$ (affine Lie algebra, WZW model) from vertex algebras associated to the Cartan subalgebra $\h$ (Heisenberg vertex algebra, free field theory)?

\bigskip

Another reasonable question is:
    Can we produce vertex algebras associated to small quantum group, or to Drinfeld centers of more general Nichols algebras? Since these have nonsemisimple representation categories, and the corresponding correlators have logarithmic singularities\footnote{To be precise, the logarithmic singularity is is due to nondiagonal action of the ribbon element on a projective module} this is called \emph{Logarithmic Kazhdan Lusztig conjecture} \cite{FGST05}.

\bigskip

A \emph{screening operator} $\zem_\cM$ for a choice of module $\cN$ acting on a module $\cN$ is given, in my perspective, as a left multiplication: Taking the intertwining operator, which exists by the very definition of the tensor product 
$$\cM\otimes \cN \to (\cM\otimes_\cV \cN)\{z\}[\log(z)]$$
In the examples below we fix an element $M\in \cM$, a highest weight vector (otherwise we keep the dependency on $M$), then we have 
$$\cN \to (\cM\otimes_\cV \cN)\{z\}[\log(z)]$$
To remove the $z$-dependency we integrate over a  circle around $z=0$.
$$\cM\otimes \cN \to (\cM\otimes_\cV \cN)$$
If $\cM=\cV$ then this is simply the residue, and $\zem_\cM$ is an endomorphism of $\cN$. The standard OPE calculus for vertex algebras imply under certain conditions ($M$ primary of conformal dimension $1$) a compatibility of $\zem_\cM$ with the Virasoro algebra action and  for different screening operators $\zem_{\cM_1},\zem_{\cM_2}$ a relation for the commutator, so we have a Lie algebra acting. We call this a \emph{local screening operator} and they appear in abundance in conformal field theory literature since Dotsenko \cite{DF84}.

If $\cM\neq \cV$ then we integrate a multivalued function over a circle, which we need to lift to the covering, and the result is much less pleasant. In particular all non-integral $z$-powers have a nonzero contribution and the result is in the algebraic closure. Also the OPE formalism starts to break down. We call such \emph{nonlocal screening operators}, and they appear for example in the Felder complex. What kind of algebra do they generate?

\begin{theorem}[L. 17]
For a free field theory, any set of screening operators the screening operators $\zem_{\cV_{\alpha_i}}$ with $M_i=e^{\alpha_i}$ fulfilling again subpolarity conditions (see above) generate the Nichols algebra with braiding matrix $q_{ij}=e^{\pi i (\alpha_i,\alpha_j)}$, which is the braiding for $\cV_{\alpha_i}$.
\bigskip
The proof essentially boils down to the analytic statement about Selberg-Varchenko integrals in Section \ref{sec_Sel}. We conjecture this is true for arbitrary vertex algebras. In general the nonlocal screening operators are not compatible with the Virasoro action (they are obviousely not even preserving the grading) but we conjecture that suitable monomials of them, corresponding precisely to the reflection on the Nichols algebra side, in fact do; for $\sl_2$ this is the content of the Felder complex. If subpolarity is not fulfilled we expect extensions of Nichols algebras by a Lie algebra made of pole terms.  
\end{theorem}

Given a set of screening operators $\zem_{\cM_1},\ldots,\zem_{\cM_n}$, then we can consider the kernel of screening operators $\cW\subset \cV$. My version of the logarithmic Kazhdan Lusztig conjecture is

\begin{conjecture}
Under suitable assumptions the category of representations of $\cW$ is given by the generalized quantum group $\YD{\NicholsOf}(\Rep(\cV))$.
\end{conjecture}

For example, if we choose $\alpha_1,\ldots,\alpha_n$ a rescaled root lattice, then this predicts a construction of a vertex algebra with category of representation given by a small quantum group. In this case we could solve the conjecture.

\begin{question}
Do this for Nichols algebras beyond the quantum group case, and do this for Nichols algebras in the categories associated to other interesting vertex algebras, such as affine Lie algebras or orbifold models.
\end{question}

\section{Kapranov-Schechtmann Equivalence}

I want to mention another very intriguing relation \cites{KS20,CER25} between Nichols algebras over $\S_n$ and quantities associated to configuration spaces of points, namely cohomology of certain perverse sheaves. I hope to include more on this in the future. 

\SimonRem{Read file:///C:/Users/fmnv131/Desktop/CarnovaleDobbiaco2025.pdf}

\SimonRem{Check again for base field characteristic and $ass_HS$}

%\chapter{References}


\begin{references}
Some notable original papers are \cite{Nichols78} (first example), \cites{Wor89, Ros98} (quantum shuffle approach), \cite{Lusz93} (systematic introduction for quantum groups), \cite{AS10} (key part in the classification of Hopf algebras) and \cite{Heck09} (classification of diagonal algebras).
\end{references}

\begin{bibdiv}
\begin{biblist}

% Historics
%%%%%%%%%%%%%%%%%%%%%%%%%%%%%%%%%%%%%%%%%%%%%%%%%%%%%%%
%\bib{Lie1888}{other}{
%    author={Lie, Sophus},
%    title={Theorie der Transformationsgruppen},
%    date={1888},
%    label={Lie1888},
%}
\bib{Gauss1808}{book}{ %q-symbol
      author={Gau\ss, Carl Friedrich},
      title={Summatio quarumdam serierum singularium},
      language={Latin},
      publisher={Dieterich},
      place={Gottingae},
      date={1808},
      note={Exhibita Societati Regiae Scientiarum Gottingensis d. XXIV August. 1808; 40 pp.},
    }
%%%%%%%%%%%%%%%%%%
\bib{Koh88}{other}{
  author={Kohno, T.},
  title={Quantized enveloping algebras and monodromy of braid groups},
  date={1988},
}
\bib{Drin89}{article}{
   author={Drinfel\cprime d, V. G.},
   title={Quasi-Hopf algebras},
   language={Russian},
   journal={Algebra i Analiz},
   volume={1},
   date={1989},
   number={6},
   pages={114--148},
   issn={0234-0852},
   translation={
      journal={Leningrad Math. J.},
      volume={1},
      date={1990},
      number={6},
      pages={1419--1457},
      issn={1048-9924},
   },
   review={\MR{1047964}},
}
\bib{Lusz84}{book}{
   author={Lusztig, George},
   title={Characters of reductive groups over a finite field},
   series={Annals of Mathematics Studies},
   volume={107},
   publisher={Princeton University Press, Princeton, NJ},
   date={1984},
   pages={xxi+384},
   isbn={0-691-08350-9},
   isbn={0-691-08351-7},
   review={\MR{0742472}},
   doi={10.1515/9781400881772},
}
\bib{Lusz89}{article}{
   author={Lusztig, G.},
   title={Modular representations and quantum groups},
   conference={
      title={Classical groups and related topics},
      address={Beijing},
      date={1987},
   },
   book={
      series={Contemp. Math.},
      volume={82},
      publisher={Amer. Math. Soc., Providence, RI},
   },
   isbn={0-8218-5089-X},
   date={1989},
   pages={59--77},
   review={\MR{0982278}},
   doi={10.1090/conm/082/982278},
}

% Section Hopf Cat
%%%%%%%%%%%%%%%%%%%%%%%%%%%%%%%%%%%%%%%%%%%%%%%%%%%%%%%


\bib{EGNO15}{book}{
   author={Etingof, Pavel},
   author={Gelaki, Shlomo},
   author={Nikshych, Dmitri},
   author={Ostrik, Victor},
   title={Tensor categories},
   series={Mathematical Surveys and Monographs},
   volume={205},
   publisher={American Mathematical Society, Providence, RI},
   date={2015},
   pages={xvi+343},
   isbn={978-1-4704-2024-6},
   review={\MR{3242743}},
   doi={10.1090/surv/205},
}
%VectG
\bib{MacL52}{article}{
   author={MacLane, Saunders},
   title={Cohomology theory of Abelian groups},
   conference={
      title={Proceedings of the International Congress of Mathematicians,
      Cambridge, Mass., 1950, vol. 2},
   },
   book={
      publisher={Amer. Math. Soc., Providence, RI},
   },
   date={1952},
   pages={8--14},
   review={\MR{0045115}},
}
\bib{JS93}{article}{
   author={Joyal, Andr\'e},
   author={Street, Ross},
   title={Braided tensor categories},
   journal={Adv. Math.},
   volume={102},
   date={1993},
   number={1},
   pages={20--78},
   issn={0001-8708},
   review={\MR{1250465}},
   doi={10.1006/aima.1993.1055},
}
%Rad85
\bib{Maj95}{book}{
   author={Majid, Shahn},
   title={Foundations of quantum group theory},
   publisher={Cambridge University Press, Cambridge},
   date={1995},
   pages={x+607},
   isbn={0-521-46032-8},
   review={\MR{1381692}},
   doi={10.1017/CBO9780511613104},
}
%Hopf Algebra
\bib{Mon93}{book}{
   author={Montgomery, Susan},
   title={Hopf algebras and their actions on rings},
   series={CBMS Regional Conference Series in Mathematics},
   volume={82},
   publisher={Conference Board of the Mathematical Sciences, Washington, DC;
   by the American Mathematical Society, Providence, RI},
   date={1993},
   pages={xiv+238},
   isbn={0-8218-0738-2},
   review={\MR{1243637}},
   doi={10.1090/cbms/082},
}
\bib{Kas97}{article}{
   author={Kassel, Ch.},
   title={Quantum groups},
   language={Spanish, with English summary},
   conference={
      title={Algebra and operator theory},
      address={Tashkent},
      date={1997},
   },
   book={
      publisher={Kluwer Acad. Publ., Dordrecht},
   },
   isbn={0-7923-5094-4},
   date={1998},
   pages={213--236},
   review={\MR{1643398}},
}
\bib{Jan96}{book}{
   author={Jantzen, Jens Carsten},
   title={Lectures on quantum groups},
   series={Graduate Studies in Mathematics},
   volume={6},
   publisher={American Mathematical Society, Providence, RI},
   date={1996},
   pages={viii+266},
   isbn={0-8218-0478-2},
   review={\MR{1359532}},
   doi={10.1090/gsm/006},
}
\bib{Schn95}{other}{
  author={Schneider, Hans-Jürgen},
  title={Lectures on Hopf algebras, notes by Sonia Natale,
  \url{http://documents.famaf.unc.edu.ar/publicaciones/documents/serie b/BMat31.pdf}},
  date={1995},
}
%Schweigert Lecture notes http://www.math.uni-hamburg.de/home/schweigert/skripten/hskript.pdf
%
%Reconstruction
\bib{Schau91}{book}{
   author={Schauenburg, Peter},
   title={Tannaka duality for arbitrary Hopf algebras},
   series={Algebra Berichte [Algebra Reports]},
   volume={66},
   publisher={Verlag Reinhard Fischer, Munich},
   date={1992},
   pages={ii+57},
   isbn={3-88927-100-6},
   review={\MR{1623637}},
}
\bib{Rad85}{article}{
   author={Radford, David E.},
   title={The structure of Hopf algebras with a projection},
   journal={J. Algebra},
   volume={92},
   date={1985},
   number={2},
   pages={322--347},
   issn={0021-8693},
   review={\MR{0778452}},
   doi={10.1016/0021-8693(85)90124-3},
}
\bib{LM25}{article}{
   author={Lentner, Simon},
   author={Mombelli, Mart\'in},
   title={Fiber functors and reconstruction of Hopf algebras},
   journal={Canad. J. Math.},
   volume={77},
   date={2025},
   number={5},
   pages={1718--1761},
   issn={0008-414X},
   review={\MR{4948339}},
   doi={10.4153/S0008414X24000531},
}
%%%% Laugwitz
\bib{Lau20}{article}{
   author={Laugwitz, Robert},
   title={The relative monoidal center and tensor products of monoidal
   categories},
   journal={Commun. Contemp. Math.},
   volume={22},
   date={2020},
   number={8},
   pages={1950068, 53},
   issn={0219-1997},
   review={\MR{4142329}},
   doi={10.1142/S0219199719500688},
}
\bib{LW21}{article}{
   author={Laugwitz, Robert},
   author={Walton, Chelsea},
   title={Braided commutative algebras over quantized enveloping algebras},
   journal={Transform. Groups},
   volume={26},
   date={2021},
   number={3},
   pages={957--993},
   issn={1083-4362},
   review={\MR{4309556}},
   doi={10.1007/s00031-020-09599-9},
}
%%% Degenerate cases
\bib{Len16}{article}{
   author={Lentner, Simon},
   title={A Frobenius homomorphism for Lusztig's quantum groups for
   arbitrary roots of unity},
   journal={Commun. Contemp. Math.},
   volume={18},
   date={2016},
   number={3},
   pages={1550040, 42},
   issn={0219-1997},
   review={\MR{3477402}},
   doi={10.1142/S0219199715500406},
}

%%% Andruskiewitsch Schneider Program
\bib{Lusz93}{book}{
   author={Lusztig, George},
   title={Introduction to quantum groups},
   series={Modern Birkh\"auser Classics},
   note={Reprint of the 1994 edition},
   publisher={Birkh\"auser/Springer, New York},
   date={2010},
   pages={xiv+346},
   isbn={978-0-8176-4716-2},
   review={\MR{2759715}},
   doi={10.1007/978-0-8176-4717-9},
}
\bib{AS10}{article}{
   author={Andruskiewitsch, Nicol\'as},
   author={Schneider, Hans-J\"urgen},
   title={On the classification of finite-dimensional pointed Hopf algebras},
   journal={Ann. of Math. (2)},
   volume={171},
   date={2010},
   number={1},
   pages={375--417},
   issn={0003-486X},
   review={\MR{2630042}},
   doi={10.4007/annals.2010.171.375},
}
\bib{AG19}{article}{
   author={Angiono, Ivan},
   author={Garcia Iglesias, Agustin},
   title={Pointed Hopf algebras: a guided tour to the liftings},
   journal={Rev. Colombiana Mat.},
   volume={53},
   date={2019},
   pages={1--44},
   issn={0034-7426},
   review={\MR{4053365}},
}
\bib{AAG17}{article}{ %Lift of A3
   author={Andruskiewitsch, Nicolas},
   author={Angiono, Ivan},
   author={Garcia Iglesias, Agustin},
   title={Liftings of Nichols algebras of diagonal type I. Cartan type $A$},
   journal={Int. Math. Res. Not. IMRN},
   date={2017},
   number={9},
   pages={2793--2884},
   issn={1073-7928},
   review={\MR{3658215}},
   doi={10.1093/imrn/rnw103},
}
%%%%%%%%%%%%%%%%AS program finish
\bib{An13}{article}{ %Relations
   author={Angiono, Iv\'an},
   title={On Nichols algebras of diagonal type},
   journal={J. Reine Angew. Math.},
   volume={683},
   date={2013},
   pages={189--251},
   issn={0075-4102},
   review={\MR{3181554}},
   doi={10.1515/crelle-2011-0008},
}
\bib{AKM15}{article}{ %
   author={Angiono, Iv\'an},
   author={Kochetov, Mikhail},
   author={Mastnak, Mitja},
   title={On rigidity of Nichols algebras},
   journal={J. Pure Appl. Algebra},
   volume={219},
   date={2015},
   number={12},
   pages={5539--5559},
   issn={0022-4049},
   review={\MR{3390038}},
   doi={10.1016/j.jpaa.2015.05.032},
}
\bib{AG11}{article}{
   author={Angiono, Ivan},
   author={Garcia Iglesias, Agustin},
   title={Pointed Hopf algebras with standard braiding are generated in
   degree one},
   conference={
      title={Groups, algebras and applications},
   },
   book={
      series={Contemp. Math.},
      volume={537},
      publisher={Amer. Math. Soc., Providence, RI},
   },
   isbn={978-0-8218-5239-2},
   date={2011},
   pages={57--70},
   review={\MR{2799091}},
   doi={10.1090/conm/537/10566},
}
%%% Nichols
\bib{Nichols78}{article}{
   author={Nichols, Warren D.},
   title={Bialgebras of type one},
   journal={Comm. Algebra},
   volume={6},
   date={1978},
   number={15},
   pages={1521--1552},
   issn={0092-7872},
   review={\MR{0506406}},
   doi={10.1080/00927877808822306},
}
\bib{Wor89}{article}{
   author={Woronowicz, S. L.},
   title={Differential calculus on compact matrix pseudogroups (quantum
   groups)},
   journal={Comm. Math. Phys.},
   volume={122},
   date={1989},
   number={1},
   pages={125--170},
   issn={0010-3616},
   review={\MR{0994499}},
}
\bib{Ros98}{article}{
   author={Rosso, Marc},
   title={Quantum groups and quantum shuffles},
   journal={Invent. Math.},
   volume={133},
   date={1998},
   number={2},
   pages={399--416},
   issn={0020-9910},
   review={\MR{1632802}},
   doi={10.1007/s002220050249},
}
\bib{Heck09}{article}{
   author={Heckenberger, I.},
   title={Classification of arithmetic root systems},
   journal={Adv. Math.},
   volume={220},
   date={2009},
   number={1},
   pages={59--124},
   issn={0001-8708},
   review={\MR{2462836}},
   doi={10.1016/j.aim.2008.08.005},
}
\bib{AHS10}{article}{
   author={Andruskiewitsch, Nicol\'as},
   author={Heckenberger, Istv\'an},
   author={Schneider, Hans-J\"urgen},
   title={The Nichols algebra of a semisimple Yetter-Drinfeld module},
   journal={Amer. J. Math.},
   volume={132},
   date={2010},
   number={6},
   pages={1493--1547},
   issn={0002-9327},
   review={\MR{2766176}},
}
\bib{Heck06}{article}{ %First root system
   author={Heckenberger, I.},
   title={The Weyl groupoid of a Nichols algebra of diagonal type},
   journal={Invent. Math.},
   volume={164},
   date={2006},
   number={1},
   pages={175--188},
   issn={0020-9910},
   review={\MR{2207786}},
   doi={10.1007/s00222-005-0474-8},
}
\bib{Heck07}{article}{ %Trees
   author={Heckenberger, I.},
   title={Examples of finite-dimensional rank 2 Nichols algebras of diagonal
   type},
   journal={Compos. Math.},
   volume={143},
   date={2007},
   number={1},
   pages={165--190},
   issn={0010-437X},
   review={\MR{2295200}},
   doi={10.1112/S0010437X06002430},
}
\bib{HY08}{article}{ %Coxeter groupoid
   author={Heckenberger, Istv\'an},
   author={Yamane, Hiroyuki},
   title={A generalization of Coxeter groups, root systems, and Matsumoto's
   theorem},
   journal={Math. Z.},
   volume={259},
   date={2008},
   number={2},
   pages={255--276},
   issn={0025-5874},
   review={\MR{2390080}},
   doi={10.1007/s00209-007-0223-3},
}
\bib{Ser96}{article}{
   author={Serganova, Vera},
   title={On generalizations of root systems},
   journal={Comm. Algebra},
   volume={24},
   date={1996},
   number={13},
   pages={4281--4299},
   issn={0092-7872},
   review={\MR{1414584}},
   doi={10.1080/00927879608825814},
}\bib{Cu11}{article}{
   author={Cuntz, M.},
   title={Crystallographic arrangements: Weyl groupoids and simplicial
   arrangements},
   journal={Bull. Lond. Math. Soc.},
   volume={43},
   date={2011},
   number={4},
   pages={734--744},
   issn={0024-6093},
   review={\MR{2820159}},
   doi={10.1112/blms/bdr009},
}
\bib{CH09}{article}{ %Rank 2 with continued fraction
   author={Cuntz, Michael},
   author={Heckenberger, Istv\'an},
   title={Weyl groupoids of rank two and continued fractions},
   journal={Algebra Number Theory},
   volume={3},
   date={2009},
   number={3},
   pages={317--340},
   issn={1937-0652},
   review={\MR{2525553}},
   doi={10.2140/ant.2009.3.317},
}
\bib{CH11}{article}{ %Rank 2 classification with triangulation
   author={Cuntz, M.},
   author={Heckenberger, I.},
   title={Reflection groupoids of rank two and cluster algebras of type $A$},
   journal={J. Combin. Theory Ser. A},
   volume={118},
   date={2011},
   number={4},
   pages={1350--1363},
   issn={0097-3165},
   review={\MR{2755086}},
   doi={10.1016/j.jcta.2010.12.003},
}
\bib{CH15}{article}{
   author={Cuntz, Michael},
   author={Heckenberger, Istv\'an},
   title={Finite Weyl groupoids},
   journal={J. Reine Angew. Math.},
   volume={702},
   date={2015},
   pages={77--108},
   issn={0075-4102},
   review={\MR{3341467}},
   doi={10.1515/crelle-2013-0033},
}
\bib{Cu19}{article}{ %affine
   author={Cuntz, M.},
   author={M\"uhlherr, B.},
   author={Weigel, C. J.},
   title={On the Tits cone of a Weyl groupoid},
   journal={Comm. Algebra},
   volume={47},
   date={2019},
   number={12},
   pages={5261--5285},
   issn={0092-7872},
   review={\MR{4019339}},
   doi={10.1080/00927872.2019.1617873},
}
\bib{HS13}{article}{ %Coideal, Reflection
   author={Heckenberger, Istv\'an},
   author={Schneider, Hans-J\"urgen},
   title={Right coideal subalgebras of Nichols algebras and the Duflo order
   on the Weyl groupoid},
   journal={Israel J. Math.},
   volume={197},
   date={2013},
   number={1},
   pages={139--187},
   issn={0021-2172},
   review={\MR{3096611}},
   doi={10.1007/s11856-012-0180-3},
}
\bib{BLS15}{article}{
   author={Barvels, Alexander},
   author={Lentner, Simon},
   author={Schweigert, Christoph},
   title={Partially dualized Hopf algebras have equivalent Yetter-Drinfel'd
   modules},
   journal={J. Algebra},
   volume={430},
   date={2015},
   pages={303--342},
   issn={0021-8693},
   review={\MR{3323984}},
   doi={10.1016/j.jalgebra.2015.02.010},
}
\bib{Wo23}{article}{
   author={Wolf, Kevin},
   title={Nichols systems and their reflections},
   journal={Comm. Algebra},
   volume={51},
   date={2023},
   number={2},
   pages={822--840},
   issn={0092-7872},
   review={\MR{4532828}},
   doi={10.1080/00927872.2022.2115054},
}
\bib{CL17}{article}{
   author={Cuntz, M.},
   author={Lentner, S.},
   title={A simplicial complex of Nichols algebras},
   journal={Math. Z.},
   volume={285},
   date={2017},
   number={3-4},
   pages={647--683},
   issn={0025-5874},
   review={\MR{3623727}},
   doi={10.1007/s00209-016-1711-0},
}
\bib{DF24}{article}{ %Parabolic root system 
   author={Dimitrov, Ivan},
   author={Fioresi, Rita},
   title={Generalized root systems},
   journal={Trans. Amer. Math. Soc. Ser. B},
   volume={11},
   date={2024},
   pages={1462--1508},
   review={\MR{4840297}},
   doi={10.1090/btran/214},
}
\bib{HS20}{book}{
   author={Heckenberger, Istv\'an},
   author={Schneider, Hans-J\"urgen},
   title={Hopf algebras and root systems},
   series={Mathematical Surveys and Monographs},
   volume={247},
   publisher={American Mathematical Society, Providence, RI},
   date={2020},
   pages={xix+582},
   isbn={978-1-4704-5232-2},
   review={\MR{4164719}},
}
\bib{AA17}{article}{ %Survey
   author={Andruskiewitsch, Nicol\'as},
   author={Angiono, Iv\'an},
   title={On finite dimensional Nichols algebras of diagonal type},
   journal={Bull. Math. Sci.},
   volume={7},
   date={2017},
   number={3},
   pages={353--573},
   issn={1664-3607},
   review={\MR{3736568}},
   doi={10.1007/s13373-017-0113-x},
}
%% Jing Char p
\bib{HW15}{article}{
   author={Wang, Jing},
   author={Heckenberger, Istv\'an},
   title={Rank 2 Nichols algebras of diagonal type over fields of positive
   characteristic},
   journal={SIGMA Symmetry Integrability Geom. Methods Appl.},
   volume={11},
   date={2015},
   pages={Paper 011, 24},
   review={\MR{3313687}},
   doi={10.3842/SIGMA.2015.011},
}
\bib{Wa17}{article}{
   author={Wang, Jing},
   title={Rank three Nichols algebras of diagonal type over arbitrary
   fields},
   journal={Israel J. Math.},
   volume={218},
   date={2017},
   number={1},
   pages={1--26},
   issn={0021-2172},
   review={\MR{3625122}},
   doi={10.1007/s11856-017-1456-4},
}
\bib{LYQW24}{arXiv}{
  author={Lei, L.},
  author={Yuan, C.},
  author={Qian, C.},
  author={Wang, J.},
  title={Higher rank Nichols algebras of diagonal type with finite arithmetic root systems in positive characteristic},
  date={2024},
  eprint={arXiv:2412.20786},
  archiveprefix={arXiv},
}

%% nonabelian Nichols
\bib{MS00}{article}{
   author={Milinski, Alexander},
   author={Schneider, Hans-J\"urgen},
   title={Pointed indecomposable Hopf algebras over Coxeter groups},
   conference={
      title={New trends in Hopf algebra theory},
      address={La Falda},
      date={1999},
   },
   book={
      series={Contemp. Math.},
      volume={267},
      publisher={Amer. Math. Soc., Providence, RI},
   },
   isbn={0-8218-2126-1},
   date={2000},
   pages={215--236},
   review={\MR{1800714}},
   doi={10.1090/conm/267/04272},
}
%\bib{Len12}{other}{
%  author={Lentner, Simon D.},
%  title={Orbifoldizing Hopf- and Nichols algebras, PhD thesis LMU Munich,
%  \url{https://edoc.ub.uni-muenchen.de/15363/1/Lentner_Simon.pdf}},
%  date={2012},
%}
\bib{Len14}{article}{
   author={Lentner, Simon},
   title={New large-rank Nichols algebras over nonabelian groups with
   commutator subgroup $\mathbb{Z}_2$},
   journal={J. Algebra},
   volume={419},
   date={2014},
   pages={1--33},
   issn={0021-8693},
   review={\MR{3253277}},
   doi={10.1016/j.jalgebra.2014.07.017},
}
\bib{FK99}{article}{
   author={Fomin, Sergey},
   author={Kirillov, Anatol N.},
   title={Quadratic algebras, Dunkl elements, and Schubert calculus},
   conference={
      title={Advances in geometry},
   },
   book={
      series={Progr. Math.},
      volume={172},
      publisher={Birkh\"auser Boston, Boston, MA},
   },
   isbn={0-8176-4044-4},
   date={1999},
   pages={147--182},
   review={\MR{1667680}},
}
\bib{Ven12}{arXiv}{ 
  author={Vendramin, Leandro},
  title={Fomin-Kirillov algebras - an abstract},
  date={2012},
  eprint={arXiv:1210.5423 },
  archiveprefix={arXiv},
}
\bib{ACG15}{article}{ 
%As example for rank 1 rack methods
    author={Andruskiewitsch, Nicolas},
   author={Carnovale, Giovanna},
   author={Garcia, Gaston Andres},
   title={Finite-dimensional pointed Hopf algebras over finite simple groups
   of Lie type I. Non-semisimple classes in ${\bf PSL}_n(q)$ },
   journal={J. Algebra},
   volume={442},
   date={2015},
   pages={36--65},
   issn={0021-8693},
   review={\MR{3395052}},
   doi={10.1016/j.jalgebra.2014.06.019}
}
\bib{HS10}{article}{
   author={Heckenberger, I.},
   author={Schneider, H.-J.},
   title={Nichols algebras over groups with finite root system of rank two
   I},
   journal={J. Algebra},
   volume={324},
   date={2010},
   number={11},
   pages={3090--3114},
   issn={0021-8693},
   review={\MR{2732989}},
   doi={10.1016/j.jalgebra.2010.06.021},
}
\bib{HV17}{article}{
   author={Heckenberger, I.},
   author={Vendramin, L.},
   title={A classification of Nichols algebras of semisimple Yetter-Drinfeld
   modules over non-abelian groups},
   journal={J. Eur. Math. Soc. (JEMS)},
   volume={19},
   date={2017},
   number={2},
   pages={299--356},
   issn={1435-9855},
   review={\MR{3605018}},
   doi={10.4171/JEMS/667},
}
\bib{HMV24}{article}{
   author={Heckenberger, I.},
   author={Meir, E.},
   author={Vendramin, L.},
   title={Finite-dimensional Nichols algebras of simple Yetter-Drinfeld
   modules (over groups) of prime dimension},
   journal={Adv. Math.},
   volume={444},
   date={2024},
   pages={Paper No. 109637, 30},
   issn={0001-8708},
   review={\MR{4729697}},
   doi={10.1016/j.aim.2024.109637},
}
%%%%%%% YD
\bib{Besp95}{article}{
   author={Bespalov, Yu.\ N.},
   title={Crossed modules, quantum braided groups and ribbon structures},
   language={English, with English and Russian summaries},
   journal={Teoret. Mat. Fiz.},
   volume={103},
   date={1995},
   number={3},
   pages={368--387},
   issn={0564-6162},
   translation={
      journal={Theoret. and Math. Phys.},
      volume={103},
      date={1995},
      number={3},
      pages={621--637},
      issn={0040-5779},
   },
   review={\MR{1472306}},
   doi={10.1007/BF02065863},
}
\bib{AF00}{article}{ %Categorical Radford 
   author={Alonso Alvarez, J. N.},
   author={Fern\'andez Vilaboa, J. M.},
   title={Cleft extensions in braided categories},
   journal={Comm. Algebra},
   volume={28},
   date={2000},
   number={7},
   pages={3185--3196},
   issn={0092-7872},
   review={\MR{1765310}},
   doi={10.1080/00927870008827018},
}
\bib{BV12}{article}{
   author={Brugui\`eres, Alain},
   author={Virelizier, Alexis},
   title={Quantum double of Hopf monads and categorical centers},
   journal={Trans. Amer. Math. Soc.},
   volume={364},
   date={2012},
   number={3},
   pages={1225--1279},
   issn={0002-9947},
   review={\MR{2869176}},
   doi={10.1090/S0002-9947-2011-05342-0},
}
%%%% Quantum groups
%Majid again
\bib{LW22}{article}{
   author={Laugwitz, Robert},
   author={Walton, Chelsea},
   title={Constructing non-semisimple modular categories with relative
   monoidal centers},
   journal={Int. Math. Res. Not. IMRN},
   date={2022},
   number={20},
   pages={15826--15868},
   issn={1073-7928},
   review={\MR{4498166}},
   doi={10.1093/imrn/rnab097},
}
%Unrolled
\bib{CGP15}{article}{
   author={Costantino, Francesco},
   author={Geer, Nathan},
   author={Patureau-Mirand, Bertrand},
   title={Some remarks on the unrolled quantum group of $\germ{sl}(2)$},
   journal={J. Pure Appl. Algebra},
   volume={219},
   date={2015},
   number={8},
   pages={3238--3262},
   issn={0022-4049},
   review={\MR{3320217}},
   doi={10.1016/j.jpaa.2014.10.012},
}

%%%Quasi Quantum 
\bib{KS11}{article}{
   author={Kondo, Hiroki},
   author={Saito, Yoshihisa},
   title={Indecomposable decomposition of tensor products of modules over
   the restricted quantum universal enveloping algebra associated to
   ${\germ{sl}}_2$},
   journal={J. Algebra},
   volume={330},
   date={2011},
   pages={103--129},
   issn={0021-8693},
   review={\MR{2774620}},
   doi={10.1016/j.jalgebra.2011.01.010},
}
\bib{CGR20}{article}{
   author={Creutzig, Thomas},
   author={Gainutdinov, Azat M.},
   author={Runkel, Ingo},
   title={A quasi-Hopf algebra for the triplet vertex operator algebra},
   journal={Commun. Contemp. Math.},
   volume={22},
   date={2020},
   number={3},
   pages={1950024, 71},
   issn={0219-1997},
   review={\MR{4082225}},
   doi={10.1142/S021919971950024X},
}
\bib{AG03}{article}{
   author={Arkhipov, Sergey},
   author={Gaitsgory, Dennis},
   title={Another realization of the category of modules over the small
   quantum group},
   journal={Adv. Math.},
   volume={173},
   date={2003},
   number={1},
   pages={114--143},
   issn={0001-8708},
   review={\MR{1954457}},
   doi={10.1016/S0001-8708(02)00016-6},
}

%VOA
%%%%%%%%%%%%%%%%%%%%%%%%%%%%%%%%%%%%%%%%%%%%%%%%%%%
% Textbook
\bib{Kac97}{book}{
   author={Kac, Victor},
   title={Vertex algebras for beginners},
   series={University Lecture Series},
   volume={10},
   publisher={American Mathematical Society, Providence, RI},
   date={1997},
   pages={viii+141},
   isbn={0-8218-0643-2},
   review={\MR{1417941}},
   doi={10.1090/ulect/010},
}
\bib{FBZ04}{book}{
   author={Frenkel, Edward},
   author={Ben-Zvi, David},
   title={Vertex algebras and algebraic curves},
   series={Mathematical Surveys and Monographs},
   volume={88},
   edition={2},
   publisher={American Mathematical Society, Providence, RI},
   date={2004},
   pages={xiv+400},
   isbn={0-8218-3674-9},
   review={\MR{2082709}},
   doi={10.1090/surv/088},
}
\bib{Schot08}{book}{
   author={Schottenloher, M.},
   title={A mathematical introduction to conformal field theory},
   series={Lecture Notes in Physics},
   volume={759},
   edition={2},
   publisher={Springer-Verlag, Berlin},
   date={2008},
   pages={xvi+249},
   isbn={978-3-540-68625-5},
   review={\MR{2492295}},
}
\bib{MS88}{article}{
   author={Moore, Gregory},
   author={Seiberg, Nathan},
   title={Polynomial equations for rational conformal field theories},
   journal={Phys. Lett. B},
   volume={212},
   date={1988},
   number={4},
   pages={451--460},
   issn={0370-2693},
   review={\MR{0962600}},
   doi={10.1016/0370-2693(88)91796-0},
}
\bib{Verl88}{article}{
   author={Verlinde, Erik},
   title={Fusion rules and modular transformations in $2$D conformal field
   theory},
   journal={Nuclear Phys. B},
   volume={300},
   date={1988},
   number={3},
   pages={360--376},
   issn={0550-3213},
   review={\MR{0954762}},
   doi={10.1016/0550-3213(88)90603-7},
}
\bib{Gab94}{article}{
   author={Gaberdiel, Matthias},
   title={Fusion in conformal field theory as the tensor product of the
   symmetry algebra},
   journal={Internat. J. Modern Phys. A},
   volume={9},
   date={1994},
   number={26},
   pages={4619--4636},
   issn={0217-751X},
   review={\MR{1295760}},
   doi={10.1142/S0217751X94001849},
}
\bib{KL93}{article}{
   author={Kazhdan, D.},
   author={Lusztig, G.},
   title={Tensor structures arising from affine Lie algebras. I, II},
   journal={J. Amer. Math. Soc.},
   volume={6},
   date={1993},
   number={4},
   pages={905--947, 949--1011},
   issn={0894-0347},
   review={\MR{1186962}},
   doi={10.2307/2152745},
}
\bib{HL94}{article}{
   author={Huang, Yi-Zhi},
   author={Lepowsky, James},
   title={Tensor products of modules for a vertex operator algebra and
   vertex tensor categories},
   conference={
      title={Lie theory and geometry},
   },
   book={
      series={Progr. Math.},
      volume={123},
      publisher={Birkh\"auser Boston, Boston, MA},
   },
   isbn={0-8176-3761-3},
   date={1994},
   pages={349--383},
   review={\MR{1327541}},
   doi={10.1007/978-1-4612-0261-5\_13},
}
\bib{HLZ06}{article}{
   author={Huang, Yi-Zhi},
   author={Lepowsky, James},
   author={Zhang, Lin},
   title={A logarithmic generalization of tensor product theory for modules
   for a vertex operator algebra},
   journal={Internat. J. Math.},
   volume={17},
   date={2006},
   number={8},
   pages={975--1012},
   issn={0129-167X},
   review={\MR{2261644}},
   doi={10.1142/S0129167X06003758},
}
%LatticeVOA
\bib{FLM88}{book}{
   author={Frenkel, Igor},
   author={Lepowsky, James},
   author={Meurman, Arne},
   title={Vertex operator algebras and the Monster},
   series={Pure and Applied Mathematics},
   volume={134},
   publisher={Academic Press, Inc., Boston, MA},
   date={1988},
   pages={liv+508},
   isbn={0-12-267065-5},
   review={\MR{0996026}},
}
\bib{DLM96}{article}{
   author={Dong, Chongying},
   author={Li, Haisheng},
   author={Mason, Geoffrey},
   title={Simple currents and extensions of vertex operator algebras},
   journal={Comm. Math. Phys.},
   volume={180},
   date={1996},
   number={3},
   pages={671--707},
   issn={0010-3616},
   review={\MR{1408523}},
}

%%%%%%%%%%%%%%%%%%%%%%%%%%%%%%%%%%%Screening
% Varchenko
\bib{DF84}{article}{
   author={Dotsenko, Vl.\ S.},
   author={Fateev, V. A.},
   title={Conformal algebra and multipoint correlation functions in $2$D
   statistical models},
   journal={Nuclear Phys. B},
   volume={240},
   date={1984},
   number={3},
   pages={312--348},
   issn={0550-3213},
   review={\MR{0762194}},
   doi={10.1016/0550-3213(84)90269-4},
}
\bib{Ros97}{article}{
   author={Rosso, Marc},
   title={Integrals of vertex operators and quantum shuffles},
   journal={Lett. Math. Phys.},
   volume={41},
   date={1997},
   number={2},
   pages={161--168},
   issn={0377-9017},
   review={\MR{1463867}},
   doi={10.1023/A:1007352917712},
}
\bib{Sel44}{article}{
   author={Selberg, Atle},
   title={Remarks on a multiple integral},
   language={Norwegian},
   journal={Norsk Mat. Tidsskr.},
   volume={26},
   date={1944},
   pages={71--78},
   issn={2387-2187},
   review={\MR{0018287}},
}
\bib{Var98}{book}{
   author={Varchenko, Alexander},
   title={Special functions, KZ type equations, and representation theory},
   series={CBMS Regional Conference Series in Mathematics},
   volume={98},
   publisher={Conference Board of the Mathematical Sciences, Washington, DC;
   by the American Mathematical Society, Providence, RI},
   date={2003},
   pages={viii+118},
   isbn={0-8218-2867-3},
   review={\MR{2016311}},
   doi={10.1090/cbms/098},
}
%DF48
\bib{FW08}{article}{
   author={Forrester, Peter J.},
   author={Warnaar, S. Ole},
   title={The importance of the Selberg integral},
   journal={Bull. Amer. Math. Soc. (N.S.)},
   volume={45},
   date={2008},
   number={4},
   pages={489--534},
   issn={0273-0979},
   review={\MR{2434345}},
   doi={10.1090/S0273-0979-08-01221-4},
}
\bib{TV03}{article}{
   author={Tarasov, V.},
   author={Varchenko, A.},
   title={Selberg-type integrals associated with $\germ {sl}_3$},
   journal={Lett. Math. Phys.},
   volume={65},
   date={2003},
   number={3},
   pages={173--185},
   issn={0377-9017},
   review={\MR{2033704}},
   doi={10.1023/B:MATH.0000010712.67685.9d},
}
\bib{War09}{article}{
   author={Warnaar, S. Ole},
   title={A Selberg integral for the Lie algebra $A_n$},
   journal={Acta Math.},
   volume={203},
   date={2009},
   number={2},
   pages={269--304},
   issn={0001-5962},
   review={\MR{2570072}},
   doi={10.1007/s11511-009-0043-x},
}
\bib{Suss24}{article}{
   author={Sussman, Ethan},
   title={The singularities of Selberg- and Dotsenko-Fateev-like integrals},
   journal={Ann. Henri Poincar\'e},
   volume={25},
   date={2024},
   number={9},
   pages={3957--4032},
   issn={1424-0637},
   review={\MR{4789304}},
   doi={10.1007/s00023-023-01402-1},
}

%Hopf Tensor Category
\bib{CMY23}{article}{ %Again singlet
   author={Creutzig, Thomas},
   author={McRae, Robert},
   author={Yang, Jinwei},
   title={Ribbon tensor structure on the full representation categories of
   the singlet vertex algebras},
   journal={Adv. Math.},
   volume={413},
   date={2023},
   pages={Paper No. 108828, 79},
   issn={0001-8708},
   review={\MR{4526492}},
   doi={10.1016/j.aim.2022.108828},
}
\bib{NORW24}{arXiv}{
  author={Nakano, Hiromu},
  author={Orosz Hunziker, Florencia},
  author={Ros Camacho, Ana},
  author={Wood, Simon},
  title={Fusion rules and rigidity for weight modules over the simple admissible affine $sl(2)$ and $N=2$ superconformal vertex operator superalgebras},
  date={2024},
  label={NORW24},
  eprint={2411.11387},
  archiveprefix={arXiv},
}

%Pre Nichols Algebra
\bib{An16}{article}{
   author={Angiono, Iv\'an},
   title={Distinguished pre-Nichols algebras},
   journal={Transform. Groups},
   volume={21},
   date={2016},
   number={1},
   pages={1--33},
   issn={1083-4362},
   review={\MR{3459702}},
   doi={10.1007/s00031-015-9341-x},
}

%%%%%%%%%%%%%%%%%%%%%%%%%%%%%%% sl_2 Proof 
%\cites{FGST05, AM08, TW13, CGR20, CLR21, GN21, CLR23}
\bib{FGST05}{article}{
   author={Ga\u inutdinov, A. M.},
   author={Semikhatov, A. M.},
   author={Tipunin, I. Yu.},
   author={Fe\u igin, B. L.},
   title={The Kazhdan-Lusztig correspondence for the representation category
   of the triplet $W$-algebra in logorithmic conformal field theories},
   language={Russian, with Russian summary},
   journal={Teoret. Mat. Fiz.},
   volume={148},
   date={2006},
   number={3},
   pages={398--427},
   issn={0564-6162},
   translation={
      journal={Theoret. and Math. Phys.},
      volume={148},
      date={2006},
      number={3},
      pages={1210--1235},
      issn={0040-5779},
   },
   review={\MR{2283660}},
   doi={10.1007/s11232-006-0113-6},
}
\bib{AM08}{article}{
   author={Adamovi\'c, Dra\v zen},
   author={Milas, Antun},
   title={On the triplet vertex algebra $\scr W(p)$},
   journal={Adv. Math.},
   volume={217},
   date={2008},
   number={6},
   pages={2664--2699},
   issn={0001-8708},
   review={\MR{2397463}},
   doi={10.1016/j.aim.2007.11.012},
}
\bib{TW13}{article}{
   author={Tsuchiya, Akihiro},
   author={Wood, Simon},
   title={The tensor structure on the representation category of the $\scr
   W_p$ triplet algebra},
   journal={J. Phys. A},
   volume={46},
   date={2013},
   number={44},
   pages={445203, 40},
   issn={1751-8113},
   review={\MR{3120909}},
   doi={10.1088/1751-8113/46/44/445203},
}
%CGR20
\bib{CLR21}{arXiv}{
  author={Creutzig, Thomas},
  author={Lentner, Simon},
  author={Rupert, Matthew},
  title={Characterizing braided tensor categories associated to logarithmic vertex operator algebras},
  date={2021},
  eprint={2104.13262},
  archiveprefix={arXiv}
}
\bib{GN21}{arXiv}{
author={Gannon, Terry},
author={Negron, Cris},
title={Quantum $\mathrm{SL}(2)$ and logarithmic vertex operator algebras at $(p,1)$-central charge}
date={2023},
eprint={2104.12821},
archiveprefix={arXiv}
}
%CLR23
%General approach 
\bib{CLR23}{arXiv}{
  author={Creutzig, Thomas},
  author={Lentner, Simon},
  author={Rupert, Matthew},
  title={An algebraic theory for logarithmic Kazhdan-Lusztig correspondences},
  date={2023},
  eprint={2306.11492},
  archiveprefix={arXiv}
}
\bib{FT10}{arXiv}{
  author={Feigin, Boris},
  author={Tipunin, Ilja},
  title={Logarithmic CFTs connected with simple Lie algebras},
  date={2010},
  eprint={1002.5047},
  archiveprefix={arXiv},
}

\bib{AM14}{article}{
   author={Adamovi\'c, Dra{\v z}en},
   author={Milas, Antun},
   title={$C_2$-cofinite $\mathcal{W}$-algebras and their logarithmic
   representations},
   conference={
      title={Conformal field theories and tensor categories},
   },
   book={
      series={Math. Lect. Peking Univ.},
      publisher={Springer, Heidelberg},
   },
   isbn={978-3-642-39382-2},
   isbn={978-3-642-39383-9},
   date={2014},
   pages={249--270},
   review={\MR{3585369}},
}

\bib{ST12}{article}{
   author={Semikhatov, A. M.},
   author={Tipunin, I. Yu.},
   title={The Nichols algebra of screenings},
   journal={Commun. Contemp. Math.},
   volume={14},
   date={2012},
   number={4},
   pages={1250029, 66},
   issn={0219-1997},
   review={\MR{2965674}},
   doi={10.1142/S0219199712500290},
}

\bib{Len21}{article}{
   author={Lentner, Simon D.},
   title={Quantum groups and Nichols algebras acting on conformal field
   theories},
   journal={Adv. Math.},
   volume={378},
   date={2021},
   pages={Paper No. 107517, 71},
   issn={0001-8708},
   review={\MR{4184294}},
   doi={10.1016/j.aim.2020.107517},
}

\bib{Sug21}{article}{
   author={Sugimoto, Shoma},
   title={On the Feigin-Tipunin conjecture},
   journal={Selecta Math. (N.S.)},
   volume={27},
   date={2021},
   number={5},
   pages={Paper No. 86, 43},
   issn={1022-1824},
   review={\MR{4305499}},
   doi={10.1007/s00029-021-00662-1},
}

\bib{Sug23}{article}{
   author={Sugimoto, Shoma},
   title={Simplicity of higher rank triplet $W$-algebras},
   journal={Int. Math. Res. Not. IMRN},
   date={2023},
   number={8},
   pages={7169--7199},
   issn={1073-7928},
   review={\MR{4574397}},
   doi={10.1093/imrn/rnac189},
}

%%%%%%%%5%%%%%%%%%
%ORBIFOLDS
\bib{ALM13}{article}{
   author={Adamovi\'c, Dra\v zen},
   author={Lin, Xianzu},
   author={Milas, Antun},
   title={ADE subalgebras of the triplet vertex algebra $\scr{W}(p)$:
   $D$-series},
   journal={Internat. J. Math.},
   volume={25},
   date={2014},
   number={1},
   pages={1450001, 34},
   issn={0129-167X},
   review={\MR{3189758}},
   doi={10.1142/S0129167X14500013},
}
%%% DEFORMATION
\bib{Bak16}{article}{ %Logarithmic Jacobi identity
   author={Bakalov, Bojko},
   title={Twisted logarithmic modules of vertex algebras},
   journal={Comm. Math. Phys.},
   volume={345},
   date={2016},
   number={1},
   pages={355--383},
   issn={0010-3616},
   review={\MR{3509017}},
   doi={10.1007/s00220-015-2503-9},
}
%Deformation
%DLM
\bib{Li97}{article}{
   author={Li, Haisheng},
   title={The physics superselection principle in vertex operator algebra
   theory},
   journal={J. Algebra},
   volume={196},
   date={1997},
   number={2},
   pages={436--457},
   issn={0021-8693},
   review={\MR{1475118}},
   doi={10.1006/jabr.1997.7126},
}
\bib{FFHST02}{article}{
   author={Fjelstad, J.},
   author={Fuchs, J.},
   author={Hwang, S.},
   author={Semikhatov, A. M.},
   author={Tipunin, I. Yu.},
   title={Logarithmic conformal field theories via logarithmic deformations},
   journal={Nuclear Phys. B},
   volume={633},
   date={2002},
   number={3},
   pages={379--413},
   issn={0550-3213},
   review={\MR{1910269}},
   doi={10.1016/S0550-3213(02)00220-1},
}
\bib{AM09}{article}{
   author={Adamovi\'c, Dra\v zen},
   author={Milas, Antun},
   title={Lattice construction of logarithmic modules for certain vertex
   algebras},
   journal={Selecta Math. (N.S.)},
   volume={15},
   date={2009},
   number={4},
   pages={535--561},
   issn={1022-1824},
   review={\MR{2565050}},
   doi={10.1007/s00029-009-0009-z},
}
% Mixed Quantum group
\bib{Gait21}{article}{
   author={Gaitsgory, Dennis},
   title={A conjectural extension of the Kazhdan-Lusztig equivalence},
   journal={Publ. Res. Inst. Math. Sci.},
   volume={57},
   date={2021},
   number={3-4},
   pages={1227--1376},
   issn={0034-5318},
   review={\MR{4322010}},
   doi={10.4171/prims/57-3-14},
}

%Irregular
\bib{FL24b}{arXiv}{
  author={Feigin, Boris},
  author={Lentner, Simon},
  title={Twisted vertex algebra modules for irregular connections: A case study},
  date={2024},
  eprint={2411.16272},
  archiveprefix={arXiv},
}



%Commutatice and VOA
%%%%%%%%%%%%%%%%%%%%%%%%%%%%%%%%%%%%%%%%%%%%%%%%%%%%

%Commutative Algebra
\bib{Par95}{article}{
   author={Pareigis, Bodo},
   title={On braiding and dyslexia},
   journal={J. Algebra},
   volume={171},
   date={1995},
   number={2},
   pages={413--425},
   issn={0021-8693},
   review={\MR{1315904}},
   doi={10.1006/jabr.1995.1019},
}
\bib{KO02}{article}{
   author={Kirillov, Alexander, Jr.},
   author={Ostrik, Viktor},
   title={On a $q$-analogue of the McKay correspondence and the ADE
   classification of $\germ {sl}_2$ conformal field theories},
   journal={Adv. Math.},
   volume={171},
   date={2002},
   number={2},
   pages={183--227},
   issn={0001-8708},
   review={\MR{1936496}},
   doi={10.1006/aima.2002.2072},
}
\bib{FFRS06}{article}{
   author={Fr\"ohlich, J\"urg},
   author={Fuchs, J\"urgen},
   author={Runkel, Ingo},
   author={Schweigert, Christoph},
   title={Correspondences of ribbon categories},
   journal={Adv. Math.},
   volume={199},
   date={2006},
   number={1},
   pages={192--329},
   issn={0001-8708},
   review={\MR{2187404}},
   doi={10.1016/j.aim.2005.04.007},
}\bib{SY24}{arXiv}{
  author={Shimizu, Kenichi},
  author={Yadav, Harshit},
  title={Commutative exact algebras and modular monoidal categories},
  date={2024},
  eprint={arXiv:2408.06314},
  archiveprefix={arXiv},
}
%Commutative Algebra and VOA
\bib{HKL15}{article}{
   author={Huang, Yi-Zhi},
   author={Kirillov, Alexander, Jr.},
   author={Lepowsky, James},
   title={Braided tensor categories and extensions of vertex operator
   algebras},
   journal={Comm. Math. Phys.},
   volume={337},
   date={2015},
   number={3},
   pages={1143--1159},
   issn={0010-3616},
   review={\MR{3339173}},
   doi={10.1007/s00220-015-2292-1},
}
\bib{CKM24}{article}{
   author={Creutzig, Thomas},
   author={Kanade, Shashank},
   author={McRae, Robert},
   title={Tensor categories for vertex operator superalgebra extensions},
   journal={Mem. Amer. Math. Soc.},
   volume={295},
   date={2024},
   number={1472},
   pages={vi+181},
   issn={0065-9266},
   isbn={978-1-4704-6724-1; 978-1-4704-7772-1},
   review={\MR{4720880}},
   doi={10.1090/memo/1472},
}
%GV Duality
\bib{BD13}{article}{
   author={Boyarchenko, Mitya},
   author={Drinfeld, Vladimir},
   title={A duality formalism in the spirit of Grothendieck and Verdier},
   journal={Quantum Topol.},
   volume={4},
   date={2013},
   number={4},
   pages={447--489},
   issn={1663-487X},
   review={\MR{3134025}},
   doi={10.4171/QT/45},
}
%GCrossed
\bib{ENOM10}{article}{
   author={Etingof, Pavel},
   author={Nikshych, Dmitri},
   author={Ostrik, Victor},
   title={Fusion categories and homotopy theory},
   note={With an appendix by Ehud Meir},
   journal={Quantum Topol.},
   volume={1},
   date={2010},
   number={3},
   pages={209--273},
   issn={1663-487X},
   review={\MR{2677836}},
   doi={10.4171/QT/6},
}
\bib{DN21}{article}{
   author={Davydov, Alexei},
   author={Nikshych, Dmitri},
   title={Braided Picard groups and graded extensions of braided tensor
   categories},
   journal={Selecta Math. (N.S.)},
   volume={27},
   date={2021},
   number={4},
   pages={Paper No. 65, 87},
   issn={1022-1824},
   review={\MR{4281262}},
   doi={10.1007/s00029-021-00670-1},
}
%%%%%% Freeness
\bib{Tak79}{article}{
   author={Takeuchi, Mitsuhiro},
   title={Relative Hopf modules---equivalences and freeness criteria},
   journal={J. Algebra},
   volume={60},
   date={1979},
   number={2},
   pages={452--471},
   issn={0021-8693},
   review={\MR{0549940}},
   doi={10.1016/0021-8693(79)90093-0},
}
\bib{Skry07}{article}{
   author={Skryabin, Serge},
   title={Projectivity and freeness over comodule algebras},
   journal={Trans. Amer. Math. Soc.},
   volume={359},
   date={2007},
   number={6},
   pages={2597--2623},
   issn={0002-9947},
   review={\MR{2286047}},
   doi={10.1090/S0002-9947-07-03979-7},
}
\bib{CMSY24}{arXiv}{
  author={Creutzig, Thomas},
  author={McRae, Robert},
  author={Shimizu, Kenichi},
  author={Yadav, Harshit},
  title={Commutative algebras in Grothendieck-Verdier categories, rigidity, and vertex operator algebras},
  date={2024},
  eprint={2409.14618},
  archiveprefix={arXiv},
}



%%%%%% Proof
%\bib{CLR23}{arXiv}{
%  author={Creutzig, Thomas},
%  author={Lentner, Simon},
%  author={Rupert, Matthew},
%  title={An algebraic theory for logarithmic Kazhdan-Lusztig correspondences},
%  date={2023},
%  eprint={2306.11492},
%  archiveprefix={arXiv}
%}
\bib{Sch01}{article}{
   author={Schauenburg, Peter},
   title={The monoidal center construction and bimodules},
   journal={J. Pure Appl. Algebra},
   volume={158},
   date={2001},
   number={2-3},
   pages={325--346},
   issn={0022-4049},
   review={\MR{1822847}},
   doi={10.1016/S0022-4049(00)00040-2},
}

\bib{BM17}{article}{
   author={Bringmann, Kathrin},
   author={Milas, Antun},
   title={$W$-algebras, higher rank false theta functions, and quantum
   dimensions},
   journal={Selecta Math. (N.S.)},
   volume={23},
   date={2017},
   number={2},
   pages={1249--1278},
   issn={1022-1824},
   review={\MR{3624911}},
   doi={10.1007/s00029-016-0289-z},
}
\bib{Len25}{arXiv}{
  author={Lentner, Simon},
  title={A conditional algebraic proof of the logarithmic Kazhdan-Lusztig correspondence},
  date={2025},
  eprint={2501.10735},
  archiveprefix={arXiv},
}
%
% Included Papers
%
%
\bib{ALS23}{article}{
   author={Angiono, Iv\'an},
   author={Lentner, Simon},
   author={Sanmarco, Guillermo},
   title={Pointed Hopf algebras over nonabelian groups with nonsimple
   standard braidings},
   journal={Proc. Lond. Math. Soc. (3)},
   volume={127},
   date={2023},
   number={4},
   pages={1185--1245},
   issn={0024-6115},
   review={\MR{4655349}},
   doi={10.1112/plms.12559},
}
\bib{ALSW21}{arXiv}{ %VOAs have GV Theorem 2.12
  author={Allen, Robert},
  author={Lentner, Simon D.},
  author={Schweigert, Christoph},
  author={Wood, Simon},
  title={Duality structures for module categories of vertex operator algebras and the Feigin Fuchs boson},
  date={2021},
  eprint={2107.05718},
  archiveprefix={arXiv},
}
%\bib{Len21}{article}{
%   author={Lentner, Simon D.},
%   title={Quantum groups and Nichols algebras acting on conformal field
%   theories},
%   journal={Adv. Math.},
%   volume={378},
%   date={2021},
%   pages={Paper No. 107517, 71},
%   issn={0001-8708},
%   review={\MR{4184294}},
%   doi={10.1016/j.aim.2020.107517},
%}
\bib{FL22}{article}{
   author={Flandoli, Ilaria},
   author={Lentner, Simon D.},
   title={Algebras of non-local screenings and diagonal Nichols algebras},
   journal={SIGMA Symmetry Integrability Geom. Methods Appl.},
   volume={18},
   date={2022},
   pages={Paper No. 018, 81},
   review={\MR{4392114}},
   doi={10.3842/SIGMA.2022.018},
}
\bib{FL24}{article}{
    author={Feigin, Boris L.},
    author={Lentner, Simon D.},
    title={Vertex algebras with big centre and a Kazhdan-Lusztig correspondence},
    journal={Advances in Mathematics},
    volume={457},
    date={2024},
}
%%% New examples
\bib{FK93}{book}{ %Affine Lie algebra MTC
   author={Fr\"ohlich, J\"urg},
   author={Kerler, Thomas},
   title={Quantum groups, quantum categories and quantum field theory},
   series={Lecture Notes in Mathematics},
   volume={1542},
   publisher={Springer-Verlag, Berlin},
   date={1993},
   pages={viii+431},
   isbn={3-540-56623-6},
   review={\MR{1239440}},
   doi={10.1007/BFb0084244},
}
\bib{CM}{arXiv}{ %Affine Lie aglebra MTC
author={Cain, E.},
author={Morrison, S.},
title={A field guide to categories with $A_n$ fusion rules},
 date={2017},
eprint={arXiv:1710.07362},
archiveprefix={arXiv},
}
\bib{GLM25}{arXiv}{ %Tambara Yamagami and generalzaion 
  author={Galindo, Cesar},
  author={Lentner, Simon},
  author={M\"oller, Sven},
  title={Computing $G$-Crossed Extensions and Orbifolds of Vertex Operator Algebras},
  date={2024},
  eprint={arXiv:2409.16357},
  archiveprefix={arXiv},
}
%%% Other papers
\bib{LP17}{article}{
   author={Lentner, Simon},
   author={Priel, Jan},
   title={Three natural subgroups of the Brauer-Picard group of a Hopf
   algebra with applications},
   journal={Bull. Belg. Math. Soc. Simon Stevin},
   volume={24},
   date={2017},
   number={1},
   pages={73--106},
   issn={1370-1444},
   review={\MR{3625786}},
   doi={10.36045/bbms/1489888815},
}
\bib{LMSS23}{book}{
   author={Lentner, Simon},
   author={Mierach, Svea Nora},
   author={Schweigert, Christoph},
   author={Sommerh\"auser, Yorck},
   title={Hochschild cohomology, modular tensor categories, and mapping
   class groups. I},
   series={SpringerBriefs in Mathematical Physics},
   volume={44},
   publisher={Springer, Singapore},
   date={2023},
   pages={ix+68},
   isbn={978-981-19-4644-8},
   isbn={978-981-19-4645-5},
   review={\MR{4841457}},
}
%%%%% KARPANOV SCHECHTMANN
\bib{KS20}{article}{
   author={Kapranov, Mikhail},
   author={Schechtman, Vadim},
   title={Shuffle algebras and perverse sheaves},
   journal={Pure Appl. Math. Q.},
   volume={16},
   date={2020},
   number={3},
   pages={573--657},
   issn={1558-8599},
   review={\MR{4176533}},
   doi={10.4310/PAMQ.2020.v16.n3.a9},
}
\bib{CER25}{arXiv}{
  author={Carnovale, G.},
  author={Esposito, F.},
  author={Rubio y Degrassi, L.},
  title={Truncated factorized perverse sheaves on Sym(C)},
  date={2025},
  eprint={arXiv:2502.21213},
  archiveprefix={arXiv},
}
\end{biblist}
\end{bibdiv}
\end{document}